\documentclass[reqno]{amsart}
\usepackage[top=1.1in, bottom=1in, left=1in, right=1in]{geometry}
\usepackage{amsfonts}
\usepackage{amssymb}
\usepackage{amsthm}
\usepackage{amsmath}
\usepackage{mathrsfs}
\usepackage{color}
\usepackage{enumerate}
\usepackage[numbers,sort&compress]{natbib}
\usepackage{pdfsync}
\usepackage{esint}
\usepackage{graphicx}
\usepackage{float}
\usepackage{caption}
\usepackage{subfigure}
\usepackage{stmaryrd}
\usepackage[titletoc]{appendix}

\allowdisplaybreaks

\usepackage{tikz} % Use TikZ for drawing
\usetikzlibrary{calc}
\usetikzlibrary{intersections}
\usetikzlibrary{patterns}
\usetikzlibrary{positioning, shapes.geometric}
\usepackage[colorlinks,
linkcolor=black,       
anchorcolor=blue,     
citecolor=blue,        
]{hyperref}
\makeatother

\numberwithin{equation}{section}

\newtheorem{theorem}{Theorem}[section]
\newtheorem{definition}[theorem]{Definition}

\newtheorem{lemma}[theorem]{Lemma}
\newtheorem{remark}[theorem]{Remark}

\newtheorem{proposition}[theorem]{Proposition}

\numberwithin{equation}{section}

\newcommand*{\Id}{\ensuremath{\mathrm{Id}}}

\newcommand{\aint}{{\fint}}

\newcommand{\T}{{\mathbb{T}}}

\newcommand{\rr}{\mathring{R}}
\newcommand{\ru}{\mathring{R}_q^u}
\newcommand{\rb}{\mathring{R}_q^B}
\newcommand{\p}{\partial}
\renewcommand{\P}{\mathbb{P}}
\renewcommand{\div}{{\mathrm{div}\,}}
\newcommand{\curl}{{\mathrm{curl}}}
\renewcommand{\u}{{u_q}}
\newcommand{\h}{{B_q}}
\renewcommand{\d}{{\rm d}}

\newcommand{\norm}[1]{\lVert#1\rVert}

\newcommand{\la}{\lambda_q}
\newcommand{\laq}{\lambda_{q+1}}

\newcommand{\rw}{\wt r_{\perp}}
\newcommand{\rs}{r_{\perp}}
\newcommand{\rp}{r_{\parallel}}

\newcommand{\wdc}{\wt D_{(k)}^c}

\def\cq{\chi_{q+1}}
\def\ut{\underline{T}}
\def\ot{\overline{T}}
\def\us{\underline{S}}
\def\os{\overline{S}}

\newcommand{\R}{{\mathbb R}}
\newcommand{\lbb}{\overline{\lambda}}

\def\a{{\alpha}}
\def\vf{{\varphi}}
\def\lbb{\lambda}
\def\wt{\widetilde}
\def\9{{\infty}}
\def\ve{{\varepsilon}}
\def\na{{\nabla}}
\def\bbr{{\mathbb{R}}}

\def\({\left(}
\def\){\right)}
\def\tq{{\theta_q}}
\def\tq1{{\theta_{q+1}}}

\def\Tm{{T^{\rm med}_{q+1}}}
\def\Sm{{S^{\rm med}_{q+1}}}
\newcommand{\cI}{{\mathcal{I}}}

\newcommand{\cC}{{\mathcal{C}}}
\newcommand{\cH}{{\mathcal{H}}}
\newcommand{\cL}{{\mathcal{L}}}
\newcommand{\cW}{{\mathcal{W}}}

\begin{document}
	
\title[Global dissipative solutions of the NSE and MHD] {Global dissipative solutions of the 3D Naiver-Stokes and MHD equations}

\author{Alexey Cheskidov}
\address{Institute for Theoretical Sciences, Westlake University, China.}
\email[Alexey Cheskidov]{cheskidov@westlake.edu.cn}
\thanks{}

\author{Zirong Zeng}
\address{School of Mathematics and Key Laboratory of MIIT, Nanjing University of Aeronautics and Astronautics, China.}
\email[Zirong Zeng]{beckzzr@nuaa.edu.cn}
\thanks{}

\author{Deng Zhang}
\address{School of Mathematical Sciences, CMA-Shanghai, Shanghai Jiao Tong University, China.}
\email[Deng Zhang]{dzhang@sjtu.edu.cn}
\thanks{}

\keywords{Convex integration, Navier-Stokes equations; MHD equations; Dissipative solutions.}
    
\subjclass[2020]{35A02,\ 35D30,\ 76W05.}

\begin{abstract}
For any divergence free initial data in $H^\frac12$, we prove the existence of infinitely many dissipative solutions to both the 3D Navier-Stokes and MHD equations, 
whose energy profiles are continuous and decreasing on $[0,\infty)$. If the initial data is only $L^2$, our construction yields infinitely many solutions with continuous energy, but not necessarily decreasing. Our theorem does not hold in the case of zero viscosity as this would violate the weak-strong uniqueness principle due to Lions. This was achieved by designing a convex integration scheme that takes advantage of the dissipative term.
\end{abstract}

\maketitle

{\small
\tableofcontents
}

\section{Introduction and main results}

We consider the 3D Navier-Stokes equations (NSE) 
\begin{equation}\label{equa-NSE}
	\left\{\aligned
	&\p_t u +(u\cdot \nabla )u = -\nabla p + \nu \Delta u,  \\
	&\div u = 0,
	\endaligned
	\right.
\end{equation}  
where $u:[0,\9)\times\T^3\rightarrow \R^3$ is the velocity field,
$p:[0,\9)\times\T^3\rightarrow \R$ is the pressure of the fluid,
and $\nu$ is the positive viscosity coefficient.

In 1934 Leray introduces a notion of weak solutions to the 3D Navier-Stokes equations, and proved that for any divergence-free initial datum with finite energy, there exists a weak solution satisfying the energy inequality
\begin{align}\label{lh-energy}
\frac12 \|u(t)\|_{L_x^2}^2 \leq \frac12 \|u(t_0)\|_{L_x^2}^2 - \nu \int_{t_0}^t \|\nabla u(\tau)\|_{L_x^2}^2 \, d\tau,
\end{align}
for all $t \geq t_0$, $t_0 \in [0, \infty) \setminus Ex$, where the Lebesgue measure of $Ex$ is zero and $0 \notin Ex$, see \cite{leray1934}. Such weak solutions will be referred to as Leray-Hopf solutions.

Here $Ex$ is the exceptional set of initial times $t_0$ starting from which the energy inequality does not hold (for at least one $t > t_0$). It is not known whether $Ex$ is empty. It is easy to see that
\[
\limsup_{t \to t_0+} \|u(t)\|_{L_x^2}^2 > \|u(t_0)\|_{L_x^2}^2, \qquad \forall t_0 \in Ex
\]
i.e., the energy has to instantaneously raise after each such exceptional $t_0$. This undesirable from physical point of view behavior can be ``fixed'' at a single chosen $t_0 \in Ex$ by restarting the solution at that time, but it is not known whether such a gain of the energy can be removed at every exceptional time, i.e., how to construct a solution with $Ex=\emptyset$. Thus the global existence of Leray-Hopf solutions with decreasing energy remains an open question.

Starting from the seminal paper \cite{dls09}, where De Lellis and Sz\'{e}kelyhidi introduced the method of convex integration to the mathematical fluid dynamics, and the groundbreaking paper \cite{bv19b} by Buckmaster and Vicol, convex integration became another way to construct global weak solutions of the NSE. While Leray-Hopf solutions are out of reach with this technique, it allows one to construct solutions with some desirable properties displayed by turbulent fluid flows, such as the energy cascade. In addition, solutions constructed via convex integration can actually have continuous energy profiles. Indeed, for arbitrary divergence-free finite energy initial data, solutions continuous in $L_x^2$ for positive time were constructed by Buckmaster, Colombo, and Vicol \cite{bcv21}, and recently solutions continuous in $L_x^2$ for all time (including the initial time) were constructed by the authors in \cite{czz24}. In \cite{bms21}, Burczak, Modena, and Sz\'{e}kelyhidi constructed non-unique weak solutions with arbitrary smooth energy profiles, which can be in particular decreasing, but only for the initial data constructed by the scheme itself, i.e., not prescribed.

We will call a weak solution {\it dissipative} if its energy is continuous and decreasing on $[0, \infty)$. In this paper we prove the global existence and non-uniqueness of dissipative weak solutions to the 3D NSE and MHD equations for any divergence-free initial data in a critical space, such as $H_x^{\frac12}$. For any divergence-free initial data in $L_x^2$, our construction gives infinitely many weak solutions with continuous energy, but not necessarily dissipative. 

The notion of {\it dissipative solutions} of the Euler equations was introduced by Lions in \cite{L96}. A weakly continuous in $L_x^2$ weak solution $u(t)$ of the Euler equations with energy $\|u(t)\|_{L_x^2}^2$ at positive time not exceeding the initial energy $\|u(0)\|_{L_x^2}^2$, is a dissipative solution (see \cite{dls10}). Lions' weak-strong uniqueness theorem says that a continuous in $L^2$ solution of the Euler equations $u(t)$ with $\nabla u + \nabla u^T \in L^1_tL^\infty_x$ is unique in the class of dissipative solutions. 

In this paper we show that smooth solutions of the Navier-Stokes equations are not unique in the analogous class of dissipative solutions - weak solutions with continuous decreasing energy. Note that constructed dissipative solutions can not satisfy the energy inequality, as the weak-strong uniqueness holds in the Leray-Hopf class. Also, our non-uniqueness theorem can not hold in the case $\nu=0$ due to Lions' weak-strong uniqueness principle. To achieve this, we design a scheme that takes advantage of the dissipative term. Indeed, standard convex integration schemes treat $\nu \Delta u$ as an error, and hence these schemes go through in the case of zero viscosity $\nu=0$ as well. 

In this paper, the dissipative solutions to the 3D NSE 
are defined as follows.

\begin{definition}[Dissipative solution]\label{def-dissipative}
    We say that a weak solution $u(t)$ to the 3D NSE \eqref{equa-NSE} is dissipative on $[0,+\infty)$ if its energy $\|u(t)\|_{L_x^2}$ is continuous and decreasing on $[0,+\infty)$.
\end{definition}

\medskip 
The first main result of the present work 
is the construction of dissipative solutions 
to the 3D NSE. 

\begin{theorem}[Dissipative solutions to NSE]\label{thm-adm-nse}
    Given any divergence-free initial data in $H_x^{\frac12}(\mathbb{T}^3)$, there exist infinitely many dissipative solutions to 3D NSE \eqref{equa-NSE} on $[0,\infty)$.    
\end{theorem}

\begin{remark}
Theorem~\ref{thm-adm-nse} is stated for initial data in $H_x^{\frac12}(\mathbb{T}^3)$, but is also holds for any initial data for which there exists a local smooth solution. For instance, local regularity was proved by Fujita and Kato \cite{FK64} for $H_x^{\frac12}(\mathbb{T}^3)$ initial data, and by Koch and Tataru \cite{KT02, MS06} for $BMO^{-1}$ initial data. 
\end{remark}

Our proof actually applies to a more general model, the 3D magnetohydrodynamic (MHD) system on the torus $\T^3:=[0,1]^3$, 
\begin{equation}\label{equa-MHD}
	\left\{\aligned
	&\p_t u - \nu_1 \Delta u+(u\cdot \nabla )u -(B\cdot \nabla )B + \nabla P =0,  \\
	&\p_t B- \nu_2\Delta B + (u\cdot \nabla )B- (B\cdot \nabla )u  =0,
	\endaligned
	\right.
\end{equation}
supplemented with the incompressibility conditions
\begin{align*}
	& \div u = 0, \quad \div B = 0.
\end{align*}
In this problem, $u$ and $B: [0,\9)\times\T^3\rightarrow \R^3$ are the velocity and magnetic fields, respectively,
$P:[0,\9)\times\T^3\rightarrow \R$ stands for the pressure of fluids,
$\nu_1, \nu_2\ (> 0)$ are the viscous and resistive coefficients, respectively. 

Similarly to Definition~\ref{def-dissipative}, we say that a global weak solution to the MHD system is dissipative if its energy 
\[
e(t):= \|u(t)\|_{L_x^2}^2+\|B(t)\|_{L_x^2}^2
\]
is continuous and decreasing on $[0,\infty)$. The existence and non-uniqueness of dissipative solutions to the 3D MHD system is formulated in the theorem below.

\begin{theorem}[Dissipative solutions to MHD]\label{thm-adm-mhd}
    Given any divergence-free initial data in $\cH^{\frac12}(\mathbb{T}^3)$, there exist infinitely many dissipative solutions to 3D MHD equations \eqref{equa-MHD} on $[0,\infty)$.
\end{theorem}

\begin{remark}
We note that, similar to the case of Euler equations \cite{L96}, one can prove that strong solutions to the Ideal MHD equations are unique in the class of weak solutions with continuous decreasing energy. Consequently, Theorem~\ref{thm-adm-mhd} shows that such a weak-strong uniqueness principle does not hold for the MHD equations with positive viscosity and resistive coefficients.
\end{remark}

In fact, in this paper we prove the following  stronger result.

\begin{theorem} \label{Thm-Non-MHD}
$(i)$ Continuous energy solution: 
Given any divergence-free initial data in $\cL^2_x$, there exist infinitely many global weak solutions to \eqref{equa-MHD} in the class $(u,B)\in C({[0,\infty)}; \cL_x^2)$ whose energy profiles 
\begin{align}\label{energy-d}	
e(t)= \|u(t)\|_{L^2_x}^2+\|B(t)\|_{L^2_x}^2. 
\end{align}
and cross helicity profiles
\[
h(t) = \int_{\T^3} u(t)\cdot B(t) \,\d x,
\]
are distinct.

$(ii)$ Dissipative energy solution:  if 
in addition the initial data is in $\cH^{\frac12}_x$,  
then there exists $T>0$ 
such that, for any $\varepsilon >0$ 
there exist 
two pairs of continuous functions $(e^{\mathrm{l}}, e^{\mathrm{h}}): [0,\infty)\rightarrow [0,\9)$ which are 
decreasing on the whole time interval $[0,\infty)$,  smooth on $(\varepsilon,T)$, with  
\[
e^{\mathrm{l}}(t) < e^{\mathrm{h}}(t), \qquad \forall \varepsilon < t<T,
\]
\begin{align*} 
e^{\mathrm{l}}(0)=e^{\mathrm{h}}(0)=\|u(0)\|_{L^2_x}^2+\|B(0)\|_{L^2_x}^2, 
\end{align*} 
and such that for every $e \in C([0,\infty))$, smooth on $(0,T)$, with
\[
e^{\mathrm{l}}(t) \leq e(t) \leq  e^{\mathrm{h}}(t), \qquad \forall t\in (0,T),
\]
there exists a global weak solution $(u,B)\in C({[0,\infty)}; \cL_x^2)$ to \eqref{equa-MHD} 
	satisfying \eqref{energy-d}. 
\end{theorem}

\begin{remark}
The non-uniqueness in part (i) of Theorem~\ref{Thm-Non-MHD} is instantaneous. More precisely, there are two solutions which are distinct on any time interval $(0,\tau)$, $\tau>0$. 
\end{remark}

\begin{remark}
$(i)$ 
In a very recent work 
\cite{czz24}, we constructed stochastic solutions 
with continuous energy for both the deterministic and stochastic NSE.  
The proof there is based on a backward convex integration scheme, 
which can treat arbitrarily prescribed $L_\sigma^2$ initial data, 
but does not yield a decreasing energy. 
To prove Theorem \ref{Thm-Non-MHD} we design what we call a telescoping convex integration scheme, 
which results in a global decreasing energy. 
This method also applies to the stochastic case, 
i.e., 
stochastic NSE and stochastic MHD driven by additive noise. 
Since the present paper is already quite technical, 
we focus on the deterministic case. 

$(ii)$
Let us also mention that the convex integration for MHD requires new constructions of 
the intermittent velocity and magnetic perturbations, 
which are quite different from those of the NSE and 
need to match the geometrical structure of the MHD system. 
\end{remark}

\subsection{Review of related literature} \label{subsec-review}
In the seminal work \cite{leray1934},  
Leray proved the existence of global solutions, 
which are now referred to Leray-Hopf solutions,  
due to also the important contributions
by Hopf \cite{hopf1951} in bounded domains.  
Similar solutions have been obtained for the MHD system by Sermange and Temam \cite{ST83}. For comprehensive surveys on the Navier-Stokes and MHD equations, we refer to monographs \cite{CF, d2001, rrs16, T}.

\subsubsection{Convex integration for the NSE: intermittent schemes}
Since the seminal paper \cite{dls09} where De Lellis and Sz\'{e}kelyhidi introduced the method of convex integration to construct of non-unique solution of the Euler equation, significant progress have been made towards non-uniqueness for various hydrodynamic models. 
An important achievement is the resolution of the flexible part of Onsager's conjecture, accomplished by Isett \cite{I18}.
Another significant achievement was made by Buckmaster and Vicol \cite{bv19b} who introduced an intermittent convex integration to prove the first non-uniqueness result for weak solutions to the 3D NSE. Later, convex integration schemes addressing the Cauchy problem for the 3D NSE were constructed by Buckmaster, Colombo and Vicol in 
\cite{bcv21} and  Burczak, Modena and Sz\'{e}kelyhidi in \cite{bms21}, where non-unique solutions in $C((0,T]; L^2_\sigma)$ were constructed for arbitrary divergence-free initial data in $L^2$. The jump discontinuity of the energy at the initial time was recently removed by the authors in \cite{czz24}, where non-unique solutions in $C([0,T]; L^2_\sigma)$ were constructed.
The sharp non-uniqueness for NSE near endpoints of Lady\v{z}enskaja-Prodi-Serrin (LPS) criteria was proved in \cite{cl20.2,cl23}, where termoral intermittency was introduced to convex integration, and later extended to the 3D hyper-dissipative NSE in  \cite{lqzz22}.

\subsubsection{Convex integration for the MHD}

While fluid turbulence has been extensively studied through the NSE, magnetohydrodynamic turbulence, modeled by the MHD system, observed in various settings such as laboratory experiments on fusion confinement devices and astrophysical systems.

As pointed out by Buckmaster and Vicol  \cite{bv19r}, the application of convex integration techniques to MHD-type equations is still in its infancy. 
A key feature of the MHD system is the interaction between velocity and magnetic fields, which is not present in the NSE.  
One quick glance at the difference 
is, that the  anti-symmetric nonlinearity in the magnetic equation necessitates the utilization of a Second Geometrical Lemma (see Lemma \ref{geometric lem 2} below) for the construction of perturbations.  
As a consequence, 
the direction of the magnetic perturbation flows  shall be perpendicular to that of the velocity 
perturbation flows.  

This particular structure of the MHD equations poses a major challenge in constructing suitable intermittent flows, 
which shall respect the geometric constraints of the MHD and, 
simultaneously,  
exhibit sufficiently strong intermittency to control the viscosity and resistivity.

\subsubsection{$L^\infty$-scheme for ideal MHD} 
For the ideal MHD without the viscosity and resistivity, 
the $L^\infty$ convex integration has been 
developed in a series of papers 
by \cite{FL18, fls21, fls21.2, fls22}.

The ideal MHD system posses three invariants:
	\begin{enumerate}
		\item[$\bullet$] The total  energy:\ \
		$\displaystyle \mathcal{E}(t) =\frac12\int_{\T^3}|u(t,x)|^2+|B(t,x)|^2 \d x$;
		\item [$\bullet$] The cross helicity: \ \
		$\displaystyle \mathcal{H}_{\omega,B}(t) =\int_{\T^3}u(t,x)\cdot B(t,x) \d x$;
		\item [$\bullet$] The magnetic helicity: \ \
		$\displaystyle \mathcal{H}_{B,B}(t):=\int_{\T^3}A(t,x)\cdot B(t,x) \d x$,
		\end{enumerate}	
where $A$ is a mean-free vector field given by $\curl A=B$. In contrast to the Onsager conjecture for Euler equations, it was conjectured by Buckmaster and Vicol \cite{bv21} that $L^3_{t,x}$ is the threshold space for the solutions to ideal MHD system to conserve 
magnetic helicity.

On the rigidity side, the magnetic helicity conservation for the 3D ideal MHD was proved
in \cite{KL07,A09,FL18}. Moreover, the construction of non-trivial smooth subsolutions for the 3D ideal MHD system was initiated by Faraco and Lindberg  \cite{FL18}. 

On the flexible side, Faraco, Lindberg, and Sz\'ekelyhidi \cite{fls21} constructed bounded weak solutions to the ideal MHD which violate energy conservation while maintaining magnetic helicity.
Recently, in the remarkable paper \cite{fls21.2}, 
Faraco, Lindberg, and Sz\'ekelyhidi addressed the conjecture regarding 
the $L_{t,x}^3$ threshold for the magnetic helicity conservation, based on the convex integration method involving staircase laminates.

\subsubsection{Intermittent scheme for MHD}  

Intermittent shear flows were first used by 
Beekie, Buckmaster and Vicol  \cite{bbv20} for 
the ideal MHD, 
which resulted in a construction of weak solutions with non-conserved magnetic helicity, 
and thus, 
disproved Taylor's conjecture for weak solutions to the ideal MHD. Such building blocks can achieve intermittency dimension $D=2+$, and hence can control the hypo viscosity and resistivity of the form $(-\Delta)^\alpha$ with $\alpha \in [0,1/2)$.
We also refer to \cite{dai21} for the non-uniqueness of weak solutions in the space $C^0_t L^2_x \cap L^2_t \dot{H}^1_x$ to the Hall MHD. For the classical MHD system, with the help of temporal intermittency (introduced in \cite{cl21, cl20.2}), a sharp non-uniqueness at one endpoint of the LPS criteria was achieved by Li and the last two authors in \cite{lzz21,lzz21.2}.

Very recently, Miao and Ye \cite{MY22} designed the first convex integration construction of finite energy weak solutions to the classical MHD. For $H_x^\beta$ $(\beta>0)$ initial data, they constructed non-unique solutions in $C_TL^2_x$. This was achieved with intermittent box flows, and inverse traveling wave and heat conduction  flows to cancel the errors.

\medskip 
\subsection{$\Lambda$-MHD approximation and telescoping energy levels} 
\label{Subsec-Novelty}

We first recall a challenge in constructing continuous in $L_x^2$ solutions with arbitrary initial data in $L_x^2$. The gluing technique developed by Buckmaster, Colombo and Vicol \cite{bcv21} or the method based on constructing appropriate cut-off functions developed by Burczak, Modena and Sz\'{e}kelyhidi \cite{bms21}, result in weak solutions of the NSE in $C((0,1];L^2)$, but with a jump at the initial time. The energy jump was a consequence of a non-vanishing initial Reynolds stress at the origin. In \cite{czz24}, the authors introduced a new approximation, called the $\Lambda$-NSE, to construct smooth solutions of the NSE-Reynolds system with Reynolds stresses vanishing at the initial time. As a consequence, a convex integration scheme in \cite{czz24} produced weak solutions of the NSE in $C([0,1];L_x^2)$ for any initial data in $L_x^2$.

In this paper we construct global dissipative solutions of the NSE and MHD equation for critical (say, $H^\frac 12_x$) initial data. Continuous in $L_x^2$ solutions for such initial data were already constructed in \cite{bcv21,bms21}. In these works, as common in convex integration schemes, the dissipative term is treated as an error, and hence those schemes go through in the case of zero viscosity. Constructing  solutions with decaying energy requires a different approach. Due to the weak-strong uniqueness principle in the inviscid case, a convex integration scheme producing such (dissipative) solutions should take advantage of the dissipation term $\Delta u$. Indeed, we will exploit a new $\Lambda$-MHD system, where the dissipation term plays a crucial role in producing a smooth approximate solution of the NSE whose energy profile decays fast enough in comparison to the size of the Reynolds and magnetic stresses. This approach is motivated by our previous work \cite{czz24} where the Reynolds stress had to be controlled near the initial time due to the roughness of the initial data. In the present paper we tweak the design of the approximation to achieve a different purpose - the decay of the energy for smooth (or critical) initial data.

More precisely, we choose a wavenumber $\Lambda(t)$ satisfying
\[
\lim_{t \to 0^+}\Lambda(t)=\infty, \qquad \text{and}, \qquad \lim_{t \to T-}\Lambda(t) = \infty,
\]
for some appropriate $T>0$ (see Figure \ref{fig:1}), and design an approximation of the MHD equations, which we call $\Lambda$-MHD, such that only low-frequency modes below $\Lambda(t)$ are advected by the velocity fields, and modes above $\Lambda(t)$ just dissipate. An appropriate frequency truncation operator ensures the energy balance for the solutions, and guarantees that
the $\Lambda$-MHD  can be reformulated as a relaxed system  with suitable Reynolds stress (see \eqref{r0u} and \eqref{r0b}) vanishing at the initial time $t=0$ as well as $T>0$:
\begin{align*}
(\mathring{R}_{0}^u(t),
\mathring{R}_{0}^B(t)) \to 0, 
\quad {\rm in}\ \mathcal{L}^1_x
\quad {\rm as}\ t\to 0^+\ 
{\rm or}\ t\to T^-.  
\end{align*}
Time $T$ is chosen large enough based on the eventual regularity estimates. We show that the solution of the $\Lambda$-MHD equations is smooth before $\Lambda(t)$ blows up, i.e., on $(0,T)$. Choosing $T$ large enough ensures that the solution remains smooth even after $T$ where $\Lambda=\infty$, i.e., when $\Lambda$-MHD coincides with the classical MHD system.

\begin{figure}[b]
    \centering
    \includegraphics[width=0.7\linewidth]{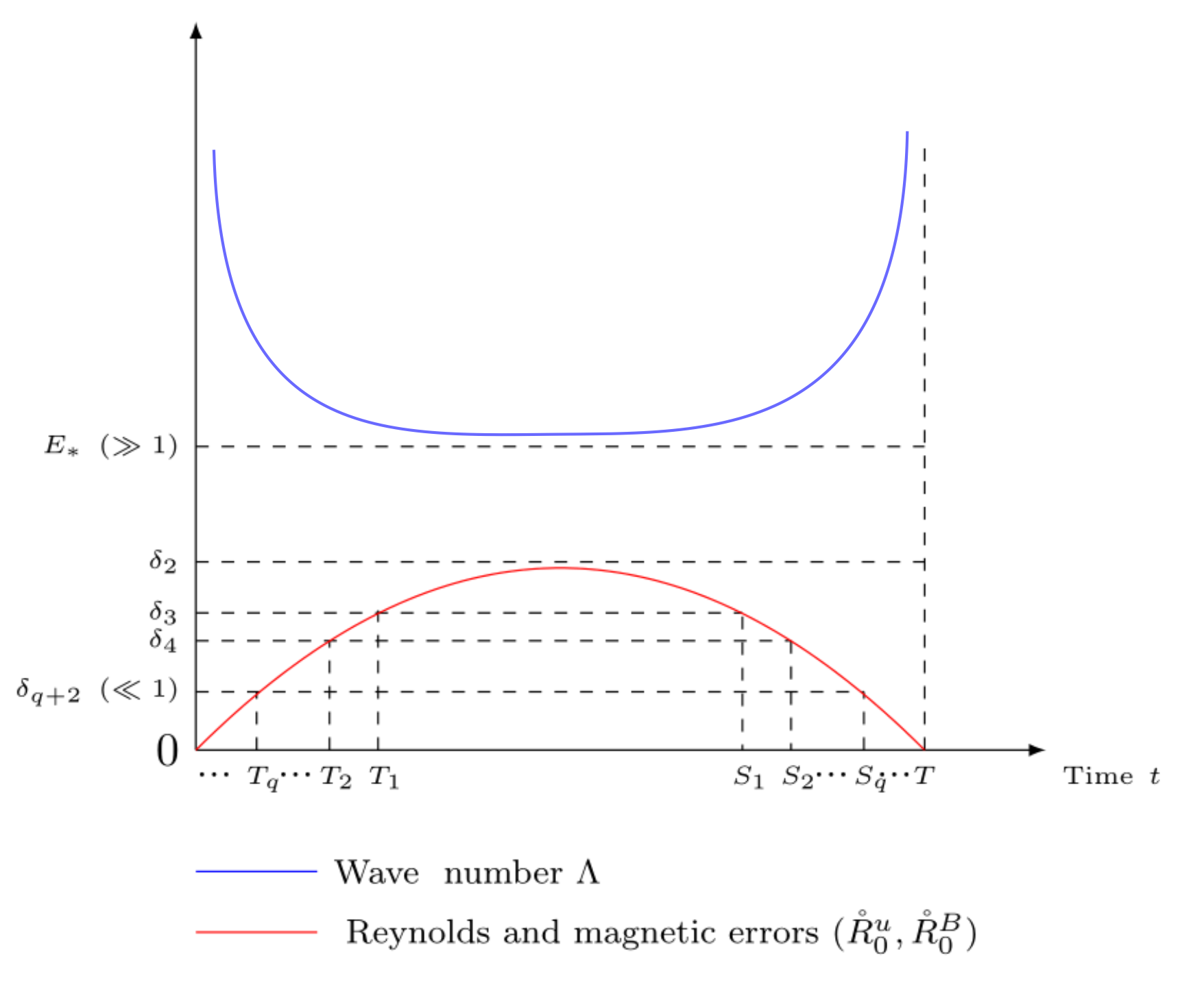}
    \caption{Wavenumber, 
    initial Reynolds and magnetic errors}
    \label{fig:1}
\end{figure}

Consequently, the convex integration construction can be focused on the inner temporal interval $[0,T]$. In this interval,
besides a backward time sequence 
$\{T_q\}$, 
we also select 
a forward time sequence $\{S_q\}$ 
according to the decay property of Reynolds and magnetic stresses at both times $0$ and $T$,  
\begin{align*}
	\|(\mathring{R}_{0}^u,\mathring{R}_{0}^B )\|_{  C_{[0,T_{q}]\cup [S_{q},T]} \cL^{1}_x} 
    \lesssim \delta_{q+2}, 
    \quad q\geq 0. 
\end{align*}  
This is possible because the initial Reynolds and magnetic stresses vanish at times $0$ and $T$. 
In practice, the backward and forward times will be modified  slightly by small parameters to take into account the mollification procedure, 
see \eqref{asp-q} below.  
It is worth noting that 
the decay parameter $\delta_q$ above is chosen as a power law of the form $10^{-3q}$, 
which decays much slower than the usual one $a^{-\beta b^q}$ 
with $a,b \gg  1$, $\beta>0$. 
Moreover, in contrast to standard convex integration schemes, 
the frequency $\lambda_q$ of building blocks depends on 
the backward and forward times $\{T_q, S_q\}$, as well as on the wavenumber  $\Lambda(t)$ in the $\Lambda$-MHD system. 
We refer to Subsection \ref{Subsec-time}  below for more details. 

To construct dissipative solutions, we will design an energy channel  
$[0,T] \ni t \mapsto [\overline e(t), \underline e(t)]$
to allow for infinitely many decreasing energy profiles. 
The boundaries of the energy channel 
are given by the upper and lower energy profiles, respectively, 
given by 
\begin{align} \label{e-f-intro}
    \overline e(t):= \|( u_0(t), B_0(t))\|_{\cL^2_x}^2+\overline f(t),\quad \underline e(t):= \|( u_0(t), B_0(t))\|_{\cL^2_x}^2+\underline f(t), 
    \quad t\in [0,T]. 
\end{align} 
As common in convex integration schemes, the velocity and magnetic perturbations produce energy errors that grow in time. To obtain a decay of the total energy, one has to control 
the growth rate of the energy perturbations in a more delicate  quantitative way. For this purpose, we choose a wavenumber  $\Lambda(t)$ for the $\Lambda$-MHD system with a blow-up rate $E_* t^{-1/3}$ near the initial time, which gives us a flexibility of choosing parameter $E_*$ large in order to decrease errors. Indeed, we derive an improved decay estimate for the Reynolds and magnetic stresses near the initial time 
(see \eqref{decay-R0-refined} below), 
\begin{align*}
\|(\mathring{R}_{0}^u,\mathring{R}_{0}^B ) \|_{C_{[0,t]}\cL^{1}_x} 
\lesssim E_*^{-3}t, 
\end{align*}  
where the implicit constant depends on the $H_x^3$ regularity of initial data.
Choosing $E_*$ sufficiently large will enable us to control the growth rate of energy perturbations to ensure that it does not exceed the decay rate of the initial energy profile, see \eqref{oe-de} and \eqref{ue-de} below. 

Note that, 
on time increments 
$[\underline{T}_{q+1},\overline{T}_{q}]$ and $[\us_{q},\os_{q+1} ]$, 
the energy perturbations have different decay orders 
\begin{align*}
 \delta_{q+2} 
 \lesssim e(t)-(\| u_0 (t)\|_{L^2_x}^2+\|  B_0 (t)\|_{L^2_x}^2) 
 \lesssim \delta_{q+1}, 
\quad t\in [\underline{T}_{q+1},\overline{T}_{q}]\cup [\us_{q},\os_{q+1}], 
\end{align*} 
where $\underline{T}_{q+1},\overline{T}_{q}$ 
are very small modifications of the time $T_q$, and so are $\underline{S}_{q+1},\overline{S}_{q}$.  
This fact allows to construct the upper energy perturbation profile $\overline{f}$ 
as a linear function on $[\overline{T}_{q+1},\overline{T}_{q}] 
\cup [\underline{S}_{q},\underline{S}_{q+1}]$  connecting levels $\delta_{q+1}$ and $\delta_{q+2}$. 
The lower energy perturbation profile $\underline{f}$ can be constructed in an analogous manner but on 
slightly different time intervals 
$[\underline{T}_{q+1},\underline{T}_{q}] 
\cup [\overline{S}_{q},\overline{S}_{q+1}]$.

We observe that $\overline{f}$ is constant in the inner region $[\ot_1,\us_1]$ and $\underline{f}$ is constant in $[\ut_1,\os_1]$. In the outer regions $[0,T]/[\ot_1,\us_1]$ and $[0,T]/[\ut_1,\os_1]$, adjusting $E_*$ ensures that the energy changes caused by perturbations are negligible relative to the initial energy.

These features of the scheme enable us to obtain a desired decreasing energy channel \eqref{e-f-intro} thereby allowing for infinitely many decreasing smooth energy profiles,   
see Figure \ref{fig:2}.

\begin{figure}[b]
    \centering
    \includegraphics[width=0.6\linewidth]{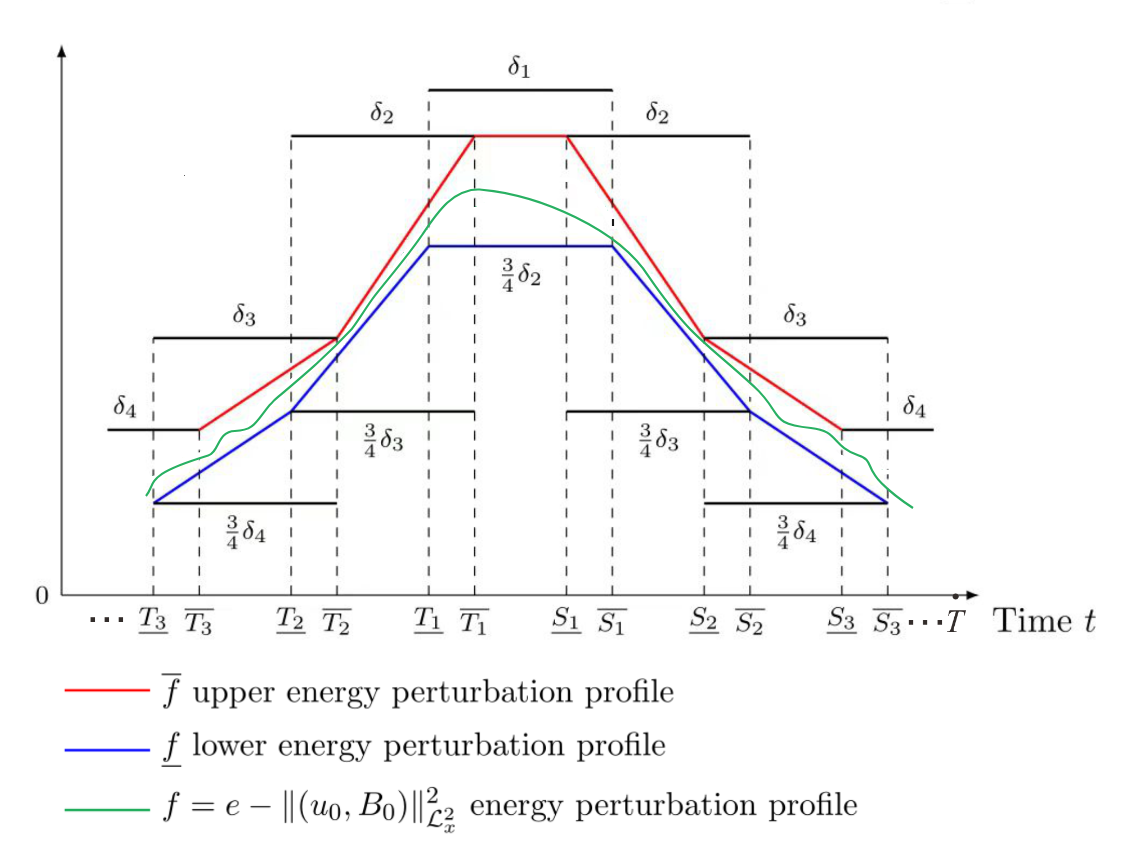}
    \caption{Upper and lower energy perturbation profiles}
    \label{fig:2}
\end{figure}

\medskip 
\paragraph{\bf Heat correctors for the 3-D MHD}  

The geometric structure of MHD forbids the use of the original 3D NSE intermittent building blocks, such as intermittent Beltrami flows and jets \cite{bv19b, bcv21}. On the other hand, the intermittent MHD shear flows used in \cite{bbv20, lzz21, lzz21.2} can only achieve the intermittency dimension $D=2+$ and hence cannot be used to construct solutions to the viscous MHD with bounded energy. In a remarkable paper \cite{MY22}, Miao and Ye introduced box flows, which resemble intermittent jets, but 
supported on cuboids with three different length scales. As in the case of jets, the intermittency dimension of such building blocks can go below $D=1$, in which case the error coming from the linear term is small. On the other hand, the interactions $\div (W\otimes W)$, $\div (D\otimes W)$, $\div (D\otimes D)$, and $\div (W\otimes D)$, where $W$ is the velocity building block and $D$ is the magnetic building block, result in large errors. While the temporal corrector introduced in \cite{bv19b} can cancel the first two errors $\div (W\otimes W)$ and $\div (D\otimes W)$, it cannot be used for the other two. To resolve this issue, Miao and Ye introduced inverse traveling wave flows and the heat conduction flows. Combined with the corrector from \cite{bv19b}, these three temporal correctors cancel all the large error terms.

Inspired by \cite{MY22}, we also use highly concentrated box-like building blocks, but with just one heat corrector instead of three to streamline the scheme. Heuristics behind this unified approach can easily be understood considering solutions to the heat equation with a space-time periodic, zero-mean in space source $f(t,x)$, 
\begin{equation} \label{eq:heat_intro}
\p_t u-\Delta u = f.
\end{equation}
The long time behavior of the solution to this equations is described by what is called the pullback attractor, which, in this case, is a single trajectory
\[
u^{\mathcal{A}}(t,x) := \int_{-\infty}^t e^{ (t-\tau)\Delta}  f(\tau,x) \, \d \tau =
 \int_{t_0}^t e^{ (t-\tau)\Delta}  f(\tau,x) \, \d \tau,
\]
for some $t_0$, which we assume is zero appropriately translating $f$ in time. Here the second equality follows from the time periodicity of $u^{\mathcal{A}}(x,t)$. All the trajectories (solutions to \eqref{eq:heat_intro}) converge to $u^{\mathcal{A}}(t,x)$ in the pullback sense (as the initial time goes to minus infinity). This holds even for the 3D NSE with small forces, see \cite{CK}. 

If spatial oscillations of $f$ dominate the temporal oscillations, then $(-\Delta)^{-1} f(t,x)$ is the leading term in $u^{\mathcal{A}}(t,x)$. However, in applications to convex integration, the non-diagonal interactions between building blocks result in errors of type 
\[
f= P_{H}\div \left(w \otimes w  \right),
\]
where $w$ is the velocity (or magnetic field) perturbation, which, for simplicity, is assumed to be a just a $T$-periodic in time building block. To control the size of the error $\Delta w$, the perturbation has to be highly intermittent (with the intermittency dimension $D<1$). Then it is easy to check that 
\[
\|(-\Delta)^{-1} \div(w\otimes w)\|_{L^2_x} \gg \|w\|_{L^2_x}.
\]
Thus, temporal oscillations of the perturbation $w$ have to dominate spatial ones.

Let
\[
\bar f(x):= \frac{1}{T}\int_0^T f(\tau,x)\, \d \tau, \qquad\p_t^{-1} f (t,x):= \int_{0}^{t} \big( f(\tau,x) - \bar f(x) \big) \, \d \tau,
\]
where $T$ is the time period of $f$. Note that $\p_t^{-1}f(t,x)$ is also time periodic. Then
\begin{equation} \label{eq:intro-heat-decomp}
u^{\mathcal{A}}(t,x)= \int_{0}^t e^{ (t-\tau)\Delta} \bar f(x) \, \d \tau +\p_t^{-1} f(t,x)+ \int_{0}^t e^{ (t-\tau)\Delta} \Delta \p_t^{-1} f(\tau,x) \, \d \tau.
\end{equation}
When the frequency of temporal oscillations is high enough, $(-\Delta)^{-1} \bar f(x)$, or the first term in \eqref{eq:intro-heat-decomp} is the leading term of $u^{\mathcal{A}}(t,x)$. Since time averaging of $f$ traveling along geodesics results in a larger intermittency dimension $D>1$, we get 
\[
\|u^{\mathcal{A}}\|_{L^2_x} \ll \|w\|_{L^2_x},
\]
and hence $u^{\mathcal{A}}$ can be used as a temporal corrector to cancel the error $f$.

\medskip
{\bf Notations.} Let $\mathbb{N}_+$ denote the set of positive integers. For $p\in [1,\infty]$ and $s\in \R$, we set
\begin{align*}
 L^p_x:=L^p(\T^3),\quad W^{s,p}_x:=W^{s,p}(\T^3),\quad H^s_x:=H^s(\T^3),
\end{align*}
where $W^{s,p}_x$ is the Sobolev space and $H^s_x=W^{s,2}_x$. We also use $L^2_\sigma$ for divergence free functions in $L^2_x$. Given any Banach space $\mathbb{X}$,
we denote by $C([0,T];\mathbb{X})$ the space of continuous functions from $[0,T]$ to $\mathbb{X}$,
equipped with the norm $\|u\|_{C_T\mathbb{X}}:=\sup_{s\in [0,T]}\|u(s)\|_\mathbb{X}$.
In particular,
for $N\in \mathbb{N}_+$ we set
\begin{align*}
	\norm{u}_{C_{T,x}^N}:=\sum_{0\leq m+|\zeta|\leq N}
	\norm{\p_t^m \na^{\zeta} u}_{L^\9([0,T];L_x^\9)},
\end{align*}
where $\zeta=(\zeta_1,\zeta_2,\zeta_3)$ denotes the multi-index
and $\na^\zeta:= \partial_{x_1}^{\zeta_1} \partial_{x_2}^{\zeta_2} \partial_{x_3}^{\zeta_3}$.
We also use the product spaces
 $$\mathcal{L}^p_x:= L^p(\T^3) \times L^p(\T^3),\ \mathcal{H}^{s}_x:= H^{s}_x\times H^{s}_x, \
\cC^1_{x}:= C^1_{x} \times C^1_{x}.$$
Throughout this paper the notation $a\lesssim b$ means that $a\leq C b$ for some constant $C>0$.
\section{Telescoping convex integration scheme} 

Let us first give the definition of solutions to \eqref{equa-MHD}
taken in the following distributional sense.

\begin{definition} \label{weaksolu}	Given any weakly divergence-free initial datum $(v_0,H_0)\in \cL^2(\T^3)$,  we say that $(u,B)\in C_w([0,+\infty);  \cL^2(\T^3))$ is a  weak solution to the MHD system \eqref{equa-MHD},
if the following hold:
\begin{itemize}
\item For all $t\geq 0$, $(u(t,\cdot), B(t,\cdot))$ is divergence free in the sense of distributions and have zero spatial mean.
\item $(u, B)$ solves \eqref{equa-MHD} in the sense of distributions, i.e.,	for any divergence-free test function  $\varphi  \in C^{\infty}_c ([0,+\infty) \times \mathbb{T}^3)$ and almost every $s>0$,
\begin{align*}
		\int_{\mathbb{T}^3} u(s)\cdot\vf(s,x) \d x &= \int_{\mathbb{T}^3} v_0 \cdot\vf(0,x) \d x+\int_0^s \int_{\mathbb{T}^3}  \partial_t \varphi \cdot u + \nu_1 \Delta \varphi \cdot u  + \nabla\varphi :( u \otimes u - B \otimes B) \d x \d t ,		\\
		\int_{\mathbb{T}^3} B(s)\cdot\vf(s,x) \d x &= \int_{\mathbb{T}^3} H_0 \cdot\vf(0,x) \d x+\int_0^s	\int_{\mathbb{T}^3}  \partial_t \varphi \cdot B +\nu_2 \Delta \varphi \cdot B + \nabla\varphi  :(B \otimes u- u \otimes B) \d x \d t.
		\end{align*}
\end{itemize}
 Here, for any $3\times 3$ matrices $A=({A_{ij}})$ and $S=({S_{ij}})$, we set 	$ A:S=\sum_{i,j=1}^{3}A_{ij}S_{ij}$. 
\end{definition}

We consider the following relaxed MHD-Reynolds system
for each integer $q\in \mathbb{N}$,
\begin{equation}\label{mhd1}
	\left\{\aligned
	&\p_t \u-\nu_1\Delta \u  + \div( \u \otimes \u - \h \otimes \h )+\nabla P_q=\div \ru,  \\
	&\p_t \h-\nu_2\Delta \h + \div( \h \otimes \u - \u \otimes \h )=\div \rb , \\
	&\div \u = 0, \quad \div \h= 0,\\
	& \u(0) = v_0, \quad \h(0)= H_0,
	\endaligned
	\right.
\end{equation}
where the Reynolds stress $\ru$ is a symmetric traceless $3\times 3$ matrix,
the magnetic stress $\rb$ is a skew-symmetric $3\times 3$ matrix.

\subsection{Initialization: 
the $\Lambda$-MHD system}   \label{Subsec-back-CI}
Given any divergence-free and mean-free initial data $v_0, H_0\in L^2_x$, we choose 
\begin{align}\label{def-T}
    T=\left(\frac{c_1}{c_2}\right)^{\frac14}\nu^{-2}\|(v_0,H_0)\|_{\cL^2_x}^2,
\end{align}
where $\nu=\min\{\nu_1, \nu_2\}$, 
and $c_1$, $c_2$ are universal constants in the energy estimate \eqref{est-hi-innp} below. Then we introduce a new type of approximate MHD system where the nonlinearties only contain low modes below certain {\it time-depending} frequency $\Lambda$ which is different from the usual Galerkin approximation:
\begin{align}\label{equa-mhd-2}
	\begin{cases}
		\p_t u-\nu_1 \Delta  u + \P_{<\Lambda(t)}  (\P_{<\Lambda(t)} u\cdot \nabla \P_{<\Lambda(t)}u-\P_{<\Lambda(t)} B\cdot \nabla \P_{<\Lambda(t)}B )+\nabla p=0,\\
		\p_t B-\nu_2 \Delta  B + \P_{<\Lambda(t)}  (\P_{<\Lambda(t)} u\cdot \nabla \P_{<\Lambda(t)}B-\P_{<\Lambda(t)} B\cdot \nabla \P_{<\Lambda(t)}u )=0,\\
		\div   u=0,\ \div   B=0,\\
		u(0)=v_0,\ B(0)=H_0.
	\end{cases}
\end{align}
Here, the wavenumber $\Lambda:(0,T)\rightarrow\R^+$ 
satisfies the following properties
\begin{itemize}
    \item $\Lambda(t)=E_* t^{-\frac13}$ on $(0,T/2]$; 
    \item $\Lambda$ is an increasing function on $[T/2,T)$; 
    \item $\Lambda(t)\rightarrow +\infty$ as $t\rightarrow T^-$.
\end{itemize}
Here, $E_*$ equals one in the case of the rough $\mathcal{L}_x^2$ initial data, 
and in the case of the critical $\mathcal{H}^\frac 12_x$ initial data 
$E_*$ is a large constant depending on the 
$\mathcal{H}^3_x$-norm of the regular solution 
to MHD \eqref{equa-MHD} at a small positive time, see \eqref{decre-condi-c} below for the 
explicit choice of $E_*$.

\begin{remark} 
We note that the wavenumber $\Lambda$ decreases on $(0,T/2]$, but increases on $[T/2,T]$, and $\Lambda$ explodes at times $0$ and $T$. 
As explained in Subsection \ref{Subsec-Novelty}, 
the explosion of $\Lambda$ at two endpoints 
is important to guarantee the vanishing of Reynolds and magnetic stresses at the corresponding times, 
which in turn allow to choose suitable backward and forward time sequences to run the whole convex integration iterative scheme. 
\end{remark}

Let $\varphi: \R^3 \rightarrow [0,1]$ be a smooth radial function such that
\begin{align*}
	\varphi(x)=\begin{cases}
		1,\quad |x|\in [0,1),\\
		0,\quad |x|\in [2,+\infty).\\
	\end{cases}
\end{align*} 
Set $\varphi_{\Lambda(t)}(\xi):= \varphi(\xi/\Lambda(t))$. For any $t\in (0,T)$, 
the frequency truncated operator $\P_{<\Lambda(t)}$ is defined by
\begin{align}\label{def-plambda}
	\widehat{\P_{<\Lambda(t)}u}(\xi):= \hat u(\xi)\varphi_{\Lambda(t)}(|\xi|).
\end{align}

The key properties of the $\Lambda$-MHD system are contained in Theorem~\ref{thm-fns} below.

\begin{theorem}   \label{thm-fns}
Given any divergence-free and mean-free initial datum $(v_0,H_0)\in \cL^2_x$, the following hold:
\begin{itemize}	
\item[(i)] 
Global existence: there exists a global solution $(\wt u,\wt B)$ to \eqref{equa-mhd-2} which is smooth for positive times and satisfies the energy equality
\begin{align}\label{eq-e-2}
\frac12  \|(\wt u, \wt B)(t)\|_{\cL^2_x}^2  + \int_{0}^{t}\nu_1 \| \nabla \wt u(s)\|_{L^2_x}^2+\nu_2 \|\nabla \wt B (s)\|_{L^2_x}^2\d s = \frac12  \|(v_0,H_0)\|_{\cL^2_x}^2 , \quad t\geq 0. 
\end{align} 

\item[(ii)] 
Strong continuity at the initial time:  
\begin{align}\label{s-con-v0}
(\wt u,\wt B)(t) \rightarrow (v_0,H_0) \quad 
{\rm strongly\  in}\ \cL^2_x, \quad  
{\rm as}\ t\to 0^+.
\end{align}
\
\item [(iii)]  
$(\wt u,\wt B)$ is a smooth solution to the classical MHD system \eqref{equa-MHD} on $[T,+\infty)$.

\item[(iv)]  
Vanishing of nonlinearity on high modes 
at two endpoints:
\begin{align}\label{deacy-non}
&	\|\P_{\geq \Lambda}(\P_{< \Lambda}\wt u\mathring \otimes \P_{< \Lambda}\wt u)\|_{C([0,t]; L^1_x)}+  	\|\P_{\geq \Lambda}(\P_{< \Lambda}\wt B\mathring \otimes \P_{< \Lambda}\wt B)\|_{C([0,t]; L^1_x)}  \to 0, \notag\\
&	\|\P_{\geq \Lambda}(\P_{< \Lambda}\wt B\mathring \otimes \P_{< \Lambda}\wt u)\|_{C([0,t]; L^1_x)}+  	\|\P_{\geq \Lambda}(\P_{< \Lambda}\wt u\mathring \otimes \P_{< \Lambda}\wt B)\|_{C([0,t]; L^1_x)}  \to 0,\quad  {\rm as}\ t\to 0^+,
\end{align}
and 
\begin{align}\label{deacy-non-2}
&	\|\P_{\geq \Lambda}(\P_{< \Lambda}\wt u\mathring \otimes \P_{< \Lambda}\wt u)\|_{C([t,T]; L^1_x)}+  	\|\P_{\geq \Lambda}(\P_{< \Lambda}\wt B\mathring \otimes \P_{< \Lambda}\wt B)\|_{C([t,T]; L^1_x)}  \to 0, \notag\\
&	\|\P_{\geq \Lambda}(\P_{< \Lambda}\wt B\mathring \otimes \P_{< \Lambda}\wt u)\|_{C([t,T]; L^1_x)}+  	\|\P_{\geq \Lambda}(\P_{< \Lambda}\wt u\mathring \otimes \P_{< \Lambda}\wt B)\|_{C([t,T]; L^1_x)}  \to 0,\quad {\rm as}\ t\to T^-.
\end{align}
\end{itemize}
\end{theorem}

Now, we choose the solution $(\wt u, \wt B)$ to the $\Lambda$-MHD system \eqref{equa-mhd-2} as the initial step of our telescoping convex integration scheme. 
Note that  the $\Lambda$-MHD system \eqref{equa-mhd-2} can be reformulated as the relaxed MHD-Reynolds system \eqref{mhd1} with the Reynolds stress
\begin{align}
	\mathring{R}_0^u &:= \P_{\geq \Lambda }(\P_{< \Lambda }\wt u\mathring \otimes \P_{< \Lambda }\wt u)+ (\P_{< \Lambda }\wt u\mathring \otimes \P_{\geq \Lambda(t)}\wt u )+ (\P_{\geq\Lambda }\wt u\mathring \otimes \wt u) \notag\\
	&\quad-(\P_{\geq \Lambda }(\P_{< \Lambda }\wt B\mathring \otimes \P_{< \Lambda }\wt B)- (\P_{< \Lambda }\wt B\mathring \otimes \P_{\geq \Lambda(t)}\wt B )- (\P_{\geq\Lambda }\wt B\mathring \otimes \wt B),   \label{r0u}
\end{align}
the  magnetic stress
\begin{align}
	\mathring{R}_0^B &:= \P_{\geq \Lambda }(\P_{< \Lambda }\wt B\mathring \otimes \P_{< \Lambda }\wt u)+ (\P_{< \Lambda }\wt B\mathring \otimes \P_{\geq \Lambda(t)}\wt u )+ (\P_{\geq\Lambda }\wt B\mathring \otimes \wt u) \notag\\
	&\quad- (\P_{\geq \Lambda }(\P_{< \Lambda }\wt u\mathring \otimes \P_{< \Lambda }\wt B)- (\P_{< \Lambda }\wt u\mathring \otimes \P_{\geq \Lambda(t)}\wt B)- (\P_{\geq\Lambda }\wt u\mathring \otimes \wt B),   \label{r0b}
\end{align}
and the pressure
$$ 
P_0 = p-\frac{1}{3} ( \P_{\geq \Lambda }(|\P_{<\Lambda }\wt u|^2-|\P_{<\Lambda }\wt B|^2 )-|\wt u|^2+|\wt B|^2). 
$$

\medskip 
The crucial vanishing property of the initial Reynolds and magnetic stress 
is stated below.

\begin{lemma} [Decay of initial Reynolds and magnetic stresses]  \label{Lem-Rey-decay} As $t \rightarrow 0^+$, we have
\begin{align}\label{pro-r0-ub}
	\|(\mathring{R}_{0}^u,\mathring{R}_{0}^B) \|_{ C_{[0,t]}\cL^{1}_x}+\|(\mathring{R}_{0}^u,\mathring{R}_{0}^B) \|_{ C_{[T-t,T]}\cL^{1}_x} \rightarrow 0.
\end{align}
\end{lemma}

\begin{proof}
By virtue of the strong convergences \eqref{deacy-non} and \eqref{u-l2-decay} below, we have
\begin{align}\label{est-non-u}
	&\quad\|\P_{\geq \Lambda }(\P_{< \Lambda }\wt u\mathring \otimes \P_{< \Lambda }\wt u) \|_{L^{1}_x} +\|(\P_{< \Lambda }\wt u\mathring \otimes \P_{\geq \Lambda }\wt u ) \|_{L^{1}_x}+ \|(\P_{\geq\Lambda }\wt u\mathring \otimes \wt u) \|_{L^{1}_x} \notag\\
	&\leq\|\P_{\geq \Lambda }(\P_{< \Lambda }\wt u\mathring \otimes \P_{< \Lambda }\wt u)\|_{L^{1}_x}+\|\P_{ <\Lambda } \wt u \|_{L^2_x}\|\P_{ \geq \Lambda } \wt u \|_{L^2_x}+ \|\P_{ \geq \Lambda } \wt u \|_{L^2_x}\|\wt u \|_{L^2_x} \notag\\
	&\lesssim \|\P_{\geq \Lambda }(\P_{< \Lambda }\wt u\mathring \otimes \P_{< \Lambda }\wt u)\|_{L^{1}_x}+\|\P_{ \geq \Lambda } \wt u \|_{L^2_x}  \|\wt u \|_{L^2_x} \to 0,\quad \text{as}\ t \to 0^+\ \text{or}\ t \to T^-,
\end{align}
and, similarly
\begin{align}\label{est-non-b}
	&\quad\|\P_{\geq \Lambda }(\P_{< \Lambda }\wt B\mathring \otimes \P_{< \Lambda }\wt B) \|_{L^{1}_x} +\|(\P_{< \Lambda }\wt B\mathring \otimes \P_{\geq \Lambda }\wt B) \|_{L^{1}_x}+ \|(\P_{\geq\Lambda }\wt B\mathring \otimes \wt B) \|_{L^{1}_x}\notag\\
	&\lesssim \|\P_{\geq \Lambda }(\P_{< \Lambda }\wt B\mathring \otimes \P_{< \Lambda }\wt B)\|_{L^{1}_x}+\|\P_{ \geq \Lambda } \wt B \|_{L^2_x}   \|\wt B \|_{L^2_x} \to 0,\quad \text{as}\ t \to 0^+\ \text{or}\ t \to T^-.
\end{align}
Taking into account \eqref{r0u}, we then obtain 
\begin{align*} 
	\|\mathring{R}_{0}^u\|_{ C_{[0,t]}L^{1}_x}+\|\mathring{R}_{0}^u\|_{ C_{[T-t,T]}L^{1}_x}\to 0,\quad \text{as}\ t\to 0^+.
\end{align*}
The case of the magnetic stress  $\mathring{R}_{0}^B$ can be proved in an analogous manner.
\end{proof}

\begin{remark}
For more regular initial data, 
the decay rate near initial time can be improved 
for the Reynolds and magnetic stresses, 
which is important to construct decreasing energy profiles. 
We refer to Lemma \ref{lem-im-decay} below  
for the case of $H_x^3$ regular data. 
\end{remark}

\subsection{Selection of backward and forward  times} \label{Subsec-time} 

We choose the amplitude parameter $\delta_{q}$ as follows:
\begin{equation}\label{la}
	\delta_{q+3}=10^{-3q},\quad q\in \mathbb{N}.
\end{equation}
The parameters $\delta_0$, $\delta_1$ and $\delta_2$ will be chosen later.

In view of the decay properties of $(\mathring{R}_0^u, \mathring{R}_0^B)$ 
in Lemma \ref{Lem-Rey-decay}, 
we can choose the backward time sequence $\{T_q\}$ and the forward time sequence $\{S_q\}$ in 
the following iterative way.

Let $\delta_2$ be a large parameter such that 
\begin{align}
	& \|(\mathring{R}_{0}^u,\mathring{R}_{0}^B ) \|_{  C_{[0,T]} \cL^{1}_x}\leq \frac18c_*\delta_2, \label{asp-0} 
\end{align}
where $c_*>0$ is a small parameter such that
\begin{align}\label{def-c8}
	0< c_*< \min\left\{1,\frac{\ve_B}{200M_*^2|\Lambda_{B}|} , \frac{\ve_u}{10^3(1+2\ve_B^{-1}M_*^2|\Lambda_{B}|)}\right\},
\end{align}
with $\ve_u$, $\ve_B$ and $M_*$ being the geometric constants given by Geometric Lemmas~\ref{geometric lem 1} and \ref{geometric lem 2}. 

\begin{remark}
The small parameter $c_*$ is introduced in order to control the energy of velocity and magnetic perturbations (see e.g. \eqref{est-gamma-part2} below),  such that the energy variation of the approximate solutions caused by perturbations is of acceptable small size and can be controlled to ensure the energy iterative estimate \eqref{energy-est} later.    
\end{remark}

Then, we choose $0<T'_1<T/2<S_1'<T$, $T_2'\in (0,T'_1/2)$ and $S_2'\in (\frac{T+S'_1}{2},T)$ such that
\begin{align*}
	\|(\mathring{R}_{0}^u,\mathring{R}_{0}^B ) \|_{  C_{[0,T'_1]\cup [S'_1,T]} \cL^{1}_x}\leq \frac18c_*\delta_3,\quad 
	\|(\mathring{R}_{0}^u,\mathring{R}_{0}^B ) \|_{  C_{[0,T'_2]\cup [S'_2,T]} \cL^{1}_x}\leq \frac18c_*\delta_4,
\end{align*}
and, take $T_1$, $S_1$ and $\ell'_0>0$ such that 
\begin{align*}
    T'_2< T_1:=T'_1-\ell'_0<T'_1,\quad S'_1< S_1:=S'_1+\ell'_0<S'_2 \quad \ell'_0\leq \min\{\frac{T'_1-\ell'_0-T'_2}{20},\frac{S'_2-\ell'_0-S'_1}{20} \}.
\end{align*}
One has
\begin{align*} 
	\|(\mathring{R}_{0}^u,\mathring{R}_{0}^B )\|_{   C_{[0,T_1+\ell'_0]\cup [ S_1-\ell'_0,T]} \cL^{1}_x}\leq \frac18c_*\delta_3.
\end{align*}

Note that $T'_1$ and $T'_2$ are chosen in the backward manner, i.e., $T'_2<T'_1$, 
while $S'_1$ and $S'_2$ are chosen in the 
forward manner. 
Such times do exist due to the vanishing properties of Reynolds and magnetic stresses 
at times $0$ and $T$. 
Moreover, $T_1$ and $S_1$ 
are the times slightly modified by $\ell'_0$, 
which is a very small parameter used in the following mollification procedure 
that will not affect the main iterative orders.

Next, we choose $T'_3\in (0,T'_2/2)$ and $S'_3\in (\frac{T+S'_2}{2},T)$ such that 
\begin{align*}
	\|(\mathring{R}_{0}^u,\mathring{R}_{0}^B )\|_{   C_{[0,T'_3]\cup [S'_3,T]} \cL^{1}_x}\leq \frac18 c_*\delta_5,
\end{align*}
and take the modified times $T_2$, $S_2$ and the small parameter $\ell'_1>0$ such that
\begin{align*}
    T'_3< T_2:=T'_2-\ell'_1<T'_2,\quad S'_2< S_2:=S'_2+\ell'_1<S'_3 \quad \ell'_1\leq \min\{\frac{T'_2-\ell'_1-T'_3}{20},\frac{S'_3-\ell'_1-S'_2}{20} \}.
\end{align*}
Proceeding iteratively, for any $q\geq 3$, we choose $T'_{q+1} \in (0,T'_{q}/2)$ and 
$S'_{q+1} \in (\frac{T+S'_q}{2},T)$ such that
\begin{align*}
	\|(\mathring{R}_{0}^u,\mathring{R}_{0}^B )\|_{ C_{[0,T'_{q+1}]\cup [S'_{q+1},T]} \cL^{1}_x}\leq \frac18 c_*\delta_{q+3},
\end{align*}
and $T_q$, $S_q$ and $\ell'_{q-1}>0$ so that $$T'_{q+1}< T_q:=T'_q-\ell'_{q-1}<T'_q,\ \ 
S'_q< S_q:=S'_q+\ell'_{q-1}<S'_{q+1}$$ 
and  
$$\ell'_{q-1}\leq \min\{\frac{T'_q-\ell'_{q-1}-T'_{q+1}}{20},\ \ \frac{S'_{q+1}-\ell'_{q-1}-S'_q}{20} \}.$$

As a consequence, 
we obtain a backward time sequence $\{T_q\}$, a forward time sequence $\{S_q\}$ and a sequence of small parameters $\{\ell'_q\}$, such that
\begin{align}  \label{ell'}
	\ell'_{q-1}\leq \min\{\frac{T_{q}-T_{q+1}}{20},\frac{S_{q+1}-S_{q}}{20} \}
\end{align}
and
\begin{align}\label{asp-q}
	\|(\mathring{R}_{0}^u,\mathring{R}_{0}^B )\|_{  C_{[0,T_{q}+\ell'_{q-1}]\cup [S_{q}-\ell'_{q-1},T]} \cL^{1}_x}\leq \frac18 c_*\delta_{q+2}.
\end{align}

Now, let $a$ be a large integer to be chosen later,
$b\in 16\mathbb{N}$ a large integer multiple of $16$,
and $\varepsilon\in \mathbb{Q}_+$ is sufficiently small such that 
\begin{align}\label{b-beta-ve} 
	\varepsilon<10^{-3},\quad b> 100 \varepsilon^{-1}\quad \text{and} \quad b\ve\in \mathbb{N}.
\end{align}
For each $q\in \mathbb{N}$, the frequency parameter $\lbb_q$ is chosen as follows \begin{align*}
	&\lambda_q^{\frac18}\in \mathbb{N},
    \quad 
    \lambda_0\geq\max \left\{ a,\, T_2^{-4},\,\Lambda^{10}(S_2),\,  (\ell'_0)^{-\frac{1}{30}}, (\ell'_1)^{-\frac{1}{30}}  \right\}, \\
	&\lambda_q\geq \max \left\{ \lambda_{q-1}^b,\, T_{q+2}^{-4}, \,\Lambda^{10}(S_{q+2}),\, (\ell'_q)^{-\frac{1}{30}}, (\ell'_{q+1})^{-\frac{1}{30}}  \right\},
\end{align*}  
and the mollification parameter $\ell_q$ is given by
\begin{align}\label{def-ell}
	\ell_q:= \lambda_q^{-30}. 
\end{align} 
One has that, 
\begin{align}\label{a-big-lambda}
	\ell_{q}\leq \min\{ \frac{T_{q+1}-T_{q+2}}{20},\frac{T_{q+2}-T_{q+3}}{20},\frac{S_{q+2}-S_{q+1}}{20},\frac{S_{q+3}-S_{q+2}}{20}\}
    \quad \text{and}\quad 1+T_{q+2}^{-2}\leq \lambda_q,\quad \forall q\in \mathbb{N}. 
\end{align}

\subsection{Selection of energy and cross helicity profiles} 
For simplicity, 
we set
\begin{align*}
    \underline{T}_q:=T_q-\ell_{q-1},\ \ \overline{T}_q:=T_q+\ell_{q-1},\   \  \underline{S}_q:=S_q-\ell_{q-1},\ \ \overline{S}_q:=S_q+\ell_{q-1}.
\end{align*}

We choose a smooth energy profile $e$, a cross helicity profile $h$ and  the parameters  $\delta_1$ large enough 
such that
\begin{align}\label{asp-e-0}
	\begin{cases}
		\displaystyle\frac34\delta_{2}\leq  e(t)-(\| u_0 (t)\|_{L^2_x}^2+ \| B_0 (t)\|_{L^2_x}^2) \leq \delta_{1},\\[2mm]
		\displaystyle \frac{1}{64}\delta_{3} \leq  h(t)- \int_{\T^3}  u_0(t)   \cdot B_0(t)  \,\d x   \leq \frac{1}{32}\delta_{2} ,
	\end{cases}\quad \text{for }\ t\in [\underline{T_1},\os_1],
\end{align}
and for $q\in \mathbb{N}_+$
\begin{align}\label{asp-e-q}
	\begin{cases}
\displaystyle\frac34\delta_{q+2}\leq  e(t)-(\| u_0 (t)\|_{L^2_x}^2+\|  B_0 (t)\|_{L^2_x}^2)\leq \delta_{q+1},\\[2mm]
		\displaystyle \frac{1}{64}\delta_{q+3} \leq h(t)-\int_{\T^3}  u_0(t) \cdot B_0(t)  \,\d x   \leq \frac{1}{32}\delta_{q+2},
	\end{cases}\quad\text{for }\ t\in [\underline{T}_{q+1},\overline{T}_{q}]\cup [\us_{q},\os_{q+1} ] ,
\end{align}
and
\begin{align}\label{bdd-e-h-c1}
    \| e\|_{C_{[T_{q+2},S_{q+2}]}^1}+\| h\|_{C_{[T_{q+2},S_{q+2}]}^1}\lesssim \lambda_q^4.
\end{align}

We note that the orders of the upper and lower bounds of cross helicity error are smaller than those of the energy error in the above iteration. 
This is important to ensure that the energy variations caused by the cross helicity correctors will not violate the energy iterative estimate \eqref{energy-est} below.

Moreover, the upper and lower orders are also different in the above iterative estimates of the energy and cross helicity. 
The open room between different upper and lower orders indeed allows us to construct 
infinitely many energy and cross helicity 
with continuous and even decreasing profiles.

\subsection{Main iterative estimates}

We are now ready to state the following main iteration result, 
which 
quantizes  the inductive estimates of the relaxed solutions $(u_q, B_q, \mathring{R}^u_q, \mathring{R}^B_q)$ to \eqref{mhd1},
and is the heart of the proof of Theorem \ref{Thm-Non-MHD}.

For every $q \in \mathbb{N}$,
we assume the following inductive estimates at level $q$:
\begin{align}
	&  \|(u_{q},B_{q})\|_{C_T\cL^2_x}\lesssim  \sum\limits_{n=0}^q \delta_{n}^{\frac12},\label{ubl2}\\
	&  \|(u_{q},B_{q}) \|_{\cC_{[T_{q+2},S_{q+2}],x}^{1}} \lesssim  \lambda_{q}^{4},\label{ubc} \\
	& 
   \|(\mathring{R}_{q}^{u},\mathring{R}_{q}^{B})\|_{ C_T\cL^{1}_x} \leq  c_*\delta_{q+ 2},   \label{rubl1s}\\
	&  \| (\mathring{R}_{q}^{u},\mathring{R}_{q}^{B}) \|_{ C^{1}_{[T_{q+2},S_{q+2}]}\cL^1_x}+ \| (\mathring{R}_{q}^{u},\mathring{R}_{q}^{B}) \|_{C_{[T_{q+2},S_{q+2}]}\cW^{1,1}_x} \lesssim \lambda_{q}^8,  \label{rbl1b-s}\\
    &\frac12\delta_{q+2} \leq e(t)-(\| u_q (t)\|_{L^2_x}^2+\| B_q (t)\|_{L^2_x}^2)\leq 2\delta_{q+1},\ \ t\in [T_{q+1},S_{q+1}], \label{energy-est} \\ 
	&\frac{1}{128}\delta_{q+3} \leq  h(t)-\int_{\T^3}  u_q(t) \cdot B_q(t)  \,\d x \leq \frac{1}{16}\delta_{q+2} , \ \ t\in [T_{q+1},S_{q+1}], \label{helicity-est} 
\end{align}
where the implicit constants are independent of $q$.

\begin{proposition} [Main iteration]\label{Prop-Iterat}
	There exists $a \in \mathbb{N}$ large enough such that the following holds:
	
	Suppose that for some $q\in \mathbb{N}$, $(\u,\h,\ru,\rb)$ solves  \eqref{mhd1}
	and satisfies  \eqref{ubl2}-\eqref{energy-est}.
	Then, there exists $(u_{q+1}, B_{q+1}, \mathring{R}^u_{q+1}$, $\mathring{R}^B_{q+1})$
	solving \eqref{mhd1} and satisfying \eqref{ubl2}-\eqref{energy-est} with $q+1$ replacing $q$.
	In addition, one has 
	\begin{align}\label{u-B-L2tx-conv}
	\|u_{q+1}-u_{q}, B_{q+1}-B_{q}\|_{C_T\cL^{2}_{x}} \lesssim	\delta_{q+ 1}^{\frac12}.
	\end{align}
	The implicit constants in the above estimates are independent of $q$.
\end{proposition}

The iterative procedure in  Proposition \ref{Prop-Iterat} 
is of backward-forward feature, 
as it contains both the backward sequence $\{T_q\}$ and the forward sequence $\{S_q\}$. 
Hence, it is very different from the 
backward iteration in \cite{czz24}. 
The proof of Proposition \ref{Prop-Iterat} occupies Sections \ref{Sec-Interm-Flow}, 
\ref{Sec-Energy} and \ref{Sec-Rey-mag-stress} and constitutes the most technical part of the present work.

\section{The $\Lambda$-MHD system}  \label{Sec-Lambda-NSE}

This section is devoted to the proof of Theorem~\ref{thm-fns} for the $\Lambda$-MHD system.  
We shall use $(u,B)$ instead of $(\wt u, \wt B)$ to simplify the notations.

\medskip
\paragraph{\bf Step $1$: Global existence.} 
We first apply the classical Galerkin method to prove the existence of a global solution $(u,B)$ to \eqref{equa-mhd-2}. More precisely, let $\{e_j\}_{j=1}^n$ be an orthogonal basis in $L^2_\sigma$ and $\P_n: L^2 \rightarrow L^2_\sigma$ be the projection operator defined by
\begin{align}\label{def-pn}
	\P_n u = \sum_{i=1}^{n} (u, e_j)e_j,
\end{align}
where $(\cdot, \cdot)$ denotes the inner product in $L_x^2$.

Now, we consider the Galerkin approximation of system
\begin{align}\label{equa-nse-g}
\begin{cases}
\p_t u_n-\nu_1  \Delta  u_n + \P_n\P_{<\Lambda(t)}  (\P_{<\Lambda(t)} u_n\cdot\nabla \P_{<\Lambda(t)}u_n-\P_{<\Lambda(t)} B_n\cdot\nabla \P_{<\Lambda(t)}B_n)=0,\\
\p_t B_n-\nu_2  \Delta  B_n + \P_n\P_{<\Lambda(t)}  (\P_{<\Lambda(t)} u_n\cdot\nabla \P_{<\Lambda(t)}B_n-\P_{<\Lambda(t)} B_n\cdot\nabla \P_{<\Lambda(t)}u_n)=0,\\
\div   u_n=0,\ \ \div B_n=0,\\
u_n(0)=\P_n v_0,\ \ B_n(0)=\P_n H_0.
\end{cases}
\end{align}

In order to solve \eqref{equa-nse-g}, we consider the solution $(u_n, B_n)$ of the form
\begin{align}\label{def-un}
	u_n(t,x)=\sum_{j=1}^{n} \theta_j^n(t) e_j(x),\quad B_n(t,x)=\sum_{j=1}^{n} \eta_j^n(t) e_j(x),
\end{align}
where $\{(\theta_j^n,\eta_j^n) \}$ are scalar functions to be determined later. Taking the $L_x^2$-inner product of \eqref{equa-nse-g} with $e_k$, $k=1,2,\dots,n$, we get the system of ODEs
\begin{align}\label{def-ode}
\begin{cases}
\displaystyle\frac{\d\theta_k^n }{\d t}=-\lambda_k \theta_k^n - \sum_{i, j=1}^n ( \P_{<\Lambda(t)} (\P_{<\Lambda(t)}e_i \cdot \nabla  \P_{<\Lambda(t)}e_j) , e_k) (\theta_i^n\theta_j^n-\eta_i^n\eta_j^n),\vspace{1mm}\\
\displaystyle\frac{\d\eta_k^n }{\d t}=-\lambda_k \eta_k^n - \sum_{i, j=1}^n ( \P_{<\Lambda(t)} (\P_{<\Lambda(t)}e_i \cdot \nabla  \P_{<\Lambda(t)}e_j) , e_k) (\theta_i^n\eta_j^n-\eta_i^n\theta_j^n).
\end{cases}
\end{align}
By the classical theory of ODEs, there exists a unique solution $\{(\theta_j^n,\eta_j^n)\}_{j=1}^n$ to \eqref{def-ode}.

In order to prove that $(\theta_j^n, \eta_j^n)$, $j=1,2,\dots,n$, does not blow up, we take the $L_x^2$ inner product of equation \eqref{equa-nse-g} with $(u_n,B_n)$ and integrate in time  to get
\begin{align}\label{eq-e}
	\frac12 (\|u_n\|_{L^2_x}^2 +\|B_n\|_{L^2_x}^2)+ \nu_1 \int_{0}^{t}\| \nabla u_n\|_{L^2_x}^2\d s+\nu_2 \int_{0}^{t}\|\nabla B_n\|_{L^2_x}^2\d s = \frac12 (\|v_0\|_{L^2_x}^2+\|H_0\|_{L^2_x}^2)<\infty
\end{align} 
for any $t \geq 0$.
Since
\begin{align*}	\sum_{j=1}^n\left|\theta_j^n(t)\right|^2=\left\|u_n(t)\right\|_{L^2_x}^2,\quad \sum_{j=1}^n\left|\eta_j^n(t)\right|^2=\left\|B_n(t)\right\|_{L^2_x}^2,
\end{align*}
it follows that $\{(\theta_j^n, \eta_j^n)\}_{j=1}^n$ is uniformly bounded on $[0,\wt T]$ and so the solution $\{(\theta_j^n, \eta_j^n)\}_{j=1}^n$ (or equivalently $(u_n, B_n)$) exists on $[0, \infty)$.

\medskip 
Next, we pass to the limit as $n \to \infty$. 
For this purpose, 
we take any test functions $\varphi\in V:= L^2_\sigma \cap H_x^1$ and take the $L_x^2$ inner product of \eqref{equa-nse-g} with $\varphi$ to get
\begin{align*}
	\begin{cases}
	( \partial_t u_n, \varphi)= (\nu_1  \Delta u_n, \varphi)-(\P_n\P_{<\Lambda(t)}  ( \P_{<\Lambda(t)} u_n\cdot\nabla \P_{<\Lambda(t)}u_n),\varphi)+(\P_n\P_{<\Lambda(t)}  ( \P_{<\Lambda(t)} B_n\cdot\nabla \P_{<\Lambda(t)}B_n),\varphi),\vspace{1mm}\\
	( \partial_t B_n, \varphi)= (\nu_2  \Delta B_n, \varphi)-(\P_n\P_{<\Lambda(t)}  ( \P_{<\Lambda(t)} u_n\cdot\nabla \P_{<\Lambda(t)}B_n),\varphi)+(\P_n\P_{<\Lambda(t)}  ( \P_{<\Lambda(t)} B_n\cdot\nabla \P_{<\Lambda(t)}u_n),\varphi).
\end{cases}
\end{align*}
By the integration-by-part formula and the H\"{o}lder inequality, we derive
\begin{align*}
	|(\nu_1 \Delta u_n+ \nu_2 \Delta B_n, \varphi)|
	\leq  \| (\nu_1\nabla u_n,\nu_2\nabla B_n)\|_{\cL^2_x} \|\varphi\|_{V}.
\end{align*}
Using the interpolation and the Sobolev embedding $\dot H^1_x\hookrightarrow L^6_x$, we have
\begin{align*}
	|(\P_n\P_{<\Lambda(t)}  (\P_{<\Lambda(t)}  u_n\cdot\nabla \P_{<\Lambda(t)}u_n),\varphi)|
	&\leq \|\P_{<\Lambda(t)} u_n\|_{L^3_x}\|\nabla \P_{<\Lambda(t)} u_n\|_{L^2_x} \|\P_{<\Lambda(t)}\P_n\varphi\|_{L^6_x}\notag\\
	&\lesssim  \left\|u_n\right\|_{L^2_x}^{\frac12}\left\|u_n\right\|_{L^6_x}^{\frac12}\left\|\nabla u_n\right\|_{L^2_x} \left\| \varphi\right\|_{V} \\
	& \lesssim \left\|u_n\right\|_{L^2_x}^{\frac12}\left\|\nabla u_n\right\|^{\frac32}_{L^2_x }\|\varphi\|_{V},
\end{align*}
and
\begin{align*}
	|(\P_n\P_{<\Lambda(t)}  (\P_{<\Lambda(t)}  u_n\cdot\nabla \P_{<\Lambda(t)}B_n),\varphi)|
&\lesssim \|\P_{<\Lambda(t)} u_n\|_{L^4_x} \|\P_{<\Lambda(t)}  B_n\|_{L^4_x }\|\nabla \P_{<\Lambda(t)}\P_n\varphi\|_{L^2}\\
&\lesssim \|u_n\|_{L^2_x}^{\frac14}\|\nabla u_n\|_{L^2_x}^{\frac34} \| B_n\|_{L^2_x }^{\frac14}\|\nabla B_n\|_{L^2_x}^{\frac34}\|\varphi\|_{V}\\
&\lesssim  \|u_n\|_{L^2_x}^{\frac14}\| B_n\|_{L^2_x }^{\frac14}( \|\nabla u_n\|_{L^2_x}^{\frac32}+  \|\nabla B_n\|_{L^2_x}^{\frac32})\|\varphi\|_{V},
\end{align*}
which imply that
\begin{align*}
    &\|\P_n\P_{<\Lambda(t)}  (\P_{<\Lambda(t)}  u_n\cdot\nabla \P_{<\Lambda(t)}u_n)\|_{V^*}+\|\P_n\P_{<\Lambda(t)}  (\P_{<\Lambda(t)}  u_n\cdot\nabla \P_{<\Lambda(t)}B_n)\|_{V^*} \notag \\ 
    \lesssim& 
    \|(u_n, B_n)\|_{\cL^2_x}^{\frac12} \|(\nabla u_n, \nabla B_n)\|^{\frac32}_{\cL^2_x }, 
\end{align*} 
where $V^*$ is the dual space of $V$. 
Similar estimate holds for the magnetic  nonlinearity as well.
Thus, for any $\wt T > 0$  we obtain
\begin{align}
	\int_{0}^{\wt T}	\|(\p_t u_n, \p_t B_n)\|_{V^*}^{\frac43}\d t &\lesssim \int_{0}^{\wt T} \| (\nabla u_n, \nabla B_n)\|_{\cL^2_x}^{\frac43}\d t+  \int_{0}^{\wt T}\|(u_n, B_n)\|_{\cL^2_x}^{\frac23} \|(\nabla u_n, \nabla B_n)\|^{2}_{\cL^2_x }\d t\notag\\
	& \lesssim \wt T^{\frac13}\left\|(u_n, B_n)\right\|_{L^2(0, \wt T; V)}^{\frac43}+ \|(u_n, B_n)\|_{L^{\infty}(0, \wt T ; \cL^2_x)}^{\frac23} \|(u_n, B_n) \|_{L^2(0, \wt T ; V)}^2\notag \\
	& \lesssim \wt T^{\frac13} \|(v_0, H_0) \|_{\cL^2_x}^{\frac43}+  \|(v_0, H_0) \|_{\cL^2_x}^{\frac83},
\end{align}
where the last step was due to \eqref{eq-e},
and the implicit constants are independent of $n$ and $\wt T$.

In view of Aubin-Lions Lemma (see, e.g., \cite{rrs16}) and a standard diagonal argument, there exist
a weakly continuous $\cL^2_\sigma$-valued function $(u,B)$ on $[0,\infty)$ and a subsequence (still denote by $\{(u_n, B_n)\}$) such that
\begin{align}
	&(u_n, B_n) \rightarrow (u, B)\quad \text { strongly in}\quad L^2(0, \wt T ; \cL^2_\sigma), \notag\\
	&(u_n, B_n) \stackrel{w^*}{\rightharpoonup} (u, B) \quad\text { weakly-}* \text {in}\quad L^{\infty}(0, \wt T ; \cL^2_\sigma), \label{conv} \\
	&(\nabla u_n, \nabla B_n) \rightharpoonup (\nabla u, \nabla B) \quad\text { weakly in }\quad L^2\left(0, \wt T ; \cL^2_\sigma \right),\notag
\end{align}
for any $\wt T > 0$. Arguing in an analogous way as in the proof of \cite[Theorem 1.2]{czz24}, one has that for any $\varphi\in C_0^\infty([0,\wt T)\times \T^3 )$,
\begin{align*}
&\int_0^{\wt T} (\P_{<\Lambda(t)}  (\P_{<\Lambda(t)}  u_n\cdot\nabla \P_{<\Lambda(t)}B_n),\varphi) \d t \rightarrow  \int_0^{\wt T} (\P_{<\Lambda(t)}  (\P_{<\Lambda(t)}  u \cdot\nabla \P_{<\Lambda(t)}B ),\varphi) \d t,
\end{align*}
and similarly for the remaining nonlinearities. This yields that the limit $(u,B)$ is a weak (distributional) solution to the $\Lambda$-MHD system \eqref{equa-mhd-2} on $[0,\infty)$. 
Moreover, the energy equality for the Galerkin approximation \eqref{eq-e} yields the energy inequality for $(u,B)$:
\begin{equation} \label{eq:eqn-inequality-for-LMHD}
\frac12  \|(\wt u, \wt B)(t)\|_{\cL^2_x}^2  + \int_{t_0}^{t} \left(\nu_1 \| \nabla \wt u(s)\|_{L^2_x}^2+\nu_2 \|\nabla \wt B (s)\|_{L^2_x}^2 \right) \d s \leq \frac12  \|(\wt u, \wt B)(t_0)\|_{\cL^2_x}^2, 
\end{equation}
for all $t \geq t_0$, a.a. $t_0\geq 0$.

\medskip
\paragraph{\bf Step $2$: Spatial-temporal regularity on $(0,T)$.}
Let us first treat the spatial regularity of the weak solution $(u,B)$ before time $t=T$. 

\medskip 
\paragraph{\bf Spatial regularity on $(0,T)$:} 
For every $\xi\in \mathbb{Z}^3$, the Fourier coefficient of $(u,B)$ satisfies
\begin{align}\label{u-xi}
	\widehat{u}(t,\xi)- \widehat{u}(0,\xi)&=\int_{0}^t -\nu |\xi|^2 \widehat{u}(\xi ) \d r-\int_{0}^t \sum_{ l+h=\xi} i \(  (\xi\cdot\widehat{u}(l) )\widehat{u}(h)  -\frac{(\xi\cdot\widehat{u}(l) )\widehat{u}(h) \cdot \xi}{|\xi|^2} \xi \)\varphi_{\Lambda}(\xi)\varphi_{\Lambda}(l)\varphi_{\Lambda}(h) \d r\notag\\
	&\quad +\int_{0}^t \sum_{ l+h=\xi} i \(  (\xi\cdot\widehat{B}(l) )\widehat{B}(h)  -  \frac{(\xi\cdot\widehat{B}(l) )\widehat{B}(h) \cdot \xi}{|\xi|^2} \xi \)\varphi_{\Lambda}(\xi)\varphi_{\Lambda}(l)\varphi_{\Lambda}(h) \d r,
\end{align}
and
\begin{align}\label{B-xi}
	\widehat{B}(t,\xi)- \widehat{B}(0,\xi)=\int_{0}^t -\nu |\xi|^2 \widehat{B}(\xi) \d r-\int_{0}^t \sum_{ l+h=\xi}
	i \(  (\xi\cdot\widehat{u}(l) )\widehat{B}(h ) - (\xi\cdot\widehat{B}(l) )\widehat{u}(h)  \)\varphi_{\Lambda}(\xi)\varphi_{\Lambda}(l)\varphi_{\Lambda}(h)\d r.
\end{align}

Let us first consider the case where $t\in (0,T/2]$. Let $k_0:= \Lambda(t)$. Then, for any $|\xi|>2k_0$ and $\tau\geq t$ one has
\begin{align*}
	\widehat{u}(\tau,\xi)- \widehat{u}(t,\xi)=-\int_{t }^{\tau} \nu |\xi|^2 \widehat{u}(r,\xi) \d r,\quad \widehat{B}(\tau,\xi)- \widehat{B}(t,\xi)=-\int_{t }^{\tau} \nu |\xi|^2 \widehat{B}(r,\xi) \d r,
\end{align*}
which yields
\begin{align}\label{u-xi-k0}
	\widehat{ u}(\tau,\xi) =\widehat{ u}(t,\xi)e^{-\nu |\xi|^2 (\tau-t)},\ \ \widehat{ B}_\xi(\tau,\xi) =\widehat{ B}_\xi(t_,\xi)e^{-\nu |\xi|^2 (\tau-t)}, \ \ { \forall\ |\xi|>2k_0,\ \tau\geq t.}
\end{align}
Moreover, for any $s\geq 0$ and $\tau\geq \frac32 t$, by \eqref{u-xi-k0} and heat kernel estimates we derive
\begin{align}\label{est-u-hs-1}
\|u(\tau)\|_{\dot{H}^s_x}^2
&\leq  \sum_{|\xi|\leq  2k_0} \left||\xi|^{s} \widehat{u} (\tau,\xi)\right|^2+ \sum_{|\xi|>2k_0 } \left| |\xi|^{s} \widehat{u} (\tau,\xi)\right|^2\notag \\
&\leq \sum_{|\xi|\leq  2k_0} \left||\xi|^{s} \widehat{u} (\tau,\xi)\right|^2+ \sum_{|\xi|>2k_0 } | |\xi|^{s} e^{-\nu |\xi|^2 (\tau-t)} \widehat{ u} (t,\xi)|^2\notag \\
&\lesssim ( \Lambda^{2s}(t)+ t^{-s})\|v_0\|_{L^2_x}^2.
\end{align}
In particular, we have
\begin{align*}
	\|u(\tau)\|_{\dot{H}^s_x}^2 \lesssim (\Lambda^{2s}(\frac23 \tau)+(\frac23 \tau)^{-s})\|(v_0,H_0)\|_{\cL_x^2}\lesssim ( 1+ \tau^{-s}) \|(v_0,H_0)\|_{\cL_x^2}^2.
\end{align*}
Since $t$ and thus $\tau$ are arbitrary in $(0,T)$, we get
\begin{equation} \label{spa-regu}
	\|u\|_{C([t,T/2];\dot{H}^s_x) } \lesssim  1+ t^{-\frac{s}{2}}  < \infty.
\end{equation}
Similar arguments also yield
\begin{equation} \label{spa-regb}
	\|B\|_{C([t,T/2];\dot{H}^s_x) } \lesssim  1+t^{-\frac{s}{2}} < \infty.
\end{equation}

Regarding the case where $t\in [T/2,T)$, for any $|\xi|>2k_0$ we have
\begin{align}\label{u-xi-k0-2}
	\widehat{ u}(t,\xi) =\widehat{ u}(T/2,\xi)e^{-\nu |\xi|^2 (t-T/2)},\ \ \widehat{ B}_\xi(t,\xi) =\widehat{ B}_\xi(T/2,\xi)e^{-\nu |\xi|^2 (t-T/2)}.
\end{align}
Then, for any $s\geq 0$, 
\begin{align}
    \|u(t)\|_{\dot{H}^s_x}^2
&\leq \sum_{|\xi|\leq  2k_0} \left||\xi|^{s} \widehat{u} (t,\xi)\right|^2+ \sum_{|\xi|>2k_0 } | |\xi|^{s} e^{-\nu |\xi|^2 (t-T/2)} \widehat{ u} (T/2,\xi)|^2\notag \\
&\lesssim \sum_{|\xi|\leq  2k_0} \left||\xi|^{s} \widehat{u} (t,\xi)\right|^2 + \|u (T/2)\|_{\dot{H}^s_x}^2\notag \\
&\lesssim \Lambda^{2s}(t)+\Lambda^{2s}(T/2)+T^{-s},
\end{align}
where in the last step we also used \eqref{est-u-hs-1}.
It follows that 
\begin{equation} \label{spa-regu-2}
	\|u\|_{C([T/2,t];\dot{H}^s_x) } \lesssim  1+ \Lambda^{s}(t)   < \infty.
\end{equation}
Similarly, we have
\begin{equation} \label{spa-regb-2}
\|B\|_{C([T/2,t];\dot{H}^s_x) } \lesssim   1+ \Lambda^{s}(t)  < \infty.
\end{equation}
Here, the implicit constants depend on $\|(v_0,H_0)\|_{\cL_x^2}$.

\medskip 
\paragraph{\bf Temporal regularity on $(0,T)$:} 
For the temporal regularity of $(u,B)$, we begin with the case where $t\in (0,T/2]$. 
By system \eqref{equa-mhd-2}, 
\begin{align}\label{tem-regu-1}
	\|\p_t u\|_{C([t,T/2];\dot{H}^s_x)}
	&\lesssim \| u \|_{C([t,T/2];\dot{H}^{s+2}_x)}+ \| \P_{<\Lambda(t)}u \|_{C([t,T/2];{H}^{s+3}_x)}^2+ \| \P_{<\Lambda(t)} B \|_{C([t,T/2];{H}^{s+3}_x)}^2\notag\\
	&\lesssim  1+t^{-\frac{s+2}{2}}  +  \Lambda^{2s+6}(t) <+\infty, 
\end{align}
and 
\begin{align}\label{tem-regb-1}
	\|\p_t B\|_{C([t,T/2];\dot{H}^s_x)}\lesssim 1+t^{-\frac{s+2}{2}}  +  \Lambda^{2s+6}(t) <+\infty.
\end{align}

The higher regularity can be prove by inductive arguments. 
More precisely, assume that $\|\p_t^{m} u\|_{C_{[t,T/2]}\dot{H}^s}$ and $\|\p_t^{m} B\|_{C_{[t,T/2]}\dot{H}^s}$ are finite for all $s>0$ and $1\leq m\leq N-1$. Using equation \eqref{u-xi}  we get
\begin{align}\label{tem-regu-2}
	\|\p_t^N u\|_{C([t,T/2];\dot{H}^s_x)}
	& \leq \|\sum_{\xi \in \mathbb{Z}^3}|\xi|^{s+2 }\widehat{\p_t^{N-1} u}_\xi \|_{C([t,T/2];L^2_\xi)}\notag\\
	&\quad +\| \sum_{\xi \in \mathbb{Z}^3}\sum_{ l+h=\xi} \sum_{N_1+N_2+N_3=N-1} |\xi|^s (\xi\cdot\widehat{\p_t^{N_1}u}_l)\widehat{\p_t^{N_2}u}_h \p_t^{N_3}(\varphi_{\Lambda}(\xi)\varphi_{\Lambda}(l)\varphi_{\Lambda}(h))\|_{C([t,T/2];L^2_\xi)}\notag\\
	&\quad +\| \sum_{\xi \in \mathbb{Z}^3}\sum_{ l+h=\xi} \sum_{N_1+N_2+N_3=N-1} |\xi|^s (\xi\cdot\widehat{\p_t^{N_1}B}_l)\widehat{\p_t^{N_2}B}_h \p_t^{N_3}(\varphi_{\Lambda}(\xi)\varphi_{\Lambda}(l)\varphi_{\Lambda}(h))\|_{C([t,T/2];L^2_\xi)}\notag\\
	&:= J_1+J_2+J_3.
\end{align}
Note that $J_1<+ \infty$, due to the inductive assumption. Moreover, by the Sobolev embedding $H^2_x\hookrightarrow L^\infty_x$, 
\begin{align}
	J_2& \leq \sum_{N_1+N_2+N_3=N-1} \left\| \sum_{\xi \in \mathbb{Z}^3}\sum_{ l+h=\xi}  |\xi|^s (\xi\cdot\widehat{\p_t^{N_1}u}_l)\widehat{\p_t^{N_2}u}_h \right\|_{C([t,T/2];L^2_\xi)}\sup_{l+h=\xi, \xi \in \mathbb{Z}^3}\|\p_t^{N_3}(\varphi_{\Lambda}(\xi)\varphi_{\Lambda}(l)\varphi_{\Lambda}(h))\|_{C([t,T/2])}\notag\\
	&\leq  \sum_{N_1+N_2+N_3=N-1}  \|  \p_t^{N_1}u \|_{C([t,T/2];H^{s+2}_x)} \|  \p_t^{N_2}u \|_{C([t,T/2];H ^{s+3}_x)}\sup_{l+h=\xi, \xi \in \mathbb{Z}^3}\|\p_t^{N_3}(\varphi_{\Lambda}(\xi)\varphi_{\Lambda}(l)\varphi_{\Lambda}(h))\|_{C([t,T/2])},
\end{align}
and
\begin{align}\label{tem-regu-j3}
	J_3 &\leq  \sum_{N_1+N_2+N_3=N-1}  \|  \p_t^{N_1}B \|_{C([t,T/2];H^{s+2}_x)} \|  \p_t^{N_2}B \|_{C([t,T];H ^{s+3}_x)}\sup_{l+h=\xi, \xi \in \mathbb{Z}^3}\|\p_t^{N_3}(\varphi_{\Lambda}(\xi)\varphi_{\Lambda}(l)\varphi_{\Lambda}(h))\|_{C([t,T/2])},
\end{align}
which are finite by the inductive assumption. Hence, we arrive at
\begin{align*}
\|\p_t^N u\|_{C([t,T/2];\dot{H}^s_x)}	<+\infty, \quad t\in (0,T/2]. 
\end{align*}
Similarly, 
for the magnetic part, 
\begin{align*}
	\|\p_t^N B\|_{C([t,T/2];\dot{H}^s_x)}	<+\infty, \quad t\in (0,T/2].
\end{align*} 

Regarding the case where $t\in (T/2,T)$, 
one also has 
\begin{align}\label{est-tem-reg-2}
   \|(\p_t u, \p_t B)\|_{C([T/2,t];\dot{\cH}^s_x)}&\lesssim \| (u,B) \|_{C([t,T/2];\dot{\cH}^{s+2}_x)}+ \| \P_{<\Lambda(t)}u \|_{C([t,T/2];{H}^{s+3}_x)}^2+ \| \P_{<\Lambda(t)} B \|_{C([t,T/2];{H}^{s+3}_x)}^2\notag\\
   &\lesssim 1+\Lambda^{2s+4}(t) +\Lambda^{2s+6}(t) <+\infty,
\end{align}
and arguing in a similar manner as \eqref{tem-regu-2}-\eqref{tem-regu-j3}, for $N\geq 2$, we get
\begin{align*}
\|(\p_t^N u, \p_t^N B)\|_{C([T/2,t];\dot{\cH}^s_x)}	<+\infty.
\end{align*}

Thus, we conclude that $(u,B)$ is smooth in both space and time on $(0,T)$. In particular, the energy balance \eqref{eq-e-2} holds due to the smoothness of the solution.

\medskip
\paragraph{\bf  Step $3$: Continuity at initial time} By the energy balance \eqref{eq-e-2}, 
one has 
\begin{align*}
	\limsup_{t\to 0^+} \|(u(t),B(t))\|_{\cL^2_x}\leq \|(v_0,H_0)\|_{\cL^2_x}.
\end{align*}
Moreover, by the weak continuity, $(u,B)(t)\rightharpoonup (v_0,H_0)$ in $\cL^2_x$ as $t\to 0^+$, 
\begin{align*}
	\|(v_0,H_0)\|_{\cL^2_x}\leq \liminf_{t\to 0^+}\|(u,B)(t)\|_{\cL^2_x}.
\end{align*}
Hence, the norm of the solution converges
\begin{align*}
	\|(u,B)(t)\|_{\cL^2_x}\to \|(v_0,H_0)\|_{\cL^2_x},\quad \text{as}\ t\to 0^+,
\end{align*}
which along with the fact that $(v_0,H_0)$ is the weak limit of $(u(t),B(t))$ in $\cL^2_x$ yields the strong convergence
\begin{align*}
	(u(t),B(t))\to (v_0,H_0)  \quad \text { strongly in}\ \cL^2_x,\quad  \text{as}\ t\to 0^+.
\end{align*}

\medskip
\paragraph{\bf Step $4$: Eventual regularity.} We will show that the weak solution $(u,B)$ is smooth on $(t_0, \infty)$ for some $t_0 < T$, which combined with Step 2 gives smoothness on the whole time interval $(0, \infty)$.

We first claim that for any $\ve>0$, there exists time $t_0 \leq T':=  (2\ve\nu)^{-1}\|(v_0,H_0)\|_{\cL^2_x}^2$, 
such that 
\begin{align}  \label{energy-small}
    \|(\na u(t_0), \na B(t_0))\|_{\cL^2_x}^2 \leq \ve.
\end{align}
To this end, by the energy inequality \eqref{eq-e-2}, we note that 
\begin{align}\label{energy-vis}
    \nu \int_{0}^{t} \|\nabla  u(t_0), \nabla  B(t_0) \|_{\cL^2_x}^2\d s \leq \frac12  \|(v_0,H_0)\|_{\cL^2_x}^2. 
\end{align}
Suppose that \eqref{energy-small} fails, 
that is, 
\begin{align*}
    \|\nabla  u(t), \nabla  B(t) \|_{\cL^2_x}^2> \ve \quad \text{for all} 
    \ t\in [0,T'].
\end{align*}
Then, one has 
\begin{align*}
 \nu \int_{0}^{T'} \|\nabla  u(t_0), \nabla  B(t_0) \|_{\cL^2_x}^2\d s>  \nu T'\ve = \frac12 \|(v_0,H_0)\|_{\cL^2_x}^2,
\end{align*}
which contradicts the energy inequality  \eqref{energy-vis}, 
thereby proving \eqref{energy-small}, as claimed. 

\medskip 
Next, we prove the eventual regularity of the solution $(u,B)$. Consider any interval of regularity of $(u,B)$. Taking the $\dot H^1$ inner product of \eqref{equa-mhd-2} with $(u,B)$ and using the H\"older, Sobolev, Young, and Poincar\'e inequalities, we get
\begin{equation}\label{est-hi-innp}
\begin{split}
    \frac12\frac{\d}{\d t}\|(\nabla u, \nabla B)\|_{ \cL^2_x}^2&\leq \frac{c_1}{\nu^3}\|( \nabla  u , \nabla  B )\|_{ \cL^2_x}^6- c_2\nu\|(\nabla u , \nabla  B )\|_{ \cL^2_x}^2\\
    &=\frac{c_1}{\nu^3}\|( \nabla  u , \nabla  B )\|_{ \cL^2_x}^2\Big(\|(\nabla u , \nabla  B )\|_{ \cL^2_x}^4-\frac{c_2\nu^4}{c_1} \Big),
\end{split}
\end{equation}
for some absolute positive constants $c_1$ and $c_2$.
Let $\ve=\frac12( \frac{c_2}{c_1})^{\frac14} \nu$ and recall that  $T=(\frac{c_1}{c_2})^{\frac14}\nu^{-2}\|(v_0,H_0)\|_{\cL^2_x}^2$ by \eqref{def-T}. Then thanks to \eqref{energy-small}, there exists $t_0\in [0,T)$ such that 
\begin{align}
    \|\nabla (u(t_0), \nabla B(t_0))\|_{\cL^2_x}\leq \ve.
\end{align}
Then by \eqref{est-hi-innp} we have
\begin{align*} 
\frac{\d}{\d t}\|(\nabla u, \nabla B)\|_{\cL^2_x}^2<0, \qquad     \|(\nabla u , \nabla  B )(t)\|_{ \cL^2_x}^4<\frac{c_2\nu}{c_1},
\qquad {\rm for\ all}\ t \in [t_0, T).
\end{align*} 
This shows that $(u,B)$, which we already know is smooth on $(0,T)$, does not actually blow up as $t \to T^-$, and by weak continuity we also have
\begin{equation} \label{eq:ID_at_T}
\|(\nabla u , \nabla  B )(T)\|_{ \cL^2_x}^4<\frac{c_2\nu}{c_1}.
\end{equation}
In fact, higher spatial and temporal derivatives estimated in Step 2 do not blow up as well and $(u,B)(T) \in C^\infty$.

Now recall that $\Lambda(t) = \infty$ for $t>T$, and hence $(u,B)$ is a weak solution to the classical MHD system \eqref{equa-MHD} on $[T,+\infty)$. Since the initial data at $t=T$ is in $\cH^1$ due to \eqref{eq:ID_at_T}, there is a local smooth solution, which has to coincide with $(u,B)$ due to the weak-strong uniqueness that holds in the class of weak solutions satisfying the energy inequality \eqref{eq:eqn-inequality-for-LMHD}. By \eqref{est-hi-innp} 
\[
\frac{\d}{\d t}\|(\nabla u, \nabla B)\|_{\cL^2_x}^2<0,
\]
so the local smooth solution is actually global, and hence $(u,B)$ is smooth on $[T,\infty)$. In particular, it satisfies the energy equality.

Thus, we infer from Steps 1--4 that 
$(u, B)\in C([0,+\infty);  \cL^2(\T^3))$, 
it is a global smooth solution to the $\Lambda$-MHD system \eqref{mhd1}, solves the classical MHD system \eqref{equa-MHD} on $[T, +\infty)$, and satisfies the energy equality \eqref{eq-e-2}.

\medskip
\paragraph{\bf Step $5$: Vanishing of nonlinearity at initial time and $T$.}
By the strong convergence \eqref{s-con-v0} and the fact that $(u,B)$ is smooth for positive times, we derive
\begin{align}\label{u-l2-decay}
	\|\P_{ \geq \Lambda(t)}  u(t)\|_{L^2_x}&\leq \|\P_{ \geq \Lambda(t)} ( u(t) - v_0) \|_{L^2_x} +  \|\P_{ \geq \Lambda(t)} v_0\|_{L^2_x}\to 0,\quad \text{as}\  t \to 0^+\ \ \text{or}\ \ t \to T^-,
\end{align}
and
\begin{align}\label{b-l2-decay}
	\|\P_{ \geq \Lambda(t)}  B(t)\|_{L^2_x}\to 0,\quad \text{as}\  t \to 0^+\ \ \text{or}\  \ t \to T^-.
\end{align}
Note that, since
\begin{align}
	\P_{\geq \Lambda(t)}\(\P_{< \frac{\Lambda(t)}{6}} u \mathring \otimes \P_{< \frac{\Lambda(t)}{6}} B\)=0,
\end{align}
one has 
\begin{align}\label{decom-I}
	\P_{\geq \Lambda(t)}(\P_{< \Lambda(t)} u\mathring \otimes \P_{< \Lambda(t)} B) &=   \P_{\geq \Lambda(t)}\(\P_{< \frac{\Lambda(t)}{6}} u \mathring \otimes \P_{[\frac{\Lambda(t)}{6},2\Lambda(t) ]} B\)+   \P_{\geq \Lambda(t)}\(\P_{[\frac{\Lambda(t)}{6},2\Lambda(t) ]} u \mathring \otimes\P_{< \frac{\Lambda(t)}{6}} B\) \notag\\
	&\quad + \P_{\geq \Lambda(t)}\(\P_{[\frac{\Lambda(t)}{6},2\Lambda(t) ]}u \mathring \otimes \P_{[\frac{\Lambda(t)}{6},2\Lambda(t) ]}\wt B\).
\end{align}
Using \eqref{u-l2-decay}, \eqref{b-l2-decay} and \eqref{decom-I}, 
we infer that as $t\to 0^+$ or $t \to T^-$,  
\begin{align}\label{decay-uu}
	\|\P_{\geq \Lambda(t)}(\P_{< \Lambda(t)} u\mathring \otimes \P_{< \Lambda(t)} B)\|_{L^1_x} 
	&\leq \| \P_{< \frac{\Lambda(t)}{3}}( u, B) \|_{\cL^2_x} \| \P_{[\frac{\Lambda(t)}{6},2\Lambda(t) ]}( u, B)\|_{\cL^2_x} + \| \P_{[\frac{\Lambda(t)}{6},2\Lambda(t) ]}( u, B)\|_{\cL^2_x}^2\to 0.
\end{align}

The remaining nonlinearities also converge to zero as time tends to $0$ and $T$.  
Therefore,  the proof of Theorem~\ref{thm-fns} is complete.

\section{Velocity and magnetic flows} \label{Sec-Interm-Flow}

This section contains the fundamental intermittent building blocks of the velocity and magnetic flows in the telescoping convex integration scheme.
They are parameterized by six parameters $\rp$,  $\rs$, $\rw$, $\lambda$, $\mu$ and $\sigma$ below
\begin{equation}\label{larsrp}
\rp:= \lambda_{q+1}^{-\frac78-16\varepsilon},\quad
\rs := \lambda_{q+1}^{-\frac78-20\varepsilon},\quad
\rw := \lambda_{q+1}^{-\frac14-4\varepsilon},\quad
	 \lambda := \lambda_{q+1},\quad \mu:=\lambda_{q+1}^{1+22\varepsilon},
     \quad \sigma=\lambda_{q+1}^{4\varepsilon},
\end{equation}
where $\varepsilon$ is the small constant satisfying \eqref{b-beta-ve}. 
Roughly speaking, 
the parameters $\rs,\ \rp,\ \rw $ parameterize the spatial concentration of the flows, $\mu$ parameterizes the temporal oscillation, 
and $\sigma$ measures the frequency of flows.

\subsection{Intermittent building blocks.}  Let $\Phi : \mathbb{R} \to \mathbb{R}$ be a smooth cut-off function supported on
the interval $[{3}/{8}, {5}/{8}]$.
We normalize $\Phi$ such that $\phi := - \frac{d^2}{dx^2}\Phi$ satisfies
\begin{equation}\label{e4.91}
	 \int_{\mathbb{R}} \phi^2(x)\d x = 1.
\end{equation}
Moreover, let $\psi: \mathbb{R} \rightarrow \mathbb{R}$ be a smooth and mean-free function,
supported on the interval $[{3}/{8}, {5}/{8}]$, satisfying
\begin{equation}\label{e4.92}
 \int_{\mathbb{R}} \psi^{2}\left(x\right) \d x=1.
\end{equation}
The corresponding rescaled cut-off functions are defined by
\begin{equation*}
	\psi_{\rp }\left(x\right) := {\rp ^{- \frac 1  2}} \psi(\frac{x}{\rp }),\quad
	\phi_{\rs}(x) := {\rs^{-\frac{1}{2}}}\phi(\frac{x}{\rs}), \quad
		\Phi_{\rs}(x):=   {\rs^{-\frac{1}{2}}} \Phi(\frac{x}{\rs}),\quad
	\phi_{\rw}(x) := {\rw^{-\frac{1}{2}}}\phi(\frac{x}{\rw}).
\end{equation*}
Note that, $\psi_{\rp }$, $\phi_{\rs}$ and $\phi_{\rw}$
 are supported in the ball of radius $\rp $, $\rs$ and $\rw$, respectively, in $\bbr$. By an abuse of notation,
we periodize $\psi_{\rp }$, $\phi_{\rs}$, $\Phi_{\rs}$ and $\phi_{\rw}$
so that they are treated as periodic functions defined on $\mathbb{T}$.

\vspace{1ex}

We define the \textit{intermittent velocity flows} by
\begin{equation*}
	W_{(k)} :=  \psi_{\rp }(\sigma N_{\Lambda}(k_1\cdot x+\mu t))\phi_{\rs }( \sigma N_{\Lambda}k\cdot (x-\alpha_k))\phi_{\rw }( \sigma N_{\Lambda}k_2\cdot (x-\alpha_k))k_1,\quad k \in \Lambda_u \cup \Lambda_B  ,
\end{equation*}
and the \textit{intermittent magnetic flows} by
\begin{equation*}
	D_{(k)} := \psi_{\rp }(\sigma N_{\Lambda}(k_1\cdot x+\mu t))\phi_{\rs }(\sigma N_{\Lambda}k\cdot (x-\alpha_k))\phi_{\rw }( \sigma N_{\Lambda}k_2\cdot (x-\alpha_k))k_2, \quad k \in \Lambda_B .
\end{equation*}
Here, $(k,k_1,k_2) \in \Lambda_u \cup \Lambda_B$ are the orthonormal bases in $\R^3$ in Geometric Lemmas~\ref{geometric lem 1} and \ref{geometric lem 2}, $N_{\Lambda}$ is given by \eqref{NLambda} below. One can choose the shifts $\a_k\in \R^3$ suitably such that $(W_{(k)}, D_{(k)})$ and $(W_{(k')}, D_{(k')})$ have disjoint supports if $k\neq k'$.

To ease notations, we set
\begin{align*}
		&\psi_{(k_1)}(t,x) :=\psi_{\rp }(\sigma N_{\Lambda}(k_1\cdot  x+\mu t)),\\
		&\phi_{(k)}(x) := \phi_{\rs }(\sigma  N_{\Lambda}k\cdot  (x-\alpha_k)), \quad \Phi_{(k )}(x) := \Phi_{\rs }(\sigma  N_{\Lambda}k \cdot  (x-\alpha_k)),\\
		&\phi_{(k_2)}(x) := \phi_{\rw }( \sigma N_{\Lambda}k_2\cdot  (x-\alpha_k)) .
\end{align*}
Note that, $\phi_{(k)}$ and $\phi_{(k_2)}$ are independent of time, and
\begin{equation}\label{snwd}
	W_{(k)} = \psi_{(k_1)}\phi_{(k)}\phi_{(k_2)} k_1,\  k\in \Lambda_u\cup \Lambda_B,
	\quad 
    \ D_{(k)} = \psi_{(k_1)}\phi_{(k)}\phi_{(k_2)} k_2,\ k\in \Lambda_B.
\end{equation} 

Since the intermittent flow $W_{(k)}$ is not divergence-free,
one also needs the corrector
\begin{equation}
	\begin{aligned}
		\label{corrector vector}
		\wt W_{(k)}^c := \frac{\rs^2}{\sigma^2N_{ \Lambda }^2} \nabla(\psi_{(k_1)}\phi_{(k_2)} )\times\curl(\Phi_{(k)} k_1),\quad k\in \Lambda_{u}\cup\Lambda_{B} .
	\end{aligned}
\end{equation}
Straightforward computations then give
\begin{equation}\label{wcwc}
	  W_{(k)} + \wt W_{(k)}^c
	=\curl \left( \frac{\rs^2}{\sigma^2N_{ \Lambda }^2}  \psi_{(k_1)} \phi_{(k_2)} \curl( \Phi_{(k)} k_1)\right)
	=\curl  W^c_{(k)},
\end{equation}
where 
\begin{align} \label{Vk-def}
	W^c_{(k)} :=  \frac{\rs^2}{\sigma^2N_{ \Lambda }^2}  \psi_{(k_1)} \phi_{(k_2)} \curl( \Phi_{(k)} k_1),\quad k\in \Lambda_{u}\cup\Lambda_{B} .
\end{align}
Hence,
\begin{align} \label{div-Wck-Wk-0}
	\div (W_{(k)}+ \wt W^c_{(k)}) =0,\quad k\in \Lambda_{u}\cup\Lambda_{B} .
\end{align}

\medskip 
Regarding the magnetic  flows, the incompressible correctors are defined by
\begin{equation}\label{dkc}
		D_{(k)}^c:=  \frac{\rs^2}{\sigma^2N_{ \Lambda }^2}  \psi_{(k_1)} \phi_{(k_2)} \curl( \Phi_{(k)} k_2),
    \quad \wt D_{(k)}^c:= \frac{\rs^2}{\sigma^2N_{ \Lambda }^2}  \nabla(\psi_{(k_1)}\phi_{(k_2)} )\times\curl(\Phi_{(k)} k_2),\quad  k\in \Lambda_B.
\end{equation}
It holds that
\begin{align} \label{D-wtD-Dc}
   D_{(k)} + \wt D_{(k)}^c
	=\curl  D^c_{(k)}, \quad k\in \Lambda_{B},
\end{align}
and so
\begin{align*}
\div ( D_{(k)} + \wt D_{(k)}^c)=0.
\end{align*}

Lemma \ref{buildingblockestlemma} below contains the useful estimates of the intermittent
velocity and magnetic flows.
The proof is similar to that of  \cite[Lemma~3.3]{lzz21} 
and so is omitted here for simplicity.

\begin{lemma} [Estimates of spatial intermittency] \label{buildingblockestlemma}
	For $p \in [1,+\infty]$, $N,\,M \in \mathbb{N}$, we have
\begin{align}
&\left\|\nabla^{N} \partial_{t}^{M} \psi_{(k_1)}\right\|_{C_T L^{p}_x }
\lesssim \sigma^N\rp ^{\frac 1p- \frac 12-N} ( \sigma\rp ^{-1} \mu)^{M}, \label{est-psi} \\
&\left\|\nabla^{N} \phi_{(k)}\right\|_{L^{p}_x }+\left\|\nabla^{N} \Phi_{(k)}\right\|_{L^{p}_x }
\lesssim \sigma^N\rs ^{\frac 1p- \frac 12-N} , \label{est-phi-1}\\
&\left\|\nabla^{N} \phi_{(k_2)}\right\|_{L^{p}_x }
\lesssim \sigma^N\rw^{\frac 1p- \frac 12-N}, \label{est-phi-2}
\end{align}
where the implicit constants are independent of $\rp ,\, \rs ,\, \rw , \,\lambda$ and $\mu$.
In particular, one has 
\begin{align}
&\displaystyle\left\|\nabla^{N} \partial_{t}^{M} W_{(k)}\right\|_{C_T  L^{p}_x }   
+\frac{\rp }{\rs }\left\|\nabla^{N} \partial_{t}^{M} \wt W_{(k)}^{c}\right\|_{C_T L^{p}_x }
+\rs^{-1}\sigma\left\|\nabla^{N} \partial_{t}^{M} W_{(k)}^c\right\|_{C_T L^{p}_x }  \nonumber \\
&\quad \lesssim \sigma^N\rs ^{-N}
(\sigma \rp ^{-1} \mu)^{M}(\rp \rs \rw)^{\frac 1p- \frac 12},\quad 
k\in \Lambda_u \cup \Lambda_B, \label{ew}
\end{align}
and
\begin{align}
& \left\|\nabla^{N} \partial_{t}^{M} D_{(k)}\right\|_{C_T  L^{p}_x }
+\frac{\rp }{\rs}\left\|\nabla^{N} \partial_{t}^{M} \wt D_{(k)}^{c}\right\|_{C_T L^{p}_x }
+\rs ^{-1}\sigma\left\|\nabla^{N} \partial_{t}^{M} D_{(k)}^c\right\|_{C_T L^{p}_x }  \notag\\ &\quad\lesssim\sigma^N\rs ^{-N}
(\sigma \rp^{-1} \mu)^{M} (\rp \rs \rw)^{\frac 1p- \frac 12}, 
\quad
k\in \Lambda_B.    \label{ed}
\end{align}
\end{lemma}

\subsection{Velocity and magnetic perturbations}   \label{Sec-Pert}

Let us first mollify the Reynolds and magnetic stresses. Let $\phi_{\varepsilon}$ be  a family of standard mollifiers on $\T^3$, 
and $\varphi_{\varepsilon}$ a family of  standard mollifiers on $\T$,
supported on a ball of radius $\ve$ ($\ve >0$) centered at $0$.

We define the mollifications of $(\ru,\rb)$ in space and time  by
\begin{equation}\label{mol}
	\mathring{R}_{\ell_q}^u:= (\ru *_{x} \phi_{\ell_q}) *_{t} \varphi_{\ell_q}, 
    \quad 
    \mathring{R}_{\ell_q}^B:=(\rb *_{x} \phi_{\ell_q}) *_{t} \varphi_{\ell_q},
\end{equation}
where the mollification parameter $\ell_q$ is defined by \eqref{def-ell}.

Note that, for any $N\in \mathbb{N}_+$, 
one has 
\begin{align*}
	&  \|( \mathring{R}_{\ell_q}^u, \mathring{R}_{\ell_q}^B )(t)\|_{\cL^{1}_{ x}} \lesssim   \sup_{s\in [t-\ell_q,t]}\left\|(\mathring R^u_{q},\mathring R^B_{q})(s)\right\|_{\cL^{1}_{x}},
 \end{align*}
 and
 \begin{align*}
	&  \| (\mathring{R}_{\ell_q}^u ,\mathring{R}_{\ell_q}^B) \|_{\cC_{T,x}^N} \lesssim  \ell_q^{-4-N}\| (\mathring{R}_{\ell_q}^u ,\mathring{R}_{\ell_q}^B) \|_{C_{T}\cL^1_x} . 
\end{align*}

We next define the suitable amplitudes of the magnetic and velocity perturbations. 

\subsubsection{\bf The magnetic amplitudes.}   \label{Subsec-Amplitude}
Set
\begin{align}
	\rho_B(t,x) :=   2 \varepsilon_B^{-1} (\ell_q^2+|\mathring{R}_{\ell_q}^B(t,x) |^2)^{\frac12}
 ,\label{rhob}
\end{align}
where $\varepsilon_{B}$ is the small radius in First Geometric Lemma \ref{geometric lem 1}.
Note that
\begin{equation}\label{rhor}
	\left|  \frac{\mathring{R}_{\ell_q}^B}{\rho_B} \right|
	= \left| \frac{\mathring{R}_{\ell_q}^B}{2 \varepsilon_B^{-1} (\ell_q^2+|\mathring{R}_{\ell_q}^B |^2)^{\frac12}}\right| \leq \varepsilon_B,
\end{equation}
and
\begin{align}	\label{rhoblp}
	&\rho_B\geq \ell_q,\quad
\norm{ \rho_B(t) }_{ L^p_{x}} \leq  2\ve_B^{-1}( \ell_q + \norm{\mathring{R}_{q}^B }_{C_{[t-\ell_q,t]}L^p_{x}}),\quad p\in [1, +\infty).
\end{align}

\medskip 
The amplitudes of the magnetic perturbations are defined as follows 
\begin{equation}\label{akb}
	a_{(k)}(t,x):= a_{k, B}(t,x)
    =  \rho_B^{\frac{1}{2} } (t,x)\gamma_{(k)}
       \left(\frac{-\mathring{R}_{\ell_q}^B(t,x)}{\rho_B(t,x)}\right), \quad k \in \Lambda_B,
\end{equation}
where $\gamma_{(k)}$ is the smooth function in First Geometric Lemma~\ref{geometric lem 1}.

\medskip 
By virtue of First Geometric Lemma~\ref{geometric lem 1}
and the expression \eqref{akb},
the following algebraic  identity holds:
\begin{align}\label{magcancel}
	   \sum\limits_{ k \in  \Lambda_B} a_{(k)}^2
	( D_{(k)} \otimes W_{(k)} - W_{(k)} \otimes  D_{(k)} )
	=   -\mathring{R}_{\ell_q}^B
	+  \sum\limits_{ k \in \Lambda_B}  a_{(k)}^2\P_{\neq 0}(  D_{(k)} \otimes W_{(k)} -W_{(k)} \otimes  D_{(k)} ) ,
\end{align}
where $\P_{\neq 0}$ denotes the spatial projection onto nonzero Fourier modes.

Moreover, one has the following analytic estimates, which can be proved similarly as in \cite{lzz21} with slight modifications. 

\begin{lemma} [Estimates of magnetic amplitudes] \label{mae}
For every $N\geq 1$, $k\in \Lambda_B$, we have
	\begin{align}
		&\norm{a_{(k)}}_{C_TL^2_{x}} \leq M_*(2\ve_B^{-1})^{\frac12}(\ell_q + c_*\delta_{q+2})^{\frac12},\label{est-akbl2}\\
			&\norm{ a_{(k)} }_{C_{T,x} } \lesssim  \ell_q^{-2}, \label{est-akbc}\\
		& \norm{ a_{(k)} }_{C_{T,x}^N} \lesssim  \ell_q^{-6N-7}, \label{est-akbn}
	\end{align}
where $M_*$ is the geometric constant given by \eqref{M bound} and the implicit constants are universal.
\end{lemma}

\subsubsection{\bf Magnetic principal part of perturbations} 
We define the magnetic principal parts $w_{q+1}^{(pb)}$ and $d_{q+1}^{(p)}$ of the velocity
and  magnetic perturbations, respectively, by
\begin{align*}
		w_{q+1}^{(pb)} := \sum_{k \in \Lambda_B } a_{(k)} W_{(k)},\quad
		d_{q+1}^{(p)} := \sum_{k \in \Lambda_B} a_{(k)} D_{(k)}.
	\end{align*}
    
By the orthogonality $W_{(k)} \otimes D_{(k')} =0$, $k\not= k'$,
and the algebraic identity \eqref{magcancel}, 
one can decrease the magnetic stress by 
using the anti-symmetric nonlinearity in the magnetic equation as follows
\begin{align} \label{mag oscillation cancellation calculation}
	d_{q+ 1}^{(p)} \otimes w_{q+ 1}^{(pb)} - w_{q+1}^{(pb)}\otimes d_{q+ 1}^{(p)}\notag + \rr^B_{\ell_q}
	=& \sum_{k \in \Lambda_B} a_{(k)}^2  (D_{(k)}\otimes W_{(k)}-W_{(k)}\otimes D_{(k)}) +  \rr^B_{\ell_q}  \notag\\
	=& \sum_{k \in \Lambda_B}  a_{(k)}^2  \P_{\neq 0} (D_{(k)}\otimes W_{(k)}-W_{(k)}\otimes D_{(k)}).
\end{align}

\subsubsection{\bf The helicity corrector}
The following helicity corrector is used to decrease the gaps between the cross helicity profile $h$ and the cross helicity of the relaxed solution at level $q+1$,  
\begin{align}\label{def-he-cor}
w_{q+1}^{(h)}=d_{q+1}^{(h)}:= \sum_{k\in \Lambda_h}(\gamma^{(h)}_{\ell_q })^{\frac12} W_{(k)}.
\end{align}
Here,
\begin{align}\label{def-theta}
	\gamma_q^{(h)}(t):= &\,    h(t)- \int_{\T^3}  u_{q}(t)\cdot B_{q}(t) \,\d x-f_q^{(h)}(t),\quad \gamma^{(h)}_{\ell_q }:= \gamma_q^{(h)}*_t\varphi_{\ell_q},
\end{align}
which measures the helicity error at level $q+1$,
$\Lambda_h$ is a vector set disjoint with $\Lambda_{u}$ and $\Lambda_{B}$, and one may choose $\Lambda_h$ such that $|\Lambda_{h}|=1$, $f_q^{(h)}$ is the step function given by
\begin{align}\label{def-gq}
	f_q^{(h)}(t) :=1_{[T_{q+1},S_{q+1}]}\frac{1}{256}\delta_{q+ 3}+ 1_{[T_{q+2}+\ell_{q},T_{q+1}]}\frac{3}{256}\delta_{q+ 4}+1_{[S_{q+1},S_{q+2}]}\frac{3}{256}\delta_{q+ 4}.
\end{align}

We note that the temporal step function $f_q^{(h)}$ is used to guarantee that 
$\gamma^{(h)}_q$ is strictly large than zero.  
Moreover, it follows from the energy and cross helicity iterative estimates in \eqref{asp-e-0}, \eqref{asp-e-q} and \eqref{helicity-est} that for $N\geq 1$
\begin{align}\label{est-theta}
 \gamma^{(h)}_q\geq \frac{1}{256}\delta_{q+4},\quad \|\gamma^{(h)}_{\ell_q }\|_{C_T}\leq \frac{1}{16}\delta_{q+2} ,\quad \|\gamma^{(h)}_{\ell_q }\|_{C_T^N}\lesssim  \ell_q^{-N}.
\end{align}

\subsubsection{\bf The velocity amplitudes.}
In order to compensate the errors raised by magnetic perturbations, 
the following matrix $\mathring{G}^{B}$ is used to construct the velocity amplitudes
\begin{equation}
	\label{def:G}
	\mathring{G}^{B}: = \sum_{k \in \Lambda_B}a_{(k)}^2 \aint_{\mathbb{T}^3} W_{(k)} \otimes W_{(k)} - D_{(k)} \otimes D_{(k)}\d x.
\end{equation} 

It follows from Lemma~\ref{mae} that 
\begin{align}\label{gl1-est}
\|\mathring{G}^{B}(t)\|_{L^1_{x}} \leq  2\ve_B^{-1}M_*^2|\Lambda_{B}| (\ell_q+ \|\rr^B_q(t)\|_{L^1} ),
\end{align}
and for $N\geq 1$,
\begin{align}
&   \norm{\mathring{G}^{B}}_{C_{T,x}} \lesssim \ell_q^{-4},\quad  \norm{\mathring{G}^{B}}_{C_{T,x}^N} \lesssim \ell_q^{-6N-14}.\label{gbc}
\end{align}

Moreover, in order to maintain the initial value of the approximate solution $(u_{q+1},B_{q+1})$ at level $q+1$, we use the smooth temporal function $\chi_{q+1}$ which satisfies
\begin{align}\label{def-chiq}
	\chi_{q+1}(t)=\begin{cases}
		1,\quad t\in [T_{q+1},S_{q+1}],\\
		0,\quad t\in [0,\Tm]\cup [\Sm,T],\\
	\end{cases} 
\end{align} 
with $\Tm= (T_{q+2}+T_{q+1})/{2}$ 
and $\Sm= (S_{q+2}+S_{q+1})/{2}$, 
and for $0\leq N\leq 2$
\begin{align}\label{est-chiq}
	\|	\chi_{q+1}\|_{C^N_t}\lesssim \ell_{q}^{-N}.
\end{align}

Then, let 
\begin{equation}\label{defrho}
	\begin{aligned}
		&  \rho_u(t,x):=  2 \varepsilon_u^{-1} (\ell_q^2+|\mathring{R}_{\ell_q}^u(t,x) + \mathring{G}^{B}(t,x)|^2)^{\frac12}.
	\end{aligned}
\end{equation}
By the iterative estimates \eqref{asp-e-0}, \eqref{asp-e-q} and \eqref{rubl1s}, one has 
\begin{align}
	&\rho_u\gtrsim \ell_q, \quad \norm{ \rho_u}_{ L^p_{x}} \lesssim  \ell_q + \norm{\mathring{R}_{\ell_q}^u +  \mathring{G}^{B} }_{L^p_{x}}, \quad p\in [1, +\infty),\label{rhoulowbound}
\end{align}
and  
\begin{align}
	& \norm{ \rho_u}_{C_{T,x}} \lesssim  \ell_q^{-4}, \quad   \norm{\rho_u }_{C_{T,x}^N}  \lesssim  \ell^{-21N+1 },\quad  1\leq N\leq 4. \label{rhou-Ctx.1}
\end{align}

Moreover, let 
\begin{align*}
	\gamma^{(e)}_{q}(t):= &\, \frac13\Big( (e -f_q^{(e)} - \|u_q \|_{L^2}^2- \|B_q\|_{L^2}^2  )*_{t} \varphi_{\ell_q}(t) \notag\\
	&\quad\ \ - \int_{\T^3}  |\chi_{q+1}(t)w_{q+1}^{(pb)}(t) |^2+  |\chi_{q+1}(t)d_{q+1}^{(p)}(t) |^2+ |\chi_{q+1}(t)w_{q+1}^{(h)}(t) |^2+ |\chi_{q+1}(t)d_{q+1}^{(h)}(t) |^2\, \d x\Big),
\end{align*}
where $\varepsilon_{u}$ is the small radius in Geometric Lemma \ref{geometric lem 2} 
and 
$f_q^{(e)}$ is the step function given by 
\begin{align}\label{def-fq}
	f_q^{(e)}(t) :=1_{[T_{q+1},S_{q+1}]}\frac{1}{4}\delta_{q+ 2}+ 1_{[T_{q+2}+\ell_{q},T_{q+1}]}\frac34\delta_{q+ 3}+ 1_{[S_{q+1},S_{q+2}]}\frac34\delta_{q+ 3}.
\end{align}

\begin{remark} 
We note that $\rho_u$ is mainly to decrease the Reynolds errors,  
and $\gamma^{(e)}_q$ mainly measures the energy error between the given energy profile $e(t)$ and the old velocity and magnetic energy $\|(u_{q},B_{q})\|_{\cL^2_x}^2$ at level $q$. 
Unlike in the NSE case, 
the present construction incorporates more perturbations 
to control the magnetic errors. 
These perturbations will contribute to an increase in the energy and 
have to be considered in the formulation of $\gamma^{(e)}_q$.  
\end{remark}

The temporal step function $f^{(e)}_q$ guarantees the positivity of $\gamma^{(e)}_q$ as stated in the following lemma. 

\begin{lemma} [Positivity of $\gamma^{(e)}_{q}$] 
\label{lem-gamma} 
For all $t\geq 0$, we have
\begin{align}\label{est-gamma-e-1}
     \gamma^{(e)}_{q}(t)\geq 0,\quad \|\gamma^{(e)}_{q}\|_{C_T}\lesssim \delta_{q+ 1}.
\end{align}
Moreover, for $1\leq M\leq 4$, 
\begin{align}\label{est-ga-end}
	& \norm{ \p_t^M \gamma^{(e)}_q}_{C_{T}}\lesssim \ell_q^{-6M-14}.
\end{align}
  \end{lemma}
\begin{proof} 
Let us first treat the first inequality of \eqref{est-gamma-e-1}. To this end, by \eqref{energy-est},
\begin{align}\label{est-gamma-part1}
(e(t)-f_q^{(e)}(t) - \|u_q(t)\|_{L^2}^2- \|B_q(t) \|_{L^2}^2  )*_{t}\varphi_{\ell_q} \geq (\frac12 \delta_{q+ 2}-\frac14\delta_{q+ 2})  \geq  \frac14 \delta_{q+ 2} .
\end{align}
Moreover, applying the decorrelation  Lemma~\ref{Decorrelation1} we have
\begin{align}\label{est-ewpb}
 \norm{w^{(pb)}_{q+1}}_{C_TL^2_{x}}^2 &\leq \sum\limits_{k \in \Lambda_B}
 \Big(\|a_{(k)}\|_{C_TL^2_{x}} \norm{  \psi_{(k_1)}\phi_{(k)}\phi_{(k_2)}}_{C_TL^2 }
+C_1\sigma^{-\frac12}\|a_{(k)}\|_{C^1_{T,x}}\norm{  \psi_{(k_1)}\phi_{(k)}\phi_{(k_2)}}_{C_TL^2 } \Big)^2 \notag\\
&\leq \sum\limits_{k \in \Lambda_B} (2 \|a_{(k)}\|_{C_TL^2_{x}}^2+ 2C_1^2\sigma^{-1} \|a_{(k)}\|_{C^1_{T,x}}^2) \notag\\
&\leq 4\ve_B^{-1}M_*^2|\Lambda_{B}| (\ell_q+ c_*\delta_{q+ 2})+ C_2\ell_q^{-26} \lambda_{q+1}^{-4\ve} \notag\\
& \leq 5\ve_B^{-1}M_*^2|\Lambda_{B}|  c_*\delta_{q+ 2},
\end{align}
where we choose $a$ sufficiently large in the last step. Analogous arguments also yield that
\begin{align}\label{est-edp}
 \norm{d^{(p)}_{q+1}}_{C_TL^2_{x}}^2 \leq 5\ve_B^{-1}M_*^2|\Lambda_{B}| c_*\delta_{q+ 2}.
\end{align}
Using \eqref{est-theta} and the fact that $|\Lambda_{h}|=1$, we have
\begin{align}\label{est-ewdh}
 \norm{w^{(h)}_{q+1}}_{C_TL^2_{x}}^2+ \norm{d^{(h)}_{q+1}}_{C_TL^2_{x}}^2 \leq 2  \sum\limits_{k \in \Lambda_h} \|\gamma^{(h)}_{\ell_q }\|_{C_{T}} \norm{\psi_{(k_1)}\phi_{(k)}\phi_{(k_2)}}_{C_TL^2 }^2
\leq \frac18 \delta_{q+2}.
\end{align}
Combining \eqref{est-ewpb}-\eqref{est-ewdh} together we arrive at
\begin{align}\label{est-gamma-part2}
&\quad \int_{\T^3}  |\chi_{q+1}w_{q+1}^{(pb)} |^2+  |\chi_{q+1}d_{q+1}^{(p)} |^2+ |\chi_{q+1}w_{q+1}^{(h)} |^2+ |\chi_{q+1}d_{q+1}^{(h)} |^2\, \d x\notag\\
&\leq  \int_{\T^3}  | w_{q+1}^{(pb)} |^2+  |d_{q+1}^{(p)} |^2+ | w_{q+1}^{(h)} |^2+ |d_{q+1}^{(h)} |^2\, \d x\notag\\
&\leq 10\ve_B^{-1}M_*^2|\Lambda_{B}| c_*\delta_{q+ 2} + \frac18 \delta_{q+ 2} < \frac14\delta_{q+ 2},
\end{align}
where  the smallness condition \eqref{def-c8} was used in the last step.

Therefore, combining \eqref{est-gamma-part1} and \eqref{est-gamma-part2} together we deduce that
\begin{align*}
\gamma^{(e)}_{q}\geq 0.
\end{align*}

The second inequality of \eqref{est-gamma-e-1} can be verified easily 
by the iterative estimate \eqref{energy-est}. 

Regarding the $C^M_T$-estimates of $\gamma^{(e)}_{q}$, by the energy iterative esitmate \eqref{energy-est}, 
\begin{align}\label{est-gamma}
	& \norm{ \p_t^M \gamma^{(e)}_q}_{C_{T}} \lesssim \ell_q^{-M} +  \left\|\int_{\T^3}  \p_t^M\Big( |\chi_{q+1}w_{q+1}^{(pb)} |^2+  |\chi_{q+1}d_{q+1}^{(p)} |^2+ |\chi_{q+1}w_{q+1}^{(h)} |^2+ |\chi_{q+1}d_{q+1}^{(h)} |^2\Big)\, \d x\right\|_{C_T}.
\end{align}
The integration-by-parts formula and Lemma~\ref{mae} then yield 
\begin{align}\label{est-ga-1}
 \left\|\int_{\T^3}  \p_t^M|w_{q+1}^{(pb)}|^2\, \d x\right\|_{C_T}
&\lesssim \sum_{k\in \Lambda_{B}}  \left\|\int_{\T^3}  \p_t^M(a_{(k)}^2\P_{\neq 0}(|W_{(k)}|^2 ) )\, \d x\right\|_{C_T}+ \sum_{k\in \Lambda_{B}}  \left\|\int_{\T^3}  \p_t^M(a_{(k)}^2)\, \d x\right\|_{C_T}\notag\\
	&\lesssim \sum_{k\in \Lambda_{B}} \sum_{M_1+M_2= M} \sigma^{-j}  \|a_{(k)}^2\|_{C_{T,x}^{j+M_1}} \| \p_t^{M_2}(\P_{\neq 0}(|W_{(k)}|^2 )  )\|_{ C_TL^2_x }  +\sum_{k\in \Lambda_{B}}   \|a_{(k)}^2\|_{C_{T,x}^{M}} \notag\\
	&\lesssim \sum_{M_1+M_2= M} \ell_q^{-6(j+M_1)-14}\sigma^{-j} (\sigma\rp^{-1}\mu)^{M_2}(\rp\rs\rw)^{-\frac12}  + \ell_q^{-6M-14} \notag\\
	&\lesssim \ell_q^{-6j-14}\lambda_{q+1}^{-4j\ve+4M}  + \ell_q^{-6M-14} \lesssim \ell_q^{-6M-14},
\end{align}
where one may let $j=[10\ve^{-1}]+1$ and choose $a$ sufficiently large in the last step. Analogous arguments also yield
\begin{align}\label{est-ga-2}
\left\|\int_{\T^3}  \p_t^M|d_{q+1}^{(p)}|^2\, \d x\right\|_{C_T}+\left\|\int_{\T^3}  \p_t^M|w_{q+1}^{(h)}|^2\, \d x\right\|_{C_T}+\left\|\int_{\T^3}  \p_t^M|d_{q+1}^{(h)}|^2\, \d x\right\|_{C_T}\lesssim \ell_q^{-6M-14},
\end{align}
which along with \eqref{est-chiq} and Leibnitz's rule yields \eqref{est-ga-end}, thereby finishing the proof of Lemma~\ref{lem-gamma}.
\end{proof}

As a consequence, 
one has from Lemma \ref{lem-gamma} 
and the definition of $\rho_u$ that 
\begin{align*} 
	& \left|  \frac{\mathring{R}_{\ell_q}^u + \mathring{G}^{B}}{\rho_u+\gamma_{q}^{(e)}} \right| \leq \varepsilon_u.
\end{align*}

\medskip 
Now, we are in position to define the 
velocity amplitudes 
as follows 
\begin{equation}\label{velamp}
	\begin{aligned}
		&a_{(k)} := a_{k, u}
		= (\rho_u+\gamma^{(e)}_{q})^{\frac{1}{2}} \gamma_{(k)}
           \left(\Id - \frac{\mathring{R}_{\ell_q}^u +  \mathring{G}^{B} }{\rho_u+\gamma^{(e)}_{q}} \right),
		\quad k \in \Lambda_u.
	\end{aligned}
\end{equation}
In view of Second Geometric Lemma~\ref{geometric lem 2}
and the expression \eqref{velamp}, the following algebraic identity holds:
\begin{align}\label{velcancel}
	\sum\limits_{ k \in  \Lambda_u} a_{(k)}^2
	W_{(k)} \otimes W_{(k)}
	& = (\rho_u+\gamma^{(e)}_{q})\Id - \mathring{R}_{\ell_q}^u -  \mathring{G}^{B}
	+  \sum\limits_{ k \in \Lambda_u}  a_{(k)}^2 \P_{\neq 0}(  W_{(k)} \otimes  W_{(k)} ).
\end{align}

Moreover, using the standard H\"older estimate \eqref{holder1}
and estimates \eqref{rhoulowbound}, \eqref{rhou-Ctx.1}, \eqref{est-gamma-e-1} and \eqref{est-ga-end},
similarly to Lemma~\ref{mae},
we have the following analytic estimates of the velocity amplitude functions. The proof follows in an analogous way as in \cite{lzz21} and thus it is omitted here.
\begin{lemma} [Estimates of velocity amplitudes]  \label{vae}
	For $1\leq N\leq 4$, $k\in \Lambda_{u} $,
	one has 
\begin{align*}
	\norm{a_{(k)}}_{C_TL^2_{x}} \lesssim \delta_{q+ 1}^{\frac12},\quad 
	\norm{ a_{(k)} }_{C_{T,x} } \lesssim  \ell_q^{-2}, \quad 
	\norm{ a_{(k)} }_{C_{T,x}^N} \lesssim  \ell_q^{-33N-9}.
\end{align*}
\end{lemma}

\subsubsection{\bf Velocity principal part of perturbations}
We define
	\begin{align} \label{def-wpu}
		w_{q+1}^{(pu)} &:= \sum_{k \in \Lambda_u} a_{(k)} W_{(k)},
	\end{align}
and the velocity principal part $w_{q+1}^{(p)}$ by
	\begin{align*}
	w_{q+1}^{(p)} &:= w_{q+1}^{(pb)}+w_{q+1}^{(pu)}.
\end{align*}

Regarding the symmetric nonlinearity in the velocity equation, 
by \eqref{def:G} and  \eqref{velcancel}, one can decrease the size of the Reynolds stress as follows
\begin{align}  \label{vel oscillation cancellation calculation}
	 w_{q+ 1}^{(p)} \otimes  w_{q+ 1}^{(p)} - d_{q+ 1}^{(p)} \otimes d_{q+1}^{(p)} + \mathring{R}^u_{\ell_q} 
	=  (\rho_{u}+\gamma^{(e)}_{q} )Id+ \sum_{k \in \Lambda_u\cup\Lambda_B} a_{(k)}^2 \P_{\neq0} (W_{(k)}\otimes W_{(k)}-D_{(k)}\otimes D_{(k)}).
\end{align}
Note that, compared with the 3D NSE, there are additional high oscillation errors  arising from the magnetic perturbations. To simplify notation, we let 
\begin{align}\label{def-dk-2}
    D_{(k)}:\equiv 0\quad \text{for}\quad k\in \Lambda_u.
\end{align}

\subsubsection{\bf Incompressibility correctors}

Because
the principal perturbations and the helicity correctors are not divergence free, 
one needs the following incompressibility correctors
\begin{subequations} \label{wqc-dqc}
	\begin{align}
		w_{q+1}^{(c)}
		&:=   \sum_{k\in \Lambda_u \cup\Lambda_B  }  \( \nabla a_{(k)} \times (W^c_{(k)})  +a_{(k)} \wt W_{(k)}^c \)+  \sum_{k\in \Lambda_h} (\gamma^{(h)}_{\ell_q })^{\frac12} \wt W_{(k)}^c, \label{wqc} \\
		d_{q+1}^{(c)} &:=   \sum_{k\in \Lambda_B } \left(\nabla a_{(k)} \times D_{(k)}^c +a_{(k)}\wdc\right) +\sum_{k\in \Lambda_h} (\gamma^{(h)}_{\ell_q })^{\frac12} \wt W_{(k)}^c, \label{dqc}
	\end{align}
\end{subequations}
where $W^c_{(k)}$ and  $\wt W_k^c $  are given by \eqref{Vk-def} and \eqref{corrector vector}, respectively,
and
$D_{(k)}^c$ and $\wt D_{(k)}^c$ are as in \eqref{dkc}.

\medskip 
Then, it follows from \eqref{wcwc} and \eqref{D-wtD-Dc}  that
\begin{subequations}
	\begin{align}
		&  w_{q+1}^{(p)}+ w_{q+1}^{(h)} + w_{q+1}^{(c)}
		= \curl \left(  \sum_{k \in \Lambda_u \cup \Lambda_B} a_{(k)} W_{(k)}^c + \sum_{k \in \Lambda_h} (\gamma^{(h)}_{\ell_q })^{\frac12} W_{(k)}^c\right), \label{div free velocity} \\
		&  d_{q+1}^{(p)}+ d_{q+1}^{(h)} + d_{q+1}^{(c)}= \curl \left(  \sum_{k \in \Lambda_B} a_{(k)}  D_{(k)}^c+ \sum_{k \in \Lambda_h} (\gamma^{(h)}_{\ell_q })^{\frac12}W_{(k)}^c\right).  \label{div free magnetic}
	\end{align}
\end{subequations}
In particular,
\begin{align} \label{div-wpc-dpc-0}
	\div (w_{q+1}^{(p)}+ w_{q+1}^{(h)} + w_{q +1}^{(c)}) = \div (d_{q+1}^{(p)}+ d_{q+1}^{(h)} + d_{q +1}^{(c)}) =  0.
\end{align} 

\subsection{Key estimates of perturbations}

The key estimates of the velocity and magnetic perturbations are summarized in Lemma \ref{totalest} below.

\begin{lemma}  [Estimates of perturbations] \label{totalest}
For any $\rho \in(1,\9)$ and $0\leq N\leq 1$ the following estimates hold:
	\begin{align}
		&\norm{(\na^N w_{q+1}^{(p)},\na^N d_{q+1}^{(p)})}_{C_T\cL^\rho_x } \lesssim \ell_q^{-2}(\sigma \rs ^{-1})^N(\rp \rs \rw )^{\frac{1}{\rho}-\frac12} ,\label{uprinlp}\\
		&\norm{(\na^N w_{q+1}^{(h)},\na^N d_{q+1}^{(h)})}_{C_T\cL^\rho_x } \lesssim  \delta_{q+1}^{\frac{1}{2}}(\sigma \rs ^{-1})^N(\rp \rs \rw )^{\frac{1}{\rho}-\frac12} ,\label{uh}\\
		&\norm{(\na^N w_{q+1}^{(c)}, \na^N d_{q+1}^{(c)}) }_{C_T\cL^\rho_x  } \lesssim \ell_q^{-2}(\sigma \rs ^{-1})^N\rs \rp ^{-1}(\rp \rs \rw )^{\frac{1}{\rho}-\frac12}.\label{divcorlp}
	\end{align}
 Moreover, one has 
\begin{align}
& \norm{ w_{q+1}^{(p)} }_{C_{T,x}^1 }+\norm{ w_{q+1}^{(h)} }_{C_{T,x}^1 }  + \norm{ w_{q+1}^{(c)} }_{C_{T,x}^1 }
\lesssim \lambda^{4},\label{principal c1 est}\\
&\norm{ d_{q+1}^{(p)} }_{C_{T,x}^1 } +\norm{ d_{q+1}^{(h)} }_{C_{T,x}^1 } + \norm{ d_{q+1}^{(c)} }_{C_{T,x}^1 }
\lesssim  \lambda^{4}.\label{dprincipal c1 est}
\end{align}
\end{lemma}

\begin{proof}
By  the spatial intermittency estimates in Lemma \ref{buildingblockestlemma} 
and the amplitude estimates in Lemmas \ref{mae} and \ref{vae}, 
we infer that for any $\rho \in (1,+\infty)$,
\begin{align*}
\norm{\nabla^N w_{q+1}^{(p)} }_{C_TL^\rho_x } + \norm{\nabla^N d_{q+1}^{(p)}  }_{C_TL^\rho_x}
\lesssim& \, \sum_{k \in \Lambda_u \cup \Lambda_B}
\sum\limits_{N_1+N_2 = N}
\|a_{(k)}\|_{C_TC^{N_1}_{x}}
\norm{ \nabla^{N_2} W_{(k)}}_{C_TL^\rho_x } \notag \\
&  + \sum_{k \in \Lambda_B}\sum\limits_{N_1+N_2 = N}\|a_{(k)}\|_{C_TC^{N_1}_{x}} \norm{\nabla^{N_2} D_{(k)} }_{C_TL^\rho_x}  \notag  \\
\lesssim&\,	\ell_q^{-2}(\sigma \rs ^{-1})^N(\rp \rs \rw )^{\frac{1}{\rho}-\frac12},
\end{align*}
which verifies \eqref{uprinlp} 
for the principal part of perturbations.

Concerning the helicity correctors, by \eqref{est-theta} and Lemma \ref{buildingblockestlemma},
\begin{align*}
	\norm{\nabla^N w_{q+1}^{(h)} }_{C_TL^\rho_x } + \norm{\nabla^N d_{q+1}^{(h)}  }_{C_TL^\rho_x}
	\lesssim&  \|(\gamma^{(h)}_{\ell_q })^{\frac12}\|_{C_{T}}
	\norm{ \nabla^{N}( \psi_{(k_1)}\phi_{(k)}\phi_{(k_2)} k_3) }_{C_TL^\rho_x } \notag \\
	\lesssim&\,  \delta_{q+1}^{\frac{1}{2}}(\sigma \rs ^{-1})^N(\rp \rs \rw )^{\frac{1}{\rho}-\frac12},
\end{align*}
which verifies \eqref{uh} for the helicity corrector.
	
Moreover, using \eqref{b-beta-ve}, \eqref{ew}, \eqref{wqc} and Lemmas \ref{mae} and \ref{vae}
we have
\begin{align*} 
 \norm{\na^N w_{q+1}^{(c)}}_{C_TL^\rho_x}
	 \lesssim &  \sum_{k\in \Lambda_u \cup \Lambda_B}
	\left\|\na^N \( \nabla a_{(k)} \times W^c_{(k)}  +a_{(k)} \wt W_{(k)}^c \) \right\|_{C_TL^\rho_x} +  \sum_{k\in \Lambda_h}
	\left\|\na^N ((\gamma^{(h)}_{\ell_q })^{\frac12}\wt W_{(k)}^c)\right\|_{C_TL^\rho_x} \notag\\
	\lesssim&
	\sum\limits_{k\in \Lambda_u \cup \Lambda_B} \sum_{N_1+N_2=N} 
        \( \norm{ \nabla^{N_1+1} a_{(k)} }_{C_{T,x}} \norm{\na^{N_2} W^c_{(k)}}_{C_TW^{1,\rho}  }
	 +  \norm{ a_{(k)} }_{C_{T,x}^{N_1}} \norm{ \na^{N_2}\wt W^c_{(k)}}_{C_TL^\rho  }  \)  \nonumber   \\
	 &+ \sum\limits_{k\in \Lambda_h}
	 \norm{(\gamma^{(h)}_{\ell_q })^{\frac12} }_{C_{T}} \norm{ \na^{N}\wt W^c_{(k)}}_{C_TL^\rho  }  \notag\\
	 \lesssim&  \ell_q^{-2}(\sigma \rs ^{-1})^N\rs \rp ^{-1}(\rp \rs \rw )^{\frac{1}{\rho}-\frac12}.
\end{align*}
Similarly, by the spatial intermittency estimate \eqref{ed}, 
the identity \eqref{dqc} and 
the magnetic amplitude estimates in Lemma~\ref{mae},
\begin{align*}  
 \norm{\na^N d_{q+1}^{(c)}}_{C_TL^\rho_x}
		&\leq   \sum_{ k \in \Lambda_B} \left\|\na^N  \left( \nabla a_{(k)} \times D_{(k)}^c +a_{(k)}\wdc\right)\right\|_{C_TL^\rho_x} +  \sum_{k\in \Lambda_h}
		\left\|\na^N ((\gamma^{(h)}_{\ell_q })^{\frac12}\wt W_{(k)}^c)\right\|_{C_TL^\rho_x} \notag  \nonumber  \\
	& \lesssim \ell_q^{-2}(\sigma \rs ^{-1})^N\rs \rp ^{-1}(\rp \rs \rw )^{\frac{1}{\rho}-\frac12}.
	\end{align*}
	Thus, we obtain \eqref{divcorlp} 
    for the incompressible correctors.

It remains to prove the $C^1$-estimates \eqref{principal c1 est} and \eqref{dprincipal c1 est} of the perturbations. By Lemmas \ref{buildingblockestlemma}, \ref{mae} and \ref{vae},
\begin{align} \label{wprincipal c1 est}
	\norm{ w_{q+1}^{(p)} }_{C_{T,x}^1}
	\lesssim \sum_{k \in \Lambda_u \cup \Lambda_B}
	\|a_{(k)}\|_{C_{T,x}^{1} }\norm{ W_{(k)}}_{C_{T,x}^{1}}  \lesssim \lambda^{4},
\end{align}
where we also used the parameter choice in  \eqref{b-beta-ve} and \eqref{larsrp} in the last step.
Similarly, we have
\begin{align} \label{uh c1 est}
	\norm{ w_{q+1}^{(h)} }_{C_{T,x}^1 }
	& \lesssim	\sum_{0\leq N_1+N_2\leq 1} \|(\gamma^{(h)}_{\ell_q })^{\frac12}\|_{C_{T}^{1} }\norm{ W_{(k)}}_{C_{T,x}^{1}} \lesssim \lambda^{4}.
\end{align}

Using  Lemmas \ref{buildingblockestlemma}, \ref{mae} and \ref{vae} we get
\begin{align} \label{uc c1 est}
	\norm{ w_{q+1}^{(c)} }_{C_{T,x}^1 }
	& \lesssim   \sum_{k\in \Lambda_u \cup \Lambda_B}
	\left\| \nabla a_{(k)} \times W^c_{(k)} +a_{(k)}\wt W_{(k)}^c \right\|_{C_{T,x}^1 }+  \sum_{k\in \Lambda_h}
	\left\|(\gamma^{(h)}_{\ell_q })^{\frac12}\wt W_{(k)}^c\right\|_{C_{T,x}^1 }  \notag \\
	& \lesssim   \sum_{k \in \Lambda_u \cup \Lambda_B}\sum_{M_1+M_2+N_1+N_2\leq 1 }
	\|\p_t^{M_1}\nabla^{N_1+1}a_{(k)}\|_{C_{T,x}}\norm{\p_t^{M_2}\nabla^{N_2}W^c_{(k)} }_{C_{T,x}}\nonumber \\
&\quad+  \sum_{k \in \Lambda_u \cup \Lambda_B}\sum_{M_1+M_2+N_1+N_2\leq 1 }
\|\p_t^{M_1}\nabla^{N_1}a_{(k)}\|_{C_{T,x}}\norm{\p_t^{M_2}\nabla^{N_2}\wt W_{(k)}^c  }_{C_{T,x}}\notag\\
&\quad + \sum_{k\in \Lambda_h}
 \|(\gamma^{(h)}_{\ell_q })^{\frac12}\|_{C_{T,x}^1 } \|\wt W_{(k)}^c \|_{C_{T,x}^1 }  \notag \\
	& \lesssim \lambda^{4}.
\end{align}

Thus, combining estimates \eqref{wprincipal c1 est}-\eqref{uc c1 est} altogether and using \eqref{larsrp}
we conclude that
\begin{align} \label{utotalc1}
	&\norm{ w_{q+1}^{(p)} }_{C_{T,x}^1 }  +\norm{ w_{q+1}^{(h)} }_{C_{T,x}^1 }  + \norm{ w_{q+1}^{(c)} }_{C_{T,x}^1 }
	\lesssim \lambda^{4}.
\end{align}
The $C^1$-estimates of the magnetic perturbations in \eqref{dprincipal c1 est} can be proved similarly.  
\end{proof}

\subsection{Heat correctors}  \label{Subsec-he-cor} 
We introduce a new corrector 
arising in particular from the specific structure of MHD system, 
not present in the NSE case. 
Actually, in the NSE case, the temporal corrector
$w_{q+1}^{(t)}$ is enough to balance the high oscillations in the nonlinearity $W_{(k)}\otimes W_{(k)}$, but insufficient 
to balance the high oscillations 
for the magnetic components $D_{(k)}\otimes D_{(k)} $ and $W_{(k)}\otimes D_{(k)} $ caused by the concentration function $\phi_{(k_2)}$. 

\medskip
In order to control those high oscillations, 
we introduce the heat correctors by
\begin{subequations}  \label{def-lap-cor}
	\begin{align}
		&w_{q+1}^{(H)} := -\sum_{k\in \Lambda_u\cup\Lambda_B} \P_{H}\P_{\neq 0}\(a_{(k)}^2\int_0^t e^{\nu_1(t-\tau)\Delta}\div  (W_{(k)} \otimes W_{(k)}- D_{(k)} \otimes D_{(k)})\d \tau\),\label{vellapcor}\\
		\label{maglapcor}
		&d_{q+1}^{(H)}: = -\sum_{k\in \Lambda_B} \P_{H}\P_{\neq 0}\(a_{(k)}^2\int_0^t e^{\nu_2(t-\tau)\Delta}\div  (D_{(k)} \otimes W_{(k)}- W_{(k)} \otimes D_{(k)})\d \tau\).
	\end{align}
\end{subequations}

Then, by the Leibniz rule, we have
\begin{equation} \label{ulap}
\begin{split}
	(\p_t -\nu_1\Delta)& w_{q+1}^{(H)}+ \sum_{k \in \Lambda_u\cup\Lambda_B}  \P_{H}\P_{\neq 0}
	\(a_{(k)}^{2} \div (W_{(k)} \otimes W_{(k)}- D_{(k)} \otimes D_{(k)})\)    \\
	=
	&+\sum_{k\in \Lambda_u\cup\Lambda_B} \P_{H}\P_{\neq 0}\((\p_t -\nu_1\Delta)a_{(k)}^2\int_0^t e^{(t-\tau)\nu_1\Delta}\div (W_{(k)} \otimes W_{(k)}- D_{(k)} \otimes D_{(k)})\d \tau\),   \\
		& - 2\nu_1\sum_{k \in \Lambda_u\cup\Lambda_B}  \P_H \P_{\neq 0}
	\(\sum_{i=1}^3\partial_{i}( a_{(k)}^{2})\P_{\neq 0} \partial_{i}\int_0^t e^{\nu_2(t-\tau)\nu_1\Delta}\div (W_{(k)} \otimes W_{(k)}- D_{(k)} \otimes D_{(k)})\d \tau \),
\end{split}
\end{equation}
and
\begin{equation} \label{blap}
\begin{split}
	(\p_t -\nu_2\Delta)& d_{q+1}^{(H)}+\sum_{k\in \Lambda_B} \P_{H}\P_{\neq 0}
	\(a_{(k)}^{2} \div(D_{(k)} \otimes W_{(k)}-W_{(k)} \otimes D_{(k)} )\)     \\
	=	&\sum_{k\in \Lambda_B} \P_{H}\P_{\neq 0}\((\p_t -\nu_2\Delta)a_{(k)}^2\int_0^t e^{(t-\tau)\nu_2\Delta}\div(D_{(k)} \otimes W_{(k)}-W_{(k)} \otimes D_{(k)} )\d \tau\),   \\
	& - 2\nu_2\sum_{k \in \Lambda_B}  \P_H \P_{\neq 0}
	\(\sum_{i=1}^3\partial_{i}( a_{(k)}^{2})\P_{\neq 0} \partial_{i}\int_0^t e^{\nu_2(t-\tau)\Delta}\div(D_{(k)} \otimes W_{(k)}-W_{(k)} \otimes D_{(k)} )\d \tau  \) .
\end{split}
\end{equation}

The main advantage of the above identities \eqref{ulap} and \eqref{blap} is, that the heat correctors allow to transfer the spatial derivative from the high frequency functions to the relatively low frequency amplitudes $\{a^2_{(k)}\}$. It turns out that the remaining errors on the right-hand side of \eqref{ulap} and \eqref{blap} are acceptable in the convex integration scheme.

\medskip
The key estimates of the heat correctors are contained in 
the following lemma.

\begin{lemma}[Estimate of heat correctors] \label{lem-heat-est} We have
\begin{align}
   \norm{ (w_{q+1}^{(H)}, d_{q+1}^{(H)} )}_{C_T\cL^2_x  }\lesssim \ell_q^{-4}\lambda^{-2\ve}, \label{lapcorlp}
\end{align}
and
\begin{align}\label{he-cor-c1}
\norm{(w_{q+1}^{(H)}, d_{q+1}^{(H)}) }_{\mathcal{C}_{T,x}^1 }
\lesssim \lambda^{4}.
\end{align}
\end{lemma}

\begin{proof}
We divide $w_{q+1}^{(H)}$ into two component: 
\begin{align*}
    w_{q+1}^{(H)}:=J_1+J_2,
\end{align*}
where 
\begin{align*}
    J_1:= -\sum_{k\in \Lambda_B} \P_{H}\P_{\neq 0}\(a_{(k)}^2\int_0^t e^{\nu_1(t-\tau)\Delta}\div  (W_{(k)} \otimes W_{(k)})\d \tau\)
\end{align*}
and 
\begin{align*}
    J_2:= \sum_{k\in \Lambda_B} \P_{H}\P_{\neq 0}\(a_{(k)}^2\int_0^t e^{\nu_1(t-\tau)\Delta}\div  (D_{(k)} \otimes D_{(k)})\d \tau\)
\end{align*}
Let us first treat $J_2$. By \eqref{def-lap-cor}, straightforward computations yield
\begin{align}\label{decom-wh}
J_2& = \sum_{k\in \Lambda_B} \P_{H}\P_{\neq 0}\(a_{(k)}^2\int_0^t e^{\nu_1(t-\tau)\Delta}\div  (\phi_{(k)}^2\phi_{(k_2)}^2 k_2\otimes k_2)\d \tau\)\notag\\
&\quad+ \sum_{k\in \Lambda_B} \P_{H}\P_{\neq 0}\(a_{(k)}^2 \div  (\p_t^{-1}(\P_{\neq 0}\psi_{(k_1)}^2)\phi_{(k)}^2  \phi_{(k_2)}^2 k_2\otimes k_2) \)\notag\\
&\quad +\sum_{k\in \Lambda_B} \P_{H}\P_{\neq 0}\(a_{(k)}^2\int_0^t e^{\nu_1(t-\tau)\Delta}\nu_1\Delta \div  ( \p_t^{-1}(\P_{\neq 0}\psi_{(k_1)}^2)\phi_{(k)}^2 \phi_{(k_2)}^2 k_2\otimes k_2)\d \tau\)\notag\\
&=: J_{21}+J_{22}+J_{23}, 
\end{align}
where the inverse time derivative is defined by
\begin{align}\label{def-inv-t}
\p_t^{-1}(\P_{\neq 0}\psi_{(k_1)}^2)&=
\int_0^t \big( \psi_{(k_1)}^2(\sigma N_\Lambda(k_1\cdot x+\mu \tau )) - 1 \big) \, d\tau
\end{align}
Note that $\p_t^{-1}(\P_{\neq 0}\psi_{(k_1)}^2)$ is periodic in time and, since $\|\psi_{(k_1)}\|_{L^2}=1$, 
\begin{equation} \label{eq:bound_on_inv_dt}
\|\p_t^{-1}(\P_{\neq 0}\psi_{(k_1)}^2)\|_{C_{T,x}} \lesssim (\sigma N_\Lambda\mu)^{-1}.
\end{equation}

Here, $J_{21}$ is the high spatial frequency component of $J_2$, $J_{22}$ and $J_{23}$ are high temporal frequency components of $J_2$.
\medskip 

\paragraph{\bf (i) High spatial frequency component:}   
For the high spatial frequency component $J_{21}$, 
using the heat kernel estimate we derive,  
\begin{align}\label{vh-est-1}
	\norm{J_{21}}_{C_TL^2_x}
	\lesssim & \,\sum_{k \in \Lambda_B}
	\|a_{(k)}^2\|_{C_{T,x} }\norm{  \Delta^{-1}\div( \phi_{(k)}^2\phi_{(k_2)}^2 k_2\otimes k_2) }_{ H^{\frac{\ve}{2}}_x }.
\end{align}	

Let us treat the second term on the right-hand side of \eqref{vh-est-1}. By interpolation inequality, we have
\begin{align}\label{est-vh-1-1}
\norm{  \Delta^{-1}\div( \phi_{(k)}^2\phi_{(k_2)}^2 k_2\otimes k_2) }_{ H^{\frac{\ve}{2}}_x }\lesssim& 
\norm{  \Delta^{-1}\div( \phi_{(k)}^2\phi_{(k_2)}^2 k_2\otimes k_2) }_{ L^2_x }^{1-\frac{\ve}{2}}\norm{  \Delta^{-1}\div( \phi_{(k)}^2\phi_{(k_2)}^2 k_2\otimes k_2) }_{ H^{1}_x }^\frac{\ve}{2}.
\end{align}
Let us set $f(\sigma  N_{\Lambda}k\cdot  (x-\alpha_k),\sigma  N_{\Lambda}k_2\cdot  (x-\alpha_k)):= \phi_{(k)}^2\phi_{(k_2)}^2 k_2\otimes k_2$. 
Using the Sobolev embedding $W^{1,\wt p}(\T^2)\hookrightarrow L^2(\T^2)$, $\wt p=\frac{4}{4-\ve}$, 
we obtain
\begin{align}\label{i1-l2}
 \norm{  \Delta^{-1}\div( \phi_{(k)}^2\phi_{(k_2)}^2 k_2\otimes k_2) }_{ L^2_x }&\lesssim  \sigma^{-1} \norm{  (\Delta^{-1}\p_2 f) (\sigma  N_{\Lambda}k\cdot  (x-\alpha_k),\sigma  N_{\Lambda}k_2\cdot  (x-\alpha_k)) }_{  L^2_x } \notag\\
 &\lesssim  \sigma^{-1} \norm{  (\Delta^{-1}\p_2 f) (\cdot,\cdot) }_{L^2(\T^2)} \notag\\
 &\lesssim \sigma^{-1} \|f (\cdot,\cdot)\|_{L^{\wt p}(\T^2)} \lesssim \sigma^{-1} \lambda^{\frac{\ve}{2}}.
\end{align}
Moreover, by the spatial intermittency estimates in Lemma~\ref{buildingblockestlemma},  
\begin{align}\label{i1-h1}
   \norm{  \Delta^{-1}\div( \phi_{(k)}^2\phi_{(k_2)}^2 k_2\otimes k_2) }_{ H^{1}_x }\lesssim \|\phi_{(k)}^2\phi_{(k_2)}^2\|_{L^2_x}\lesssim \lambda^\frac58. 
\end{align}
Plugging estimates \eqref{i1-l2} and \eqref{i1-h1} into \eqref{est-vh-1-1} we arrive at
\begin{align}\label{est-wh-f}
    \norm{  \Delta^{-1}\div( \phi_{(k)}^2\phi_{(k_2)}^2 k_2\otimes k_2) }_{ H^{\frac{\ve}{2}}_x }\lesssim& (\sigma^{-1} \lambda^{\frac{\ve}{2}})^{ 1-\frac{\ve}{2}} \lambda^\frac{5\ve}{16}.
\end{align}

Hence, using the estimates of the magnetic amplitudes in  Lemma~\ref{mae} we get
\begin{align}\label{vh-est-1-end}
	\norm{J_{21}}_{C_TL^2_x}
	\lesssim & \, \ell_q^{-4}\sigma^{-1}\lambda^{\ve }.
\end{align}

\paragraph{\bf (ii) High temporal frequency components:}   
We first treat the high temporal frequency component $J_{22}$. 
By Lemmas~\ref{buildingblockestlemma}, \ref{mae} and \eqref{eq:bound_on_inv_dt}, 
\begin{align}\label{vh-est-2}
	\norm{J_{22}}_{C_TL^2_x}
	&\lesssim \sum_{k \in \Lambda_B}
	\|a_{(k)}^2\|_{C_{T,x} } \|\p_t^{-1}(\P_{\neq 0}\psi_{(k_1)}^2) \|_{C_{T,x}} \norm{ \phi_{(k)}^2 }_{L^2_x }\norm{ \nabla(\phi_{(k_2)}^2) }_{L^2_x}\notag\\
	&\lesssim\ell_q^{-4}\mu^{-1} \rs^{-\frac12}\rw^{-\frac32}.
\end{align}

Regarding the remaining high temporal frequency term $J_{23}$,  the interpolation inequality gives 
\begin{align}\label{est-i3-1}
&\quad\|\div  ( \p_t^{-1}(\P_{\neq 0}\psi_{(k_1)}^2)\phi_{(k)}^2 \phi_{(k_2)}^2 k_2\otimes k_2)\|_{C_TH^{\frac{\ve}{2} }_x}\notag\\
&\lesssim  \|\div  ( \p_t^{-1}(\P_{\neq 0}\psi_{(k_1)}^2)\phi_{(k)}^2 \phi_{(k_2)}^2 k_2\otimes k_2)\|_{C_TL^{2}_x}^{1-\frac{\ve}{2}} \|\div  ( \p_t^{-1}(\P_{\neq 0}\psi_{(k_1)}^2)\phi_{(k)}^2 \phi_{(k_2)}^2 k_2\otimes k_2)\|_{C_TH^{1}_x}^{\frac{\ve}{2}}.
\end{align}
Using Lemma~\ref{buildingblockestlemma} and \eqref{eq:bound_on_inv_dt} again we estimate 
\begin{align}\label{est-i3-1-1}
&\quad \|\div  ( \p_t^{-1}(\P_{\neq 0}\psi_{(k_1)}^2)\phi_{(k)}^2 \phi_{(k_2)}^2 k_2\otimes k_2)\|_{C_TL^{2}_x}\notag\\
&\lesssim \| \p_t^{-1}(\P_{\neq 0}\psi_{(k_1)}^2)\|_{C_{T}L^2_x}\|\phi_{(k)}^2\|_{C_TL^{2}_x}\|\nabla(\phi_{(k_2)}^2) \|_{C_TL^{2}_x}\notag\\
&\lesssim  (\mu\sigma)^{-1}\rs^{-\frac12}\sigma\rw^{-\frac32},
\end{align}
and
\begin{align}\label{est-i3-1-2}
&\quad\|\div  ( \p_t^{-1}(\P_{\neq 0}\psi_{(k_1)}^2)\phi_{(k)}^2 \phi_{(k_2)}^2 k_2\otimes k_2)\|_{C_TH^{1}_x}\notag\\
&\lesssim  \sum_{N_1+N_2+N_3=1}\| \nabla^{N_1}\p_t^{-1}(\P_{\neq 0}\psi_{(k_1)}^2)\|_{C_{T}L^2_x}\|\nabla^{N_2}(\phi_{(k)}^2)\|_{C_TL^{2}_x}\|\nabla^{N_3+1}(\phi_{(k_2)}^2 )\|_{C_TL^{2}_x}\notag\\
&\lesssim \sum_{N_1+N_2+N_3=1}  (\mu\sigma)^{-1}(\sigma \rp^{-\frac12})^{N_1} (\sigma \rs^{-1})^{N_2} \rs^{-\frac12} (\sigma \rw^{-1})^{N_3+1} \rw^{-\frac12}\notag\\
&\lesssim  (\mu\sigma)^{-1}\lambda \rs^{-\frac12}\sigma\rw^{-\frac32}.
\end{align}
Plugging estimates \eqref{est-i3-1-1} and \eqref{est-i3-1-2} into \eqref{est-i3-1}, we get 
\begin{align}\label{est-i3-1-end}
\|\div  ( \p_t^{-1}(\P_{\neq 0}\psi_{(k_1)}^2)\phi_{(k)}^2 \phi_{(k_2)}^2 k_2\otimes k_2)\|_{C_TH^{\frac{\ve}{2} }_x}\lesssim  (\mu\sigma)^{-1}\lambda^{\ve}\sigma \rs^{-\frac12}\rw^{-\frac32}.
\end{align}
Hence, by Lemma \ref{mae} and \eqref{est-i3-1-end}, we arrive at
\begin{align}\label{vh-est-3}
	\norm{J_{23}}_{C_TL^2_x}
	&\lesssim \sum_{k \in \Lambda_B}
	\|a_{(k)}^2\|_{C_{T,x} } \|\div  ( \p_t^{-1}(\P_{\neq 0}\psi_{(k_1)}^2)\phi_{(k)}^2 \phi_{(k_2)}^2 k_2\otimes k_2)\|_{C_TH^{\frac{\ve}{2}}_x}\notag\\
	&\lesssim \ell_q^{-4}\mu^{-1} \lambda^{\ve} \rs^{-\frac12}\rw^{-\frac32}.
\end{align}	
Therefore, 
combining estimates \eqref{vh-est-1-end}, \eqref{vh-est-2} and \eqref{vh-est-3} altogether we have
\begin{align}\label{est-j2-heat}
\norm{J_2}_{C_TL^2_x}\lesssim \ell_q^{-4}(\sigma^{-1}\lambda^{\ve}+ \mu^{-1} \lambda^{\ve} \rs^{-\frac12}\rw^{-\frac32})\lesssim \ell_q^{-4}\lambda^{-3\ve}.
\end{align}
Arguing in a similar manner as above, we have
\begin{align*}
\norm{J_1}_{C_TL^2_x}\lesssim \ell_q^{-4}\lambda^{-2\ve},
\end{align*}
which along with \eqref{est-j2-heat} yields
\begin{align*}
\norm{w_{q+1}^{(H)}}_{C_TL^2_x}\lesssim \ell_q^{-4}\lambda^{-2\ve},
\end{align*}
thereby verifies the $\mathcal{L}_x^2$-decay estimate \eqref{lapcorlp} 
for the heat corrector of velocity perturbations. 

\medskip 
The $C^1$-estimate of heat corrector 
can be estimated in a more direct way. 
Actually, using Lemmas \ref{buildingblockestlemma} and \ref{mae} and the Sobolev embedding $W^{\ve,p'}(\mathbb{T}^3) \hookrightarrow L^\9(\mathbb{T}^3)$, $p'>\frac3\ve$, we get
\begin{align*}
	\norm{ w_{q+1}^{(H)} }_{C_{T,x}^1} \lesssim &\,  \sum_{0\leq N_1+N_2\leq 1} \norm{ \p_t^{N_1}\nabla^{N_2} w_{q+1}^{(H)} }_{C_TW^{\ve,p'}_x}\notag\\
	\lesssim &\, \sum_{k \in \Lambda_u\cup\Lambda_B}\|a_{(k)}^2\|_{ C_{T,x}^2}
	 \sum_{0\leq N_1+N_2\leq 1 }
\|\p_t^{N_1}\nabla^{N_2}(\psi^2_{(k_1)} \phi^2_{(k)}\phi^2_{(k_2)})\|_{C_TW^{\ve,p'}_x}\notag\\
	\lesssim & \, \ell_q^{-84} \sum_{0\leq N_1+N_2\leq 1 } (\sigma\rs^{-1})^{N_2+\ve}( \sigma\rp^{-1}\mu )^{N_1} (\rp\rs\rw)^{\frac{1}{p'}-1} 
	\lesssim  \lambda^4,
\end{align*}
where the last two steps were due to \eqref{b-beta-ve} and \eqref{larsrp}. 

The estimates of magnetic perturbations in \eqref{lapcorlp} and \eqref{he-cor-c1} can be proved similarly. 
\end{proof}

\subsection{Verification of iterative  estimates}\label{Subsec-induc-vel-mag} 

We are now in position to define the velocity and magnetic perturbations $w_{q+1}$ and $d_{q+1}$,  respectively, at level $q+1$ by
\begin{subequations}   \label{perturbation}
\begin{align}
	w_{q+1} &:= \chi_{q+1}w_{q+1}^{(p)}+\chi_{q+1} w_{q+1}^{(h)} +\chi_{q+1} w_{q+1}^{(c)}+\chi_{q+1}^2 w_{q+1}^{(H)},
	\label{velocity perturbation}\\
		d_{q+1} &:= \chi_{q+1}d_{q+1}^{(p)}+\chi_{q+1} d_{q+1}^{(h)} +\chi_{q+1} d_{q+1}^{(c)}+\chi_{q+1}^2 d_{q+1}^{(H)},
	\label{magnetic perturbation}
\end{align}
\end{subequations}
where $\chi_{q+1}$ is the temporal cutoff function defined by \eqref{def-chiq}.

To ease notations, let
\begin{align*}
	\wt w_{q+1}^{(*_1)}:= \chi_{q+1}w_{q+1}^{(*_1)}, \quad \wt w_{q+1}^{(*_2)} :=\chi_{q+1}^2 w_{q+1}^{(*_2)}
\end{align*}
and
\begin{align*}
	\wt d_{q+1}^{(*_1)}:= \chi_{q+1}d_{q+1}^{(*_1)}, \quad \wt d_{q+1}^{(*_2)} :=\chi_{q+1}^2  d_{q+1}^{(*_2)},
\end{align*}
where $*_1\in\{p,h,c\}$ and $*_2\in \{H\}$. 
Then,
\begin{align}
	w_{q+1} = \wt w_{q+1}^{(p)} +\wt w_{q+1}^{(h)} +\wt  w_{q+1}^{(c)} +\wt w_{q+1}^{(H)},\quad d_{q+1} = \wt d_{q+1}^{(p)}+\wt d_{q+1}^{(h)} +\wt  d_{q+1}^{(c)}+\wt d_{q+1}^{(H)}.
\end{align}

The velocity and magnetic fields at level $q+1$ are defined by
\begin{align}  \label{q+1 iterate}
		  u_{q+1}:= u_{q} + w_{q+1},  \ \
		  B_{q+1}:= B_{q}+ d_{q+1} .
\end{align}

\medskip
Below we verify the inductive estimates \eqref{ubl2}, \eqref{ubc} and \eqref{u-B-L2tx-conv}  for $u_{q+1}$ and $B_{q+1}$. 
To ease notations, let us set
\begin{align*}
	&w_{q+1}^{(p)+(h)}:=  w_{q+1}^{(p) }+  w_{q+1}^{ (h) },\quad   d_{q+1}^{(p)+(h)}:=  d_{q+1}^{(p) }+  d_{q+1}^{ (h) },\notag\\
	&w_{q+1}^{(c)+(H)}:=  w_{q+1}^{(c) }+  w_{q+1}^{ (H)},\quad   d_{q+1}^{(c)+(H)}:=  d_{q+1}^{(c) }+  d_{q+1}^{ (H)}.
\end{align*}
We first apply the decorrelation Lemma~\ref{Decorrelation1} with $f= a_{(k)}$ and $g =\mathbb{P}_{\sigma}( \psi_{(k_1)}\phi_{(k)}\phi_{(k)})$ to get
\begin{align}
	\label{Lp decorr vel}
	\norm{w^{(p)}_{q+1}}_{C_TL^2_{x}}
	&\lesssim \sum\limits_{k\in \Lambda_u \cup \Lambda_B}
       \Big(\|a_{(k)}\|_{C_TL^2_{ x}}\norm{ \psi_{(k_1)}\phi_{(k)}\phi_{(k_2)}}_{C_TL^2_x } +\sigma^{-\frac12}\|a_{(k)}\|_{C^1_{T,x}} \norm{  \psi_{(k_1)}\phi_{(k)}\phi_{(k_2)}}_{C_TL^2_x }\Big).
\end{align}
By \eqref{la}, \eqref{b-beta-ve} and 
the amplitude estimates in Lemmas \ref{mae} and \ref{vae},
\begin{align}\label{Lp-wdp-1}
	\norm{w^{(p)}_{q+1}}_{C_T L^2_{x}}
	&\lesssim \delta_{q+ 1}^{\frac12}+\ell_q^{-42}\lambda^{-2\ve}_{q+1}\lesssim  \delta_{q+ 1}^{\frac12}.
\end{align}

Note that, 
for the velocity perturbation $w_{q+1}$,
by \eqref{velocity perturbation}, \eqref{Lp-wdp-1} and Lemma \ref{totalest}, 
\begin{align}  \label{e3.41}
	\norm{w_{q+1}}_{C_T L^2_{x}}
     &\lesssim  \delta_{q+1}^{\frac{1}{2}}+ \ell_q^{-2} r_{\perp} r_{\parallel}^{-1}
   + \ell_q^{-4}  \lambda^{-2\ve}_{q+1}\lesssim \delta_{q+1}^{\frac{1}{2}}.
\end{align}
Then, in view of \eqref{q+1 iterate},
we get
\begin{align*} 
	\norm{u_{q+1} -u_{q} }_{ C_TL^2_{x}}
	= \norm{w_{q+1}}_{C_TL^2_{x}} 
	\lesssim  \delta_{q+1}^{\frac{1}{2}}, 
\end{align*} 
which verifies  \eqref{u-B-L2tx-conv} 
for the velocity difference. 

Moreover, one has 
\begin{align*} 
\| u_{q+1}\|_{C_TL^2_{x} }  \leq \| u_{q+1} -u_{q} \|_{C_TL^2_{x} } + \| u_{q}\|_{C_TL^2_{x} } \leq \sum\limits_{n=0}^{q+1} \delta_{n}^{\frac12}, 
\end{align*}
thus justifying \eqref{ubl2} for the velocity field.

For the $C^1$-estimate \eqref{ubc} of the velocity, using \eqref{a-big-lambda}, \eqref{principal c1 est}, \eqref{he-cor-c1} and \eqref{est-wtu-c} below, we get that for $a$ sufficiently large,
\begin{align*}
	\norm{u_{q+1}}_{C^1_{[T_{q+3},S_{q+3}],x}}&  \lesssim  \norm{u_{q}}_{C^1_{[T_{q+3},S_{q+3}],x}}+\norm{ w_{q+1}}_{C^1_{T,x}}\notag\\
 &\lesssim \norm{u_{0}}_{C^1_{[T_{q+3},S_{q+3}],x}}+ \sum_{n=1}^{q}\norm{w_{n}}_{C^1_{T,x}} + \norm{w_{q+1}}_{C^1_{T,x}} \notag\\
 &\lesssim \lambda_{q+1}^{4}+ \sum_{n=1}^{q} \lambda_{n}^{4}+\lambda_{q+1}^{4}\notag\\    
 &\lesssim \lambda^{4}_{q+1}.
\end{align*}

At last, the iterative estimates \eqref{ubl2}, \eqref{ubc} and \eqref{u-B-L2tx-conv} of the magnetic perturbations can be verified in an analogous manner.

\section{Decreasing energy and cross helicity}
\label{Sec-Energy}

We aim to verify the iterative estimates \eqref{asp-e-0} and \eqref{asp-e-q} at level $q+1$ 
for the energy and the cross helicity, 
respectively.  
For simplicity, 
we set 
$\Tm:= (T_{q+2}+T_{q+1})/2$ 
and 
$\Sm:= (S_{q+2}+S_{q+1})/2$.

\subsection{Verification of energy iterative estimate}  
\label{Subsec-Energy-Verif}

We consider the perturbed regime 
$\cI^{\rm p}:=[\Tm, \Sm]$ 
and the non-perturbed regime 
$[T_{q+2},S_{q+2}]\setminus \cI^{\rm p}$, 
respectively.

In the non-perturbed regime where $t\in [T_{q+2},S_{q+2}]\setminus \cI^{\rm p}$,  
since no perturbations are added to the velocity and magnetic fields 
at level $q+1$, 
one has $(u_{q+1},B_{q+1})(t)=(u_0,B_0)(t)$. Hence, by \eqref{asp-e-0} and \eqref{asp-e-q}, 
it follows immediately that  
\begin{align}\label{e-t-1}
	\frac12\delta_{q+3}\leq e(t) -  (\|u_{q+1}(t) \|_{L_x^2}^2+\|B_{q+1}(t) \|_{L_x^2}^2)\leq 2\delta_{q+ 2}, 
    \quad 
    t\in [T_{q+2},S_{q+2}]\setminus \cI^{\rm p}. 
\end{align}

For the more delicate peturbed regime where $t\in \cI^{\rm p}$, by the definition of $u_{q+1}$ and $B_{q+1}$ in \eqref{q+1 iterate},
\begin{align}\label{ver-en-diff}
	& 	e(t) -(\|u_{q+1}(t)\|_{L_x^2}^2+\|B_{q+1}(t)\|_{L_x^2}^2)-f_q^{(e)}(t)\notag\\
	=
	&  \Big( -2 \|  (\wt w^{(p)}_{q+1}+\wt w^{(h)}_{q+1} )\cdot(\wt w_{q+1}^{(c)+(H)})(t)\|_{L^1_x}- \|( \wt w_{q+1}^{(c)+(H)})(t)\|_{ L^2_x}^2
	 - 2 \int_{\T^3}  u_q(t) \cdot w _{q+1}(t) \d x\Big)\notag\\
	&  + \Big( -2   \|   (\wt d^{(p)}_{q+1}+\wt d^{(h)}_{q+1} ) \cdot(\wt d_{q+1}^{(c)+(H)})(t)\|_{L^1_x}-  \|( \wt d_{q+1}^{(c)+(H)})(t)\|_{ L^2_x}^2 - 2 \int_{\T^3} B_q(t) \cdot d _{q+1}(t) \d x\Big)   \notag \\ 
    & + \Big(e(t) -(\|u_{q}(t)\|_{L_x^2}^2+\|B_{q}(t)\|_{L_x^2}^2)-f_q^{(e)}(t)\notag\\
	&\qquad -(\|\wt w_{q+1}^{(pu)}(t)\|_{L^2_x}^2+\|\wt w_{q+1}^{(pb)}(t)\|_{L^2_x}^2+\|\wt d_{q+1}^{(p)}(t)\|_{L^2_x}^2+\|\wt w_{q+1}^{(h)}(t)\|_{L^2_x}^2+\|\wt d_{q+1}^{(h)}(t)\|_{L^2_x}^2) \Big)   \notag \\
 \notag\\
	=:& \delta e_1+\delta e_2+\delta e_3,
\end{align}
where $f_q^{(e)}$ is defined by \eqref{def-fq}. Next, we estimate each energy error  $\delta e_i$, $i=1,2,3$, separately as follows. 

\medskip 
Let us first consider the energy error  $\delta e_1$. 
In view of estimates \eqref{Lp-wdp-1} and \eqref{e3.41}, we get
\begin{align}\label{ver-en-u-2}
	2 \|(\wt  w^{(p)}_{q+1}+\wt  w^{(h)}_{q+1})\cdot \wt  w_{q+1}^{(c)+(H)}\|_{L^1_x}
	&\lesssim   (\|  w^{(p)}_{q+1} \|_{ C_TL^2_x}+\|  w^{(h)}_{q+1} \|_{ C_TL^2_x})  \| w_{q+1}^{(c)+(H)} \|_{ C_TL^2_x}\notag\\
	&\lesssim \delta_{q+1}^{\frac{1}{2}} \ell_q^{-4} \laq^{-2\ve}\leq \frac{1}{192}\delta_{q+3}.
\end{align}
For the second component of $\delta e_1$, we have
\begin{align}\label{ver-en-u-3}
	\|\wt w_{q+1}^{(c)+(H)}\|_{L^2_x}^2
	&\lesssim \|  w^{(c)}_{q+1}\|_{L^2_x}^2+\|  w^{(H)}_{q+1}\|_{L^2_x}^2 \lesssim \ell_q^{-8} \laq^{-4\ve} \leq  \frac{1}{192}\delta_{q+ 3}.
\end{align}
Regarding the third component of $\delta e_1$, by \eqref{ubc} and Lemmas~\ref{totalest} and \ref{lem-heat-est},
\begin{align}\label{ver-en-u-4}
	 \|  u_q \cdot w _{q+1} \|_{L^1_x}
	&\lesssim \|  u_q \|_{ C_{\cI^{\rm p}} C_x }\| w_{q+1} \|_{ C_TL^1_x } 
	\lesssim \la^4\ell_q^{-4}\lambda_{q+1}^{-3\ve}\leq  \frac{1}{192}\delta_{q+ 3}.
\end{align}
Thus, combining \eqref{ver-en-u-2}-\eqref{ver-en-u-4} altogether we arrive at
\begin{align}\label{ver-en-i2}
	|\delta e_1|\leq \frac{1}{48}\delta_{q+ 3}.
\end{align}
Analogous argument also yields
\begin{align}\label{ver-en-i3}
	|\delta e_2|\leq \frac{1}{48}\delta_{q+ 3}.
\end{align}

It thus remains to treat the last energy error $\delta e_3$ on the right-hand side of \eqref{ver-en-diff}.  
For this purpose, 
by the identity \eqref{velcancel}, we note that 
\begin{align*}
 |\wt w_{q+1}^{(pu)}(t)|^2=&\, 3\cq^2 (\rho_u+\gamma^{(e)}_q)+ \sum_{k \in \Lambda_u  } \cq^2a_{(k)}^2(t)P_{\neq 0}(\left|W_{(k)}(t)\right|^2).
\end{align*}
Hence, 
\begin{align}\label{ver-en-diff-3}
\delta e_3 & = \Big(e(t) -\|u_{q}(t)\|_{L_x^2}^2-\|B_{q}(t)\|_{L_x^2}^2-f_q^{(e)}(t)\notag\\
   &\qquad-\chi_{q+1}^2(t)\( e -\|u_{q}\|_{L_x^2}^2-\|B_{q}\|_{L_x^2}^2-f_q^{(e)}\)*_{t}\varphi_{\ell_q}(t)\Big)\notag\\
&\quad-\chi_{q+1}^2(t)\int_{\T^3} 
\Big( 6\ve_u^{-1} (\ell_q^2+ |\mathring{R}^u_{\ell_q}(t)+\mathring{G}(t) |^2)^{\frac12}+\sum_{k \in \Lambda_u} a_{(k)}^2(t) P_{\neq 0}(\left|W_{(k)}(t)\right|^2) \Big) \d x\notag\\
	&:=\delta e_{31}+\delta e_{32}.
\end{align} 

We further divide the perturbed regime 
into three subregimes 
$$ \cI^{\rm p}= \bigcup_{i=1}^3 \mathcal{I}^p_i,$$
where
$\cI^{\rm p}_1 = [\overline{T}_{q+1},S_{q+1}]$ 
is the intermediate time interval between the backward and forward time sequence, 
$\cI^{\rm p}_2 = [T_{q+1}, \overline{T}_{q+1}]$ 
is the small mollified backward regime,  
and the remaining regime 
$\cI^{\rm p}_3 = \cI^{\rm p} \setminus (\cI^{\rm p}_1 \cup \cI^{\rm p}_2)$ 
is the time interval on which no perturbations are added at level $q$ 
(not at level $q+1$).

\medskip 
\paragraph{\bf (i) The  intermediate temporal regime $\mathcal{I}^p_1$:} 
In the intermediate temporal regime where $t\in \mathcal{I}^p_1=[\overline{T}_{q+1},S_{q+1}]$, 
using Lemmas~\ref{vae} and \ref{lem-mean} in the Appendix we derive 
\begin{align}\label{est-en-high}
	\sum_{k \in \Lambda_u   } \int_{\T^3} a_{(k)}^2 P_{\neq 0}( \left|W_{(k)}\right|^2) \d x
    & \lesssim \sum_{k \in \Lambda_u   } \sigma^{-1}\|\nabla (a_{(k)}^2)\|_{C_{T,x}} \|P_{\neq 0}( \left|W_{(k)}\right|^2)\|_{C_TL^1_x} \notag \\ 
    &\lesssim \sigma^{-1}\ell_q^{-44}
    \leq \frac{1}{48} \delta_{q+ 3} .
\end{align}
By \eqref{def-c8}, \eqref{rubl1s} and \eqref{gl1-est}, we get
\begin{align}\label{prin-end}
	6\ve_u^{-1}( \ell_q+  \|\mathring{R}_{\ell_q}^u(t)+\mathring{G}(t)\|_{L^1_x})\leq 6\ve_u^{-1}(  \ell_q+ c_*\delta_{q+2}+2\ve_B^{-1}M_*^2|\Lambda_{B}| (\ell_q+ c_*\delta_{q+ 2} )) \leq \frac{1}{96} \delta_{q+ 2},
\end{align}
where in the last step we used the smallness \eqref{def-c8}. Hence, combining \eqref{est-en-high} and \eqref{prin-end} altogether we obtain
\begin{align}\label{est-j2}
	|\delta e_{32}| \leq \frac{1}{48} \delta_{q+2}.
\end{align}

Regarding the first term $\delta e_{31}$ on the right-hand side of \eqref{ver-en-diff-3}, since $\chi_{q+1}\equiv 1$ 
on $[\overline{T}_{q+1},S_{q+1}]$, by the iterative estimate \eqref{ubc} and the standard mollification estimates,
\begin{align}\label{est-moll-uqzq}
	( \|u_q\|_{L^2}^2)*_{t} \varphi_{\ell_q}(t)- \|u_q(t)\|_{L^2}^2 
	&\lesssim \ell_q \|u_q\|_{C^{1}_{[t-\ell_q,t]}L^2_x}\|u_q\|_{C_{[t-\ell_q,t]}L^2_x}\lesssim \ell_q \la^4 \sum\limits_{n=0}^q \delta_{n}^{\frac12}   \leq \frac{1}{144} \delta_{q+ 3},
\end{align}
and similarly,
\begin{align}\label{est-moll-bqzq}
(\|B_q\|_{L^2}^2)*_{t} \varphi_{\ell_q}(t)- \|B_q(t)\|_{L^2}^2 \leq \frac{1}{144} \delta_{q+ 3}.
\end{align}
Using standard mollification estimates and the growth estimate \eqref{bdd-e-h-c1}, we derive
\begin{align}\label{est-e}
	|   e(t)- e*_t \varphi_{\ell_q}(t)|\lesssim \ell_q \| e\|_{C_{[T_{q+2},S_{q+2}]}^1}\leq \frac{1}{144} \delta_{q+ 3},
\end{align}
where we also choose $a$ sufficiently large in the last step. Thus, taking into account $f_q^{(e)}(t)- (f_q^{(e)})*_{t}\varphi_{\ell_q}(t)\equiv 0$  on $[\overline{T}_{q+1},S_{q+1}]$ and using \eqref{est-moll-uqzq} and \eqref{est-e} together we get 
\begin{align}\label{est-j1}
	|\delta e_{31}|\leq \frac{1}{48}\delta_{q+ 3}.
\end{align}

Thus, putting estimates \eqref{ver-en-diff}, \eqref{ver-en-i2}, \eqref{ver-en-i3}, \eqref{est-j2} and \eqref{est-j1} altogether we obtain 
\begin{align*}
| e(t) -
(\|u_{q+1}(t)\|_{L^2_x}^2+\|B_{q+1}(t)\|_{L^2_x}^2)-\frac14 \delta_{q+2}| \leq \frac{1}{48}\delta_{q+2}+ \frac{1}{16}\delta_{q+3}, 
\quad 
t\in \cI^{\rm p}_1, 
\end{align*}
which yields the energy iterative estimate  \eqref{energy-est} at level $q+1$ on the interval $\cI^{\rm p}_1$.

\medskip 
\paragraph{\bf (ii) The mollified regime $\cI^{\rm p}_2$:}
In the mollified regime where $t\in \mathcal{I}^p_2=[T_{q+1}, \overline{T}_{q+1}]$, by definition,
\begin{align*}
    \chi_{q+1}(t)=1,\quad 0\leq f_q^{(e)}(t)- (f_q^{(e)})*_{t}\varphi_{\ell_q}(t)\leq \frac{1}{4}\delta_{q+2}-\frac34\delta_{q+3}.
\end{align*}
Since estimates \eqref{est-j2}-\eqref{est-e} still hold, we deduce that
\begin{align}\label{i11-t-3}
-\frac{1}{48} \delta_{q+3} 
\leq \delta e_{31} \leq \frac{1}{4} \delta_{q+2},
\quad 
|\delta e_{32}| \leq \frac{1}{48}\delta_{q+2},
\end{align}
which together with \eqref{ver-en-diff}, \eqref{ver-en-i2} and  \eqref{ver-en-i3} 
lead to \eqref{energy-est} at level $q+1$ on the mollified regime $\cI^{\rm p}_2$.

\medskip 
\paragraph{\bf (iii) The remaining regime $\cI^{\rm p}_3$:}
In the remaining temporal regime 
$\cI^{\rm p}_3$, 
we note that there is no perturbation added to the initial step 
at level $q$, hence for $t\in \cI^{\rm p}_3$, $(\rr_q^u(t),\rr_q^B(t))=(\rr_0^u(t),\rr_0^B(t))$ and $ (u_q(t),B_q(t)) = (u_0(t),B_0(t))$.

Concerning the first energy error  $\delta e_{31}$ of \eqref{ver-en-diff-3}, 
by the initial energy conditions  \eqref{asp-e-0} and \eqref{asp-e-q}, 
\begin{align}\label{est-e-134}
	\frac{3}{4}\delta_{q+3}\leq e(t) - (\|u_{q}(t)\|_{L_x^2}^2+\|B_{q}(t)\|_{L_x^2}^2)\leq \delta_{q+2}.
\end{align}
First note that \eqref{est-moll-uqzq}-\eqref{est-e} still hold for $t\in [T_{q+2},S_{q+2}]$. 
When $t\in [\Tm, T_{q+1}) \cup 
(\overline{S}_{q+1}, \Sm]$, 
using \eqref{asp-e-0}, \eqref{asp-e-q} and \eqref{def-fq} we get
\begin{align}\label{j1-2-low}
	\delta e_{31}& = (1-\chi_{q+1}^2)(e - \|u_{q}\|_{L_x^2}^2- \|B_{q}\|_{L_x^2}^2)- \frac34\delta_{q+ 3} +\chi_{q+1}^2 f_q^{(e)}*_{t}  \varphi_{\ell_q}\notag\\
	&\quad+\chi_{q+1}^2( e - \|u_{q}\|_{L_x^2}^2- \|B_{q}\|_{L_x^2}^2 - ( e - \|u_{q}\|_{L_x^2}^2- \|B_{q}\|_{L_x^2}^2)*_{t}  \varphi_{\ell_q})\notag\\
	&\geq(1-\chi_{q+1}^2)\frac34\delta_{q+3} - \frac34\delta_{q+ 3} +\chi_{q+1}^2 \frac34\delta_{q+ 3}- \chi_{q+1}^2\frac{1}{48} \delta_{q+3}\geq -\frac{1}{48} \delta_{q+3},
\end{align}
and
\begin{align}\label{j1-2-up}
	\delta e_{31}
	\leq(1-\chi_{q+1}^2)\delta_{q+2} - \frac34\delta_{q+ 3} +\chi_{q+1}^2 \frac34\delta_{q+ 3}+ \chi_{q+1}^2\frac{1}{48} \delta_{q+3}\leq  \delta_{q+2}.
\end{align}
When $t\in (S_{q+1}, \overline{S}_{q+1}]$, similar argument yields
\begin{align}\label{j1-2-2-up}
    \delta e_{31} \geq (1-\chi_{q+1}^2)\frac34\delta_{q+3} - \frac34\delta_{q+ 3} +\chi_{q+1}^2 \frac34\delta_{q+ 3}- \chi_{q+1}^2\frac{1}{48} \delta_{q+3}\geq -\frac{1}{48} \delta_{q+3}
\end{align}
while for the upper bound, since the mollification procedure may effect the value of $f_q^{(e)}$ in this regime,
we deduce that 
\begin{align}\label{j1-2-2-up-2}
    \delta e_{31} \leq(1-\chi_{q+1}^2)\delta_{q+2} - \frac34\delta_{q+ 3} +\chi_{q+1}^2 \frac14\delta_{q+ 2}+ \chi_{q+1}^2\frac{1}{48} \delta_{q+3}\leq  \delta_{q+2}.
\end{align}

Regarding the estimate of the second component $\delta e_{32}$ of \eqref{ver-en-diff-3}, 
since $(\rr_q^u,\rr_q^B)=(\rr_0^u,\rr_0^B)$ on $\cI^{\rm p}_3$, 
using the initial estimates of the Reynolds and magnetic stresses in \eqref{asp-0} and \eqref{asp-q}, estimate \eqref{gl1-est} of $\mathring{G}^{B}$  
and arguing as in \eqref{est-en-high} 
we get
\begin{align}\label{est-i12-3}
	|\delta e_{32}| &\leq \chi_{q+1}^2\( \sum_{k \in \Lambda_u   } \int_{\T^3} a_{(k)}^2 P_{\neq 0}( \left|W_{(k)}\right|^2) \d x+ 6\ve_u^{-1}( \ell_q+ \|\mathring{R}_{\ell_q}^u(t)+ \mathring{G}^{B}(t)\|_{L^1_x})\)\notag\\
	&\leq \chi_{q+1}^2\(\frac{1}{48}\delta_{q+3} +6\ve_u^{-1}(  \ell_q+ \frac18c_*\delta_{q+3}+2\ve_B^{-1}M_*^2|\Lambda_{B}| (\ell_q+ \frac18c_*\delta_{q+3} ))\)\notag\\
	&\leq \frac{1}{24} \chi_{q+1}^2\delta_{q+3}.
\end{align}

Therefore, combining estimates \eqref{ver-en-diff}, \eqref{ver-en-i2}, \eqref{ver-en-i3}, \eqref{est-j2}, \eqref{j1-2-low}, \eqref{j1-2-up} and  \eqref{est-i12-3} altogether
we arrive at    
\begin{align*} 
  e(t) -
(\|u_{q+1}(t)\|_{L^2_x}^2+\|B_{q+1}(t)\|_{L^2_x}^2)\geq  \frac34 \delta_{q+3}-\frac{1}{48} \delta_{q+3} -\frac{1}{24} \delta_{q+3}  -\frac{1}{24} \delta_{q+3}\geq \frac12 \delta_{q+3}
\end{align*}
and
\begin{align*} 
 e(t) -
(\|u_{q+1}(t)\|_{L^2_x}^2+\|B_{q+1}(t)\|_{L^2_x}^2)\leq\delta_{q+2}+\frac34 \delta_{q+3}-\frac{1}{48} \delta_{q+3} +\frac{1}{24} \delta_{q+3} +\frac{1}{24} \delta_{q+3}  \leq  \frac32\delta_{q+2},
\end{align*}
thereby verifying \eqref{energy-est} at level $q+1$ 
for $t\in \cI^{\rm p}_3$. 
The proof of \eqref{energy-est} is complete.

\subsection{Verification of helicity iterative estimate} 

We keep using the temporal regimes 
$\cI^{\rm p}$, $\cI^{\rm p}_i$, $i=1,2,3$, 
as in the previous energy case. 
Set 
\begin{align*}
	\delta h(t):=  h(t)- \int_{\T^3}  u_{q+1}(t) \cdot B_{q+1}(t)  \,\d x -f_q^{(h)}(t).
\end{align*}

As in Subsection \ref{Subsec-Energy-Verif}, in the non-perturbed regime 
$[T_{q+2}, S_{q+2}]\setminus \cI^{\rm p}$,  
since $\chi_{q+1}=0$, 
by \eqref{asp-e-0} and \eqref{asp-e-q}, 
\begin{align}
  \frac{1}{64}\delta_{q+4} \leq h(t)- \int_{\T^3}  u_{q+1}(t) \cdot B_{q+1}(t)  \,\d x\leq \frac{1}{32} \delta_{q+3}, 
  \quad 
  t\in [T_{q+2}, S_{q+2}]\setminus \cI^{\rm p}. 
\end{align}

Regarding the more delicate perturbed regime $\cI^{\rm p}$, 
by the definitions of $u_{q+1}$ and $B_{q+1}$, 
\begin{align}\label{ver-he-diff}
	\delta h(t)
	& =  \Big(\gamma_q^{(h)} - \int_{\T^3}\wt w_{q+1}^{(h)} (t)\cdot \wt d_{q+1}^{(h)} (t)\, \d x \Big) -  \Big( \int_{\T^3} (\wt w^{(p)}_{q+1}+ \wt w^{(h)}_{q+1}) \cdot\wt  d_{q+1}^{(c)+(H)}+  d_{q+1} \cdot\wt w_{q+1}^{(c)+(H)}\, \d x \Big)\notag\\
 &\quad - \Big(  \int_{\T^3}  u_{q} \cdot d _{q+1}+  B_{q} \cdot w_{q+1} \,\d x\Big).
\end{align}          
We estimate each error of the cross helicity on the right-hand side of \eqref{ver-he-diff} separately. 
First, for the second term on the right-hand side of \eqref{ver-he-diff}, using Lemma~\ref{totalest} and \eqref{Lp-wdp-1} we get
\begin{align}\label{ver-he-2}
	   \|  \wt w^{(p)+(h)}_{q+1}\cdot \wt   d_{q+1}^{(c)+(H)} \|_{C_TL^1_x}
	&\lesssim   \| \wt  w^{(p)+(h)}_{q+1} \|_{C_TL^2_x}  \|  \wt d_{q+1}^{(c)+(H)} \|_{ C_TL^2_x} \notag\\
    &\lesssim   \delta_{q+1}^{\frac{1}{2}}(\ell_q^{-2} r_{\perp} r_{\parallel}^{-1}
   +  \ell_q^{-4}  \mu^{-1} (\rs\rp\rw)^{-\frac12}+ \ell_q^{-4}  \lambda^{-3\ve}_{q+1} )\notag\\
   &\leq 10^{-5}\delta_{q+4},
\end{align}
and also 
\begin{align}\label{ver-he-2-2}
\| d_{q+1}\cdot \wt w_{q+1}^{(c)+(H)}\|_{C_TL^1_x} \leq 10^{-5}\delta_{q+4}.
\end{align}

Regarding the last term on the right-hand side of \eqref{ver-he-diff}, by the estimates of perturbations in Lemma~\ref{totalest},
\begin{align}\label{ver-he-3}
	 \|  u_{q} \cdot d_{q+1}\|_{C_{\cI^{\rm p}}L^1_x} 
	&\lesssim \|u_{q} \|_{ C_{[\Tm,T]}C_x}\| d_{q+1} \|_{ C_TL^1_x } \notag \\ 
    &\lesssim \la^4\ell_q^{-4}\lambda_{q+1}^{-3\ve} \leq  10^{-5}\delta_{q+4},
\end{align}
and
\begin{align}\label{ver-he-3-2}
	 \| w_{q+1} \cdot B_{q} \|_{C_{\cI^{\rm p}}L^1_x} \leq 10^{-5}\delta_{q+4}.
\end{align}

It remains to treat the first term on the right-hand side of \eqref{ver-he-diff}.  We note that 
\begin{align}
    \gamma_q^{(h)}(t) - \int_{\T^3}\wt w_{q+1}^{(h)} (t)\cdot \wt d_{q+1}^{(h)} (t)\, \d x =  \gamma_q^{(h)}(t) -\chi_{q+1}^2(t)\gamma_{\ell_q}^{(h)}(t).
\end{align}
As in the previous energy case, 
we divide $\cI^{\rm p}$ into three subregimes 
$\cI^{\rm p} = \bigcup\limits_{i=1}^3 \cI^{\rm p}_i$, 
where $\cI^{\rm p}_i$ is the time interval as in Subsection \ref{Subsec-Energy-Verif}, 
$i=1,2,3$.

In the intermidiate regime where $t\in 
\cI^{\rm p}_1=[\overline{T}_{q+1},S_{q+1}]$, one has $\chi_{q+1}(t)\equiv 1$ and $f_q^{(h)}(t)- (f_q^{(h)})*_{t}\varphi_{\ell_q}(t)\equiv 0$. Using the iterative estimates \eqref{bdd-e-h-c1}, \eqref{ubc} and choosing $a$ sufficiently large, we have
\begin{align}\label{est-h-l}
   | h*_t \varphi_{\ell_q}(t)-h(t)|\lesssim  \ell_q \|h\|_{C_{[T_{q+2},S_{q+2}]}^1} \leq 10^{-5}\delta_{q+4}
\end{align}
and 
\begin{align}\label{est-ub-l}
  \Big|\Big( \int_{\T^3}  u_{q } \cdot B_{q }   \,\d x\Big)*_t \varphi_{\ell_q}(t)- \int_{\T^3}u_{q }(t)\cdot B_{q }(t)\,\d x \Big|
  & \lesssim \ell_q  \|u_{q } \|_{ C^{1}_{[t-\ell_q,t]}L_x^2} \| B_{q }  \|_{ C^{1}_{[t-\ell_q,t]}L_x^2} \notag \\ 
  & \lesssim \ell_q \la^8 \leq 10^{-5}\delta_{q+4}.
\end{align} 
Hence, 
\begin{align}\label{ver-he-1}
 \Big|\gamma_q^{(h)} - \int_{\T^3} \wt w_{q+1}^{(h)} (t)\cdot  \wt d_{q+1}^{(h)} (t)\, \d x\Big| \leq 10^{-4}\delta_{q+4},
\end{align}
Therefore, plugging \eqref{ver-he-2}-\eqref{ver-he-3-2} and \eqref{ver-he-1} into \eqref{ver-he-diff} we get \eqref{helicity-est} for $t\in \cI^{\rm p}_1$ 
at level $q+1$. 

In the mollified temporal regime where $t\in \cI^{\rm p}_2=[T_{q+1}, \overline{T}_{q+1}]$, one has 
\begin{align*}
    \chi_{q+1}(t)=1,\quad \frac{3}{256}\delta_{q+4}-\frac{1}{256}\delta_{q+3}\leq (f_q^{(h)})*_{t}\varphi_{\ell_q}(t)-f_q^{(h)}(t) \leq 0.
\end{align*}
Since estimates \eqref{est-h-l} and \eqref{est-ub-l} still hold, we infer that 
\begin{align}\label{he-end-2}
 \frac{3}{256}\delta_{q+4}-\frac{1}{256}\delta_{q+3} -2*10^{-5}\delta_{q+4} \leq \gamma_q^{(h)} - \int_{\T^3} \wt w_{q+1}^{(h)} (t)\cdot \wt d_{q+1}^{(h)} (t)\, \d x \leq 2*10^{-5}\delta_{q+4}.
\end{align}
Thus, plugging \eqref{ver-he-3-2}-\eqref{ver-he-1} and \eqref{he-end-2} into \eqref{ver-he-diff} we obtain \eqref{helicity-est} at level $q+1$ for $t\in \cI^{\rm p}_2$.

At last, in the remaining temporal regime  where $t\in \cI\setminus (\cI^{\rm p}_1 \cup \cI^{\rm p}_2)$, 
$\chi_{n}(t)\equiv0$ for all $n\leq q$, hence $(u_q(t),B_q(t)) = (u_0(t),B_0(t))$. 
It follows that 
\begin{align*}
    \gamma_q^{(h)} - \int_{\T^3} \wt w_{q+1}^{(h)} (t)\cdot  \wt d_{q+1}^{(h)} (t)\, \d x&=(1-\chi^2_{q+1}) (h(t)-\int_{\T^3} u_0(t)\cdot B_0(t)\,\d x)-f_q^{(h)}(t)+ \chi_{q+1}^2 f_q^{(h)}*_{t}  \varphi_{\ell_q}(t) \notag\\
    &\quad +  \chi^2_{q+1} \Big(h(t)-\int_{\T^3} u_0(t)\cdot B_0(t)\,\d x- (h-\int_{\T^3}  u_{0} \cdot B_{0}   \,\d x)*_t \varphi_{\ell_q}(t) \Big).
\end{align*}
Using \eqref{asp-e-0}, \eqref{asp-e-q}, \eqref{est-h-l} and \eqref{est-ub-l},  we get
\begin{align}\label{est-i1-3}
   -10^{-4} \delta_{q+4}\leq  \gamma_q^{(h)} - \int_{\T^3} \wt w_{q+1}^{(h)} (t)\cdot  \wt d_{q+1}^{(h)} (t)\, \d x \leq \frac{1}{32} \delta_{q+3}.
\end{align}
Thus, plugging \eqref{ver-he-3-2}-\eqref{ver-he-1} and \eqref{est-i1-3} into \eqref{ver-he-diff} we finally obtain \eqref{helicity-est} 
at level $q+1$ for $t\in \cI\setminus (\cI^{\rm p}_1 \cup \cI^{\rm p}_2)$ 
and finish the proof of \eqref{helicity-est}.

\section{Reynolds and magnetic stresses}   \label{Sec-Rey-mag-stress}
In this section, we treat the Reynolds and magnetic stresses
and prove the corresponding iterative estimates in Proposition \ref{Prop-Iterat}.

\subsection{Decomposition}

\subsubsection{Magnetic stress} 
Using the related MHD-Reynolds system \eqref{mhd1} with $q+1$ replacing $q$, 
the new perturbations \eqref{perturbation} and \eqref{q+1 iterate}
we derive the equation for the magnetic stress $\rr^B_{q+1}$:
\begin{align}
&\displaystyle\div\mathring{R}_{q+1}^B   \notag\\
&=\underbrace{\chi_{q+1}\partial_t d^{(p)+(c)+(h)}_{q+ 1} -\nu_2\Delta (d_{q+1}-\wt d_{q+1}^{(H)} ) +\div  \big( B_{q} \otimes w_{q+1} + d_{q + 1} \otimes u_{q} -u_{q} \otimes d_{q+1} - w_{q +1} \otimes B_{q}\big) }_{ \div\mathring{R}_{lin}^B  }   \notag\\
&\quad+\underbrace{\div ( \wt d_{q+ 1}^{(p)} \otimes  \wt w_{q+1}^{(p)} - \wt w_{q+ 1}^{(p)} \otimes  \wt d_{q+ 1}^{(p)} + \chi_{q+1}^2 \mathring{R}_{q}^B)+  \chi_{q+1}^2 (\partial_t-\nu_2\Delta  ) d_{q+1}^{(H)}  }_{\div\mathring{R}_{osc}^B }  \notag\\
&\quad+\underbrace{\div\Big(   \wt d_{q+1}^{(p)+(h)} \otimes \wt  w_{q+1}^{(c)+(H)} -  \wt w_{q+1}^{(c)+(H)} \otimes  d_{q+1}  + \wt  d_{q+1}^{(c)+(H)}\otimes w_{q+1}-   \wt w_{q+1}^{(p)+(h)}\otimes\wt  d_{q+1}^{(c)+(H)}\Big) }_{\div\mathring{R}_{cor}^B } \notag   \\
&\quad +  \underbrace{  \partial_t\chi_{q+1}d^{(p)+(c)+(h)}_{q+ 1}+\partial_t(\chi_{q+1}^2)d^{(H)}_{q+ 1} +  (1-\chi_{q+1}^2)\div \mathring{R}_{q}^B }_{\div\mathring{R}_{cut}^B} .  \label{rb}
\end{align}
Then, using the inverse divergence operator $\mathcal{R}^B$ in \eqref{operarb} below
we can define the magnetic stress at level $q+1$ by
\begin{align}\label{rbcom}
	\mathring{R}_{q+1}^B := \mathring{R}_{lin}^B +   \mathring{R}_{osc}^B+ \mathring{R}_{cor}^B+\mathring{R}_{cut}^B,
\end{align}
where
\begin{align}
	\mathring{R}_{lin}^B
	:= &  \mathcal{R}^B \chi_{q+1} \partial_t d^{(p)+(c)+(h)}_{q+ 1} -\nu_2\mathcal{R}^B\Delta (d_{q+1}-\wt d_{q+1}^{(H)} )   + B_{q}\otimes w_{q+1} + d_{q + 1} \otimes u_{q} -u_{q}\otimes d_{q+1} - w_{q +1} \otimes B_{q},\label{rbp}
\end{align}
the oscillation error
\begin{align} \label{rob}
	\mathring{R}_{osc}^B &:= \chi_{q+1}^2  \sum_{k \in \Lambda_B}\mathcal{R}^B\P_{H}\P_{\neq 0}\left ( \P_{\neq 0} (D_{(k)}\otimes W_{(k)}-W_{(k)}\otimes D_{(k)} )\nabla (a_{(k)}^2)\right)\notag\\
	&\quad + \chi_{q+1}^2 \sum_{k \in \Lambda_B} \mathcal{R}^B \P_H \P_{\neq 0}
\( (\partial_{t}-\nu_2\Delta)( a_{(k)}^{2}) \int_{0}^t e^{(t-\tau)\Delta}\div(D_{(k)} \otimes W_{(k)}-W_{(k)} \otimes D_{(k)} )\d \tau \)   \nonumber  \\
	&\quad - 2\nu_2\chi_{q+1}^2 \sum_{k \in \Lambda_B} \mathcal{R}^B \P_H \P_{\neq 0} \sum_{i=1}^3
\( \partial_{i}( a_{(k)}^{2})  \partial_{i}\int_{0}^t e^{(t-\tau)\Delta}\div(D_{(k)} \otimes W_{(k)}-W_{(k)} \otimes D_{(k)} ) \d \tau \)\notag\\
&\quad +\chi_{q+1}^2(\mathring{R}_{q}^B-\mathring{R}_{\ell_q}^B) ,
\end{align}
the corrector error
\begin{align}\label{rbp2}
	\mathring{R}_{cor}^B := &\, \wt  d_{q+1}^{(p)+(h)} \otimes \wt w_{q+1}^{(c)+(H)} - \wt   w_{q+1}^{(c)+(H)} \otimes d_{q+1}  + \wt  d_{q+1}^{(c)+(H)}\otimes w_{q+1}-  \wt   w_{q+1}^{(p)+(h)} \otimes \wt d_{q+1}^{(c)+(H)},
\end{align}
and the cut-off error
\begin{align} \label{RBcut-def}
	\mathring{R}_{cut}^B
	& =  \mathcal{R}^B \( \partial_t\chi_{q+1}d^{(p)+(c)+(h)}_{q+ 1}+\partial_t(\chi_{q+1}^2)d^{(H)}_{q+ 1} \) +  (1-\chi_{q+1}^2)  \mathring{R}_{q}^B.
\end{align}

\subsubsection{Reynolds stress}

Concerning the Reynolds stress we compute
\begin{align}
		&\displaystyle\div\mathring{R}_{q+1}^u - \nabla P_{q+1}  \notag\\
		&\displaystyle = \underbrace{\chi_{q+1}\partial_t w^{(p)+(c)+(h)}_{q+ 1} -\nu_1 \Delta ( w_{q+1}-\wt w_{q+1}^{(H)}) +\div \big( u_{q}  \otimes w_{q+1} + w_{q+ 1} \otimes  u_{q}  - B_{q} \otimes d_{q+1} - d_{q+1} \otimes B_{q}\big) }_{ \div\mathring{R}_{lin}^u +\nabla P_{lin} }   \notag\\
		&\displaystyle\quad+ \underbrace{\div ( \wt w_{q+1}^{(p)} \otimes   \wt w_{q+1}^{(p)} -  \wt d_{q+1}^{(p)} \otimes  \wt d_{q+1}^{(p)} + \chi_{q+1}^2 \mathring{R}_{q}^u)+ \chi_{q+1}^2 (\partial_t-\nu_1\Delta)  w_{q+1}^{(H)}}_{\div\mathring{R}_{osc}^u +\nabla P_{osc}}  \notag\\
		&\quad+\underbrace{\div\Big(\wt w_{q+1}^{(c)+(H)}\otimes w_{q+1}+   \wt w_{q+1}^{(p)+(h)}\otimes \wt w_{q+1}^{(c)+(H)} - \wt  d_{q+1}^{(c)+(H)}\otimes d_{q+1}-  \wt d_{q+1}^{(p)+(h)} \otimes \wt d_{q+1}^{(c)+(H)}\Big)}_{\div\mathring{R}_{cor}^u +\nabla P_{cor}}\notag\\
		&\quad +  \underbrace{ \partial_t\chi_{q+1}w^{(p)+(c)+(h)}_{q+ 1}+\partial_t(\chi_{q+1}^2)w^{(H)}_{q+ 1}+  (1-\chi_{q+1}^2)\div \mathring{R}_{q}^u }_{\div\mathring{R}_{cut}^u +\nabla P_{cut}} . \label{ru}
\end{align}

Then, using the inverse divergence operator $\mathcal{R}^u$  in \eqref{operaru}
we choose the Reynolds stress at level $q+1$
\begin{align}\label{rucom}
	\mathring{R}_{q+1}^u := \mathring{R}_{lin}^u +   \mathring{R}_{osc}^u+ \mathring{R}_{cor}^u+\mathring{R}_{cut}^u, 
\end{align}  
which consists of four components: 
the linear error 
\begin{align}
	\mathring{R}_{lin}^u & := \mathcal{R}^u \chi_{q+1}\partial_t w^{(p)+(c)+(h)}_{q+ 1} -\nu_1  \mathcal{R}^u\Delta ( w_{q+1}- \wt w_{q+1}^{(H)}) +u_{q} \mathring\otimes w_{q+1} + w_{q+ 1} \mathring\otimes u_{q}- B_{q} \mathring\otimes d_{q+1}
	-d_{q+1} \mathring\otimes B_{q}, \label{rup}
\end{align}
the oscillation error
\begin{align}\label{rou}
\mathring{R}_{osc}^u &:= \chi_{q+1}^2 \sum_{k \in \Lambda_u\cup\Lambda_B} \mathcal{R}^u \P_H\P_{\neq 0}\left( \P_{\neq 0}(W_{(k)}\otimes W_{(k)}-D_{(k)}\otimes D_{(k)})\nabla (a_{(k)}^2)\right)\notag\\
&\quad + \chi_{q+1}^2 \sum_{k \in \Lambda_u\cup\Lambda_B} \mathcal{R}^u \P_H \P_{\neq 0}
	\( (\partial_{t}-\nu_1\Delta)( a_{(k)}^{2}) \int_{0}^t e^{(t-\tau)\Delta}\div (W_{(k)} \otimes W_{(k)}- D_{(k)} \otimes D_{(k)})\d \tau \)   \nonumber  \\
	&\quad -2\nu_1 \chi_{q+1}^2 \sum_{k \in \Lambda_u\cup\Lambda_B} \mathcal{R}^u\P_H \P_{\neq 0} \sum_{i=1}^3
	\( \partial_{i}( a_{(k)}^{2})  \partial_{i}\int_{0}^t e^{(t-\tau)\Delta}\div (W_{(k)} \otimes W_{(k)}- D_{(k)} \otimes D_{(k)}) \d \tau \) \notag\\
	&\quad +\chi_{q+1}^2(\mathring{R}_{q}^u-\mathring{R}_{\ell_q}^u) ,
\end{align}
the corrector error
\begin{align}
	\mathring{R}_{cor}^u &
	:= \wt w_{q+1}^{(c)+(H)}\mathring\otimes w_{q+1}+ \wt w_{q+1}^{(p)+(h)}\mathring\otimes \wt w_{q+1}^{(c)+(H)} - \wt d_{q+1}^{(c)+(H)}\mathring\otimes d_{q+1}- \wt d_{q+1}^{(p)+(h)} \mathring\otimes \wt d_{q+1}^{(c)+(H)}, \label{rup2}
\end{align}
and the cut-off error
\begin{align} \label{Rucut-def}
	\mathring{R}_{cut}^u
	& =  \mathcal{R}^u\( \partial_t\chi_{q+1} w^{(p)+(c)+(h)}_{q+ 1} +\partial_t(\chi_{q+1}^2)w^{(H)}_{q+ 1} \) +  (1-\chi_{q+1}^2)  \mathring{R}_{q}^u.
\end{align}

\begin{remark}
    We note that the heat correctors contribute to extra two oscillation errors of the Reynolds and magnetic stresses.
\end{remark}

\subsection{Verification of decay estimate}

The purpose of this subsection is to verify the decay estimate \eqref{rubl1s} 
for the Reynolds and magnetic stresses 
at level $q+1$. Since the Calder\'{o}n-Zygmund operators are bounded in the space $L^p_x$ for $1<p<+\9$,
we choose
\begin{align}\label{def-p}
p: =\frac{2+40\varepsilon}{2+39\varepsilon}\in (1,2),
\end{align}
while $\ve$ given by \eqref{b-beta-ve}.
In particular, $(2+40\varepsilon)(1-\frac{1}{p})=\ve$, which, via \eqref{larsrp}, yields that
\begin{align}  \label{rs-rp-p-ve}
 (\rp\rs\rw)^{\frac 1p-1} = \lambda^{\varepsilon},
   \quad 
   (\rp\rs\rw)^{\frac 1p-\frac 12} = \lambda^{-1-19\varepsilon}.
\end{align}

Note that,  
in the non-perturbed case where $t\in 
[0, T]\setminus \cI^{\rm p}$,
by the telescoping procedure, 
no perturbations are added on this time interval  
and so, 
$\mathring{R}_{lin}^B=\mathring{R}_{osc}^B =\mathring{R}_{cor}^B=0$ and $\mathring{R}_{cut}^B=\mathring{R}_{q}^B$.
Thus, using \eqref{asp-0} and \eqref{asp-q}
we get
\begin{align}\label{rq1b-3}
	\|\mathring{R}_{q+1}^B \|_{ C_{[0, T]\setminus \cI^{\rm p}} L^1_{x}}
    =\|\mathring{R}_{cut}^B   
    \|_{C_{[0, T]\setminus \cI^{\rm p}} L^1_x}
    = \|\mathring{R}_{q}^B   \|_{ C_{[0, T]\setminus \cI^{\rm p}} L^1_x}  \leq \frac18 c_*\delta_{q+ 3}. 
\end{align}   
The Reynolds stress can be estimated in the same way. 
This verifies the decay estimate 
\eqref{rubl1s} on the non-perturbed time interval. 

Thus, in the following we only need to consider the more difficult perturbed  regime $\cI^{\rm p}=[\Tm, \Sm]$.

\subsubsection{Decay estimate of magnetic stress}

Below we consider the perturbed regime 
where $t\in \cI^{\rm p}$. 
Let us consider the four components of 
the magnetic error separately as follows: 

\medskip 
\paragraph{\bf $(i)$ Magnetic linear error.}

Note that,
by the identity \eqref{div free magnetic},
\begin{align}\label{r-lin-t}
 \| \cq\mathcal{R}^B\partial_t( d_{q+1}^{(p)+(h)+(c)}) \|_{C_{\cI^{\rm p}}L_x^p}
	\lesssim&\ \sum_{k \in \Lambda_B}\| \mathcal{R}^B \curl\partial_t( a_{(k)} D^c_{(k)})   \|_{C_TL_x^p} + \sum_{k \in \Lambda_h}\| \mathcal{R}^B \curl\partial_t( (\gamma^{(h)}_{\ell_q })^{\frac12}  W^c_{(k)})   \|_{C_TL_x^p} \nonumber \\
	\lesssim&\ \sum_{k \in \Lambda_B}\(\|  a_{(k)} \|_{C_{T,x}^1}\| D^c_{(k)} \|_{C_T L^{p}_x }
            +\|  a_{(k)} \|_{C_{T,x} }\| \p_t D^c_{(k)} \|_{C_T L^{p}_x }\) \notag\\
       &\    +   \sum_{k \in \Lambda_h}\(\| (\gamma^{(h)}_{\ell_q })^{\frac12} \|_{C_{T,x}^1}\| W^c_{(k)} \|_{C_T L^{p}_x }
            +\| (\gamma^{(h)}_{\ell_q })^{\frac12} \|_{C_{T,x} }\| \p_t W^c_{(k)} \|_{C_T L^{p}_x }\) .
\end{align}
Then, by Lemmas  \ref{buildingblockestlemma}, \ref{mae} and \eqref{rs-rp-p-ve},
\begin{align}  \label{mag time derivative}
	 \| \cq \mathcal{R}^B\partial_t(  d_{q+1}^{(p)+(h)+(c)}) \|_{C_{\cI^{\rm p}}L_x^p}
    \lesssim&\ \sigma^{-1}\ell_q^{-13} \rs(\rp\rs\rw)^{\frac 1p-\frac 12}
        +\ell_q^{-2} \rs\rp^{-1}\mu(\rp\rs\rw)^{\frac 1p-\frac 12}\notag\\
        &\ +\sigma^{-1}\ell_q^{-2}\rs(\rp\rs\rw)^{\frac 1p-\frac 12}
        +\rs\rp^{-1}\mu(\rp\rs\rw)^{\frac 1p-\frac 12}\notag\\
	\lesssim&\   \ell_q^{-2}\lambda^{-\ve}.
\end{align}

The control of the viscosity term requires the strong spatial intermittency of the magnetic flows, 
which actually motivates the constructions in Section \ref{Sec-Interm-Flow} 
and can be provided by Lemma \ref{totalest}. 
More precisely, an application of Lemma \ref{totalest} gives
\begin{align}\label{mag viscosity}
&\quad 	\norm{-\nu_2\mathcal{R}^B\Delta(d_{q+1}- \wt d_{q+1}^{(H)} ) }_{C_{\cI^{\rm p}}L^p_x} \notag\\
	& \lesssim \norm{ \mathcal{R}^B(-\Delta) \wt d_{q+1}^{(p)} }_{C_TL^p_x}+\norm{ \mathcal{R}^B(-\Delta) \wt d_{q+1}^{(h)} }_{C_TL^p_x} +\norm{ \mathcal{R}^B(-\Delta) \wt d_{q+1}^{(c)} }_{C_TL^p_x}  \notag\\
&\lesssim\ell_q^{-2}(\sigma \rs ^{-1}) (\rp \rs \rw )^{\frac{1}{p}-\frac12}+ (\sigma \rs ^{-1}) (\rp \rs \rw )^{\frac{1}{p}-\frac12} +\ell_q^{-2}(\sigma \rs ^{-1}) \rs \rp ^{-1}(\rp \rs \rw )^{\frac{1}{p}-\frac12}\notag\\
 &\lesssim\ell_q^{-2}\lambda^{-\frac18+5\ve}+ \lambda^{-\frac18+5\ve}+\ell_q^{-2}\lambda^{-\frac18+\ve}\notag\\
 &\lesssim \ell_q^{-2}\lambda^{-\frac18+5\ve} .
\end{align}

Regarding the remaining nonlinear terms in \eqref{rbp}, by \eqref{ubc} one has 
\begin{align} \label{magnetic linear estimate1}
	&\norm{   B_{q}  \otimes w_{q+1} + d_{q + 1} \otimes  u_{q} 
 - u_{q}  \otimes d_{q+1} - w_{q +1} \otimes  B_{q} }_{C_{\cI^{\rm p}}L^p_x}  \nonumber \\	
	\lesssim\,& 
   \norm{u_{q} }_{C_{\cI^{\rm p},x}} \norm{d_{q+1}}_{C_TL^p_x} +\norm{B_{q} }_{C_{\cI^{\rm p},x} }\norm{w_{q+1}}_{C_TL^p_x} \nonumber \\
	\lesssim\, & \ell_q^{-4} \lambda_q^4 (\lambda^{-1-19\ve}+\lambda^{-1-25\ve}+\lambda^{-2\ve} )\notag\\
    \lesssim\, & \ell_q^{-4}\lambda_q^4\lambda^{-2\ve}.
\end{align}

Therefore, combining \eqref{mag time derivative}-\eqref{magnetic linear estimate1} together
and using  \eqref{b-beta-ve} we obtain the acceptable order 
\begin{align}   \label{magnetic linear estimate}
	\norm{\mathring{R}_{lin}^B }_{C_{\cI^{\rm p}}L^p_x}
     & \lesssim  \ell_q^{-2}\lambda^{-\ve}.
\end{align}

\paragraph{\bf $(ii)$ Magnetic oscillation error.}
In order to treat the delicate magnetic oscillation,
we further decompose it into three components: 
\begin{align*}
	\mathring{R}_{osc}^B = \mathring{R}_{osc.1}^B +  \mathring{R}_{osc.2}^B+  \mathring{R}_{osc.3}^B,
\end{align*}
where $\mathring{R}_{osc.1}^B$ contains the 
high-low spatial  oscillations
\begin{align*}
	\mathring{R}_{osc.1}^B
	&:=   \sum_{k \in \Lambda_B}\mathcal{R}^B \P_{H}\P_{\neq 0}\left(  \P_{\neq 0}(D_{(k)}\otimes W_{(k)}-W_{(k)}\otimes D_{(k)})\nabla (a_{(k)}^2) \right),
\end{align*}
$\mathring{R}_{osc.2}^B$ contains the oscillation error caused by the heat correctors after balancing the high 
oscillations:
\begin{align}\label{def-rob3}
	\mathring{R}_{osc.3}^B &
:=\chi_{q+1}^2 \sum_{k \in \Lambda_B} \mathcal{R}^B \P_H \P_{\neq 0}
\( (\partial_{t}-\nu_2\Delta)( a_{(k)}^{2}) \int_{0}^t e^{(t-\tau)\Delta}\div(D_{(k)} \otimes W_{(k)}-W_{(k)} \otimes D_{(k)} )\d \tau \)   \nonumber  \\
&\quad -2\nu_2 \chi_{q+1}^2 \sum_{k \in \Lambda_B} \mathcal{R}^B \P_H \P_{\neq 0} \sum_{i=1}^3
\( \partial_{i}( a_{(k)}^{2})  \partial_{i}\int_{0}^t e^{(t-\tau)\Delta}\div(D_{(k)} \otimes W_{(k)}-W_{(k)} \otimes D_{(k)} ) \d \tau \),
\end{align}
and  $\mathring{R}_{osc.3}^B $ contains the mollification error
\begin{align*}
	\mathring{R}_{osc.3}^B
	&:= \cq^2(\mathring{R}_{q}^B-\mathring{R}_{\ell_q}^B).
\end{align*}

In order to estimate the high-low oscillation component $\mathring{R}_{osc.1}^B$, since
\begin{align*}
   \P_{\not=0} (D_{(k)} \otimes W_{(k)} - W_{(k)} \otimes D_{(k)})
    = \P_{\geq \frac 12 \sigma} (D_{(k)} \otimes W_{(k)} - W_{(k)} \otimes D_{(k)}),
\end{align*}
applying the stationary lemma~\ref{commutator estimate1}
with $a = \nabla (a_{(k)}^2)$ and $f =  \psi_{(k_1)}^2\phi_{(k)}^2\phi_{(k_2)}^2$
we get
\begin{align}  \label{I1-esti}
	\norm{\mathring{R}_{osc.1}^B }_{C_{\cI^{\rm p}}L^p_x}
	& \lesssim  \sum_{ k \in \Lambda_B} \sigma^{-1} \norm{ \na^3(a^2_{(k)})}_{C_{T,x}}
        \norm{\psi_{(k_1)}^2\phi_{(k)}^2\phi_{(k_2)}^2}_{C_TL^p_x }  \nonumber  \\
	& \lesssim  \ell_q^{-32} \sigma^{-1} \norm{ \psi^2_{(k_1)}}_{C_TL^p_x } \norm{\phi^2_{(k)} }_{C_TL^p_x }  \norm{\phi^2_{(k_2)} }_{C_TL^p_x } \nonumber  \\
	& \lesssim  \ell_q^{-32}\sigma^{-1}  (\rp\rs\rw)^{\frac{1}{p}-1}\lesssim \ell_q^{-32}\lambda^{-3\ve},
\end{align}
where we also used the amplitude estimates in  Lemmas \ref{buildingblockestlemma} and \ref{mae}.

Next, for the oscillation error $\mathring{R}_{osc.2}^B$ caused by the heat corrector, we derive
\begin{align}\label{I3-est}
	\norm{\mathring{R}_{osc.2}^B}_{C_{\cI^{\rm p}}L^p_x}
	& \lesssim \|\sum_{k \in \Lambda_B} \mathcal{R}^B \P_H \P_{\neq 0}
	\( (\partial_{t}-\nu_2\Delta)( a_{(k)}^{2}) \int_{0}^t e^{(t-\tau)\Delta}\div(D_{(k)} \otimes W_{(k)}-W_{(k)} \otimes D_{(k)} )\d \tau \)  \|_{C_TL^p_x} \nonumber  \\
	&\quad +\|\sum_{k \in \Lambda_B} \mathcal{R}^B \P_H \P_{\neq 0} \sum_{i=1}^3
	\( \partial_{i}( a_{(k)}^{2})  \partial_{i}\int_{0}^t e^{(t-\tau)\Delta}\div(D_{(k)} \otimes W_{(k)}-W_{(k)} \otimes D_{(k)} ) \d \tau \) \|_{C_T L^p_x} \nonumber  \\
		&=:J_1^B+J_2^B.
\end{align}
We shall estimate the two terms on the right-hand side of \eqref{I3-est} separately.

Concerning the first term on the right-hand side of \eqref{I3-est}, by Lemmas~\ref{buildingblockestlemma} and \ref{mae},
\begin{align}\label{I3-est-1}
J_1^B	&\lesssim  \sum_{k\in \Lambda_B}\norm{    a_{(k)}^2}_{C_{T,x}^2} \norm{  \Delta^{-1} \div(\phi_{(k )}^2 \psi_{(k_1)}^2 \phi_{(k_2)}^2 )}_{C_TW^{\ve,p} }\notag\\
	&\lesssim \sigma^{-1+\ve}\ell_q^{-26} (\rp\rs\rw)^{\frac 1p-1} \lesssim \ell_q^{-26}\lambda^{-2\ve}.
\end{align}
where the heat kernel estimate as in \eqref{vh-est-1} was also used in the last step.

For the second term on the right-hand side of \eqref{I3-est}, by the Leibnitz rule,
\begin{align}
	J_2^B &= \| \sum_{k \in \Lambda_B} \mathcal{R}^B \P_H \P_{\neq 0} \sum_{i=1}^3
	 \partial_{i} \( \partial_{i}( a_{(k)}^{2}) \int_{0}^t e^{(t-\tau)\Delta}\div(D_{(k)} \otimes W_{(k)}-W_{(k)} \otimes D_{(k)} ) \d \tau \)\|_{C_T L^p_x}\notag\\
	 & \quad+ \| \sum_{k \in \Lambda_B} \mathcal{R}^B \P_H \P_{\neq 0}
	\(  \Delta( a_{(k)}^{2}) \int_{0}^t e^{(t-\tau)\Delta}\div(D_{(k)} \otimes W_{(k)}-W_{(k)} \otimes D_{(k)} ) \d \tau \)\|_{C_T L^p_x}.
\end{align}
Hence, using the heat kernel estimate again we get
\begin{align}\label{I3-est-2}
	J_2^B
	&\lesssim \sum_{k\in \Lambda_B}\norm{    a_{(k)}^2}_{C_{T,x}^2} \norm{ \Delta^{-1} \div (\phi_{(k )}^2 \psi_{(k_1)}^2 \phi_{(k_2)}^2 )}_{C_TW^{\ve,p} } \lesssim \ell_q^{-26}\lambda^{-2\ve}.
\end{align}
Plugging estimates \eqref{I3-est-1}-\eqref{I3-est-2} into \eqref{I3-est}, we come to
\begin{align}\label{I3-est-end}
\norm{\mathring{R}_{osc.2}^B }_{C_{\cI^{\rm p}} L^p_x}\lesssim  \ell_q^{-26}\lambda^{-2\ve}.
\end{align}

Regarding the last mollification error $\mathring{R}_{osc.3}^B$, by 
the iterative estimate \eqref{rbl1b-s} 
we get
\begin{align}  \label{I4-esti}
\norm{\mathring{R}_{osc.3}^B   }_{C_{\cI^{\rm p}}L_x^1}
	&\lesssim \norm{ \cq^2}_{C_{T}}  \norm{(\mathring{R}_{q}^B-\mathring{R}_{\ell_q}^B)}_{C_{\cI^{\rm p}}L_x^1}\notag\\
	&\lesssim  \ell_q  ( \|\mathring{R}_{q}^B \|_{  C^{1}_{[T_{q+2},S_{q+2}]}L^1_x }+
	\| \mathring{R}_{q}^B  \|_{C_{[T_{q+2},S_{q+2}]}W^{1,1}_x} )\notag\\
 &\lesssim \ell_q \lambda_q^8. 
\end{align}

Therefore, putting estimates \eqref{I1-esti}, \eqref{I3-est-end} and \eqref{I4-esti} altogether
and using \eqref{larsrp} and \eqref{rs-rp-p-ve} we arrive at
\begin{align}\label{magnetic oscillation estimate}
\norm{\mathring{R}_{osc}^B }_{ C_{\cI^{\rm p}}L^1_x} \lesssim \ell_q^{\frac12}.
\end{align}

\paragraph{\bf $(iii)$ Magnetic corrector error.}
Using H\"older's inequality,
applying Lemma \ref{totalest} to \eqref{rbp2}
and using  \eqref{Lp-wdp-1} and \eqref{e3.41} we get
\begin{align}\label{magnetic corrector estimate}
	\norm{\mathring{R}_{cor}^B  }_{C_{\cI^{\rm p}} L^{1}_x}
	\lesssim& \norm{\wt  w_{q+1}^{(c)+(H)}}_{C_T L^{2}_x} (\norm{\wt  d^{(p)+(h)}_{q+1}  }_{C_TL^2_{x}} + \norm{d_{q+1} }_{C_TL^2_{x}}) \notag\\
	& +  (\norm{  \wt w_{q+1}^{(p)+(h)} }_{C_TL^2_{x}} + \norm{ w_{q+1} }_{C_T L^2_{x}}) \norm{  \wt d_{q+1}^{(c)+(H)}}_{C_TL^{2}_x}\notag  \\
     \lesssim& \delta_{q+ 1}^{\frac12}   \(\ell_q^{-2}r_{\perp} r_{\parallel}^{-1}+ \ell_q^{-4} \lambda^{-2\ve} \) \notag\\
	\lesssim & \ell_q^{-4} \lambda^{-2\ve} .
\end{align}

Therefore,  combining estimates \eqref{magnetic linear estimate},
\eqref{magnetic oscillation estimate} and
\eqref{magnetic corrector estimate} we obtain
\begin{align} \label{rq1b}
 \| \mathring{R}_{lin}^B  \|_{ C_{\cI^{\rm p}}  L^p_x} +  \| \mathring{R}_{osc}^B \|_{ C_{\cI^{\rm p}} L^1_x}	+  \|\mathring{R}_{cor}^B  \|_{ C_{\cI^{\rm p}} L^{ 1}_x} \lesssim \ell_q^{-2}\lambda^{-\ve}+ \ell_q^{\frac12} +\ell_q^{-4} \lambda^{-2\ve}	\leq  \frac14 c_*\delta_{q+3}.
\end{align}

\paragraph{\bf $(iv)$ Magnetic cut-off error.}  
First, one has $\mathring{R}_{cut}^B=0$ on the subinterval $[T_{q+1},S_{q+1}]$. 
While for $t\in \cI^{\rm p}\setminus [T_{q+1}, S_{q+1}] 
= [\Tm,T_{q+1})\cup (S_{q+1}, \Sm]$,  Lemma~\ref{totalest} yields 
\begin{align}  \label{R-cut-esti-1}
\|\mathring{R}_{cut}^B \|_{C_{\cI^{\rm p}\setminus [T_{q+1}, S_{q+1}]}L^1_x}
	&\lesssim \|\mathcal{R}^B (  \partial_t\chi_{q+1}d^{(p)+(c)+(h)}_{q+ 1}+\partial_t(\chi_{q+1}^2)d^{(H)}_{q+ 1})+ (1-\chi_{q+1}^2 )\mathring{R}^B_{\ell_q} \|_{ C_{\cI^{\rm p}\setminus [T_{q+1}, S_{q+1}]}L^1_x} \notag \\
	&\lesssim \ell_q^{-5} (\lambda^{-1-19\ve}+\lambda^{-1-23\ve}
  + \lambda^{-2\ve})+\| \mathring R_{q}^B\|_{C_{[ T_{q+2} ,T_{q+1}]\cup [ S_{q+1}, S_{q+2}] }L^{1}_{x}} \notag\\
	&\lesssim \ell_q^{-5}\lambda^{-3\ve} + \frac18 c_*\delta_{q+ 3}\leq \frac14 c_*\delta_{q+ 3}.
\end{align}

Finally,  combining  \eqref{rq1b} and \eqref{R-cut-esti-1} together we obtain the desirable estimate of the magnetic stress on the perturbed temporal regime  $\cI^{\rm p}$: 
\begin{align} \label{rq1b-end}
	\|\mathring{R}_{q+1}^B \|_{C_{\cI^{\rm p}}L^1_{x}}
	&\leq \| \mathring{R}_{lin}^B  \|_{C_{\cI^{\rm p}}L^p_x} +  \| \mathring{R}_{osc} ^B \|_{C_{\cI^{\rm p}}L^1_x}
	+  \|\mathring{R}_{cor}^B  \|_{C_{\cI^{\rm p}}L^{ 1}_x}   +\|\mathring{R}_{cut}^B\|_{ C_{\cI^{\rm p}}L^1_x}\leq \frac12c_*\delta_{q+3}.
\end{align}

\subsubsection{Decay estimate of Reynolds stress}

Let us now treat the Reynolds stress $\rr_{q+1}^u$ given by \eqref{rucom}. 
As in the previous magnetic stress case,  
we only need to consider the perturbed temporal regime $t\in \cI^{\rm p}$. 

Let us begin with the estimates of  
the velocity linear error. 

\medskip 
\paragraph{\bf $(i)$ Velocity linear error.}
By the  $L^p$-boundedness of the inverse-divergence operator $\mathcal{R}^u$, the identity 
\eqref{div free magnetic} and Lemmas  \ref{mae}-\ref{totalest},
\begin{align*}
  \|  \mathcal{R}^u\partial_t( w_{q+1}^{(p)+(h)+(c)})\|_{C_{\cI^{\rm p}} L_x^p}
\lesssim& \sigma^{-1}\(\sum_{k \in \Lambda_u\cup\Lambda_{B}}\| \mathcal{R}^u \curl\partial_t( a_{(k)} (W^c_{(k)}))    \|_{ C_TL_x^p} + \sum_{k \in \Lambda_h}\| \mathcal{R}^u\curl\partial_t( (\gamma^{(h)}_{\ell_q })^{\frac12} (W^c_{(k)})) \) \nonumber \\
\lesssim& \sum_{k \in \Lambda_u\cup\Lambda_{B}}\(\|  a_{(k)} \|_{C_{T,x}^1}\| W^c_{(k)} \|_{C_T L^{p}_x }
+\|  a_{(k)} \|_{C_{T,x} }\| \p_t W^c_{(k)} \|_{C_TL^{p}_x }\) \notag\\
 &\    +   \sum_{k \in \Lambda_h}\(\| (\gamma^{(h)}_{\ell_q })^{\frac12} \|_{C_{T,x}^1}\| W^c_{(k)} \|_{C_T L^{p}_x }
+\| (\gamma^{(h)}_{\ell_q })^{\frac12} \|_{C_{T,x} }\| \p_t W^c_{(k)} \|_{C_T L^{p}_x }\) \notag\\
\lesssim&  \ell_q^{-2}\lambda^{-\ve}.
\end{align*}
Regarding the viscosity term in \eqref{rup},
as in the proof of \eqref{mag viscosity},
the application of the spatial intermittency yields that
\begin{align*}
	\norm{-\nu_1\mathcal{R}^u\Delta(w_{q+1}-  w_{q+1}^{(H)} )}_{ C_{\cI^{\rm p}}L^p_x}& \lesssim  \ell_q^{-2}\lambda^{-\frac18+5\ve} .
\end{align*}
Moreover, similarly to  \eqref{magnetic linear estimate1},
\begin{align*}
	\norm{ u_{q} \mathring\otimes w_{q+1} + w_{q + 1}\mathring \otimes u_{q}-B_{q}\mathring \otimes d_{q+1} - d_{q +1}\mathring \otimes B_{q} }_{C_{\cI^{\rm p}} L^p_x} 	
	\lesssim \ell_q^{-5}\lambda^{-3\ve}.
\end{align*}
Thus, combining the above estimates together we arrive at
\begin{align}  \label{Reynolds linear estimate}
\norm{ \mathring{R}_{lin}^u  }_{ C_{\cI^{\rm p}}L^p_x}
   \lesssim \ell_q^{-2}\lambda^{-\ve} .
\end{align}

\paragraph{\bf $(ii)$ Velocity oscillation error.}
For the velocity oscillation $\mathring{R}_{osc}^u $,
using \eqref{rou} we decompose
\begin{align*}
	\mathring{R}_{osc}^u  &= \mathring{R}_{osc.1}^u + \mathring{R}_{osc.2}^u+ \mathring{R}_{osc.3}^u+ \mathring{R}_{osc.4}^u,
\end{align*}
where $\mathring{R}_{osc.1}^u $ is the high-low  spatial frequency part
\begin{align}\label{ulhfp}
	\mathring{R}_{osc.1}^u  &:= \sum_{k \in \Lambda_u\cup\Lambda_B} \mathcal{R}^u \P_H\P_{\neq 0}\left( \P_{\neq 0} (W_{(k)}\otimes W_{(k)}-D_{(k)}\otimes D_{(k)})\nabla (a_{(k)}^2)\right),
\end{align}
$\mathring{R}_{osc.2}^u$ contains the  oscillation error caused by the heat corrector
\begin{align*}
	\mathring{R}_{osc.2}^u &
	:=\chi_{q+1}^2 \sum_{k \in \Lambda_u\cup\Lambda_B} \mathcal{R}^u \P_H \P_{\neq 0}
	\( (\partial_{t}-\nu_1 \Delta)( a_{(k)}^{2}) \int_{0}^t e^{(t-\tau)\Delta}\div (W_{(k)}\otimes W_{(k)}-D_{(k)}\otimes D_{(k)})\d \tau \)   \nonumber  \\
	&\quad -2\nu_1 \chi_{q+1}^2 \sum_{k \in \Lambda_u\cup\Lambda_B} \mathcal{R}^u \P_H \P_{\neq 0} \sum_{i=1}^3
	\( \partial_{i}( a_{(k)}^{2})  \partial_{i}\int_{0}^t e^{(t-\tau)\Delta}\div(W_{(k)}\otimes W_{(k)}-D_{(k)}\otimes D_{(k)}) \d \tau \),
\end{align*}
and  $\mathring{R}_{osc.3}^u $ contains the mollification error
\begin{align*}
	\mathring{R}_{osc.3}^u
	&:= \cq^2(\mathring{R}_{q}^u-\mathring{R}_{\ell_q}^u).
\end{align*}

Let us treat the four Reynolds oscillation errors separately.
First, by Lemmas \ref{buildingblockestlemma},  \ref{mae}, \ref{vae} and \ref{commutator estimate1}, 
\begin{align} \label{J1-esti}
	\norm{\mathring{R}_{osc.1}^u }_{C_{\cI^{\rm p}}L_x^p}
	& \lesssim
	 \sum_{k \in \Lambda_u\cup\Lambda_B}
     \left\||\na|^{-1}\P_{\neq 0}\(  \P_{\neq 0}(W_{(k)}\otimes W_{(k)}-D_{(k)}\otimes D_{(k)})\nabla (a_{(k)}^2)\) \right\|_{C_TL_x^p} \notag \\
	& \lesssim  \sum_{k \in \Lambda_u \cup \Lambda_B } \sigma^{-1} \|\na^3 (a^2_{(k)})\|_{C_{T,x}}
      \|\psi^2_{(k_1)} \phi_{(k)}^2\phi_{(k_2)}^2 \|_{C_TL^p_x } \notag \\
	& \lesssim  \ell_q^{-117}\sigma^{-1}  (\rp\rs\rw)^{\frac{1}{p}-1}\lesssim  \ell_q^{-117}\lambda^{-3\ve}.
\end{align}
By Lemma~\ref{buildingblockestlemma}, as in  \eqref{I3-est}-\eqref{I3-est-end}, the oscillation error $\mathring{R}_{osc.2}^u$ caused by the heat corrector can be bounded by
\begin{align}\label{J3-est}
 \norm{\mathring{R}_{osc.2}^u }_{C_{\cI^{\rm p}}L^p_x}
	&\lesssim \ell_q^{-84}\lambda^{-2\ve}.
\end{align}
At last, the mollification error $\mathring{R}_{osc.3}^u$ can be controlled by the usual mollification estimate and the iterative estimate \eqref{rbl1b-s}:  
\begin{align}  \label{J4-esti}
	\norm{\mathring{R}_{osc.3}^u   }_{C_{\cI^{\rm p}}L_x^1}
	&\lesssim \norm{ \cq^2}_{C_{T}}  \norm{\mathring{R}_{q}^u-\mathring{R}_{\ell_q}^u}_{C_{I_2}L_x^1}\lesssim \ell^{\frac12}.
\end{align}
Thus, putting estimates \eqref{J1-esti}-\eqref{J4-esti} altogether
we arrive at
\begin{align}\label{Reynolds oscillation estimate}
\norm{\mathring{R}_{osc}^u }_{ C_{\cI^{\rm p}}L^1_x}\lesssim \ell_q^{\frac12}.
\end{align}

\paragraph{\bf $(iv)$ Velocity corrector error.}
Applying Lemma~\ref{totalest} and using \eqref{est-edp} and \eqref{Lp decorr vel}
we get
\begin{align}  \label{Reynolds corrector estimate}
	\norm{\mathring{R}_{cor}^u  }_{C_{I_2}L^{ 1}_x} &\lesssim \norm{ \wt w_{q+1}^{(c)+(H)}}_{C_TL^{ 2}_x}
	(\norm{ \wt w^{(p)+(h)}_{q+1} }_{C_T L^2_{x}} + \norm{w_{q+1}  }_{C_T L^2_{x}})\notag\\
	&\quad +   (\norm{ \wt d^{(p)+(h)}_{q+1} }_{C_TL^2_{ x}} + \norm{d_{q+1}  }_{ C_TL^2_{x}}) \norm{\wt d_{q+1}^{(c)+(H)}}_{C_T L^{ 2}_x} \notag \\
	& \lesssim  \ell_q^{-2}\lambda^{-2\ve}.
\end{align}

\paragraph{\bf $(v)$ Velocity cut-off error.} Using \eqref{asp-0}, \eqref{asp-q}, \eqref{est-chiq} and Lemma~\ref{totalest}, similarly to  \eqref{R-cut-esti-1},  the velocity cut-off error can be bounded by
\begin{align} \label{Ru-cut-esti}
	\|\mathring{R}_{cut}^u \|_{ C_{\cI^{\rm p}}L^1_{x}}
	\leq \frac14c_*\delta_{q+3}.
\end{align}

Therefore, combining  \eqref{Reynolds linear estimate},
\eqref{Reynolds oscillation estimate},  \eqref{Reynolds corrector estimate} and \eqref{Ru-cut-esti} altogether we arrive at 
\begin{align} \label{rq1u}
	\|\mathring{R}_{q+1}^u  \|_{ C_{\cI^{\rm p}}L^1_{x}}\leq \frac12c_*\delta_{q+3}.
\end{align}

\subsection{Verification of growth estimate}  \label{sucsec-est-r-b}

This subsection is to verify the $C^{1}_{[T_{q+3},S_{q+3}]} \mathcal{L}^1_x$ and $C_{[T_{q+3},S_{q+3}]} \mathcal{W}_x^{1,1}$ estimates in \eqref{rbl1b-s} for the Reynolds and magnetic stresses. 
We divide the regime 
$[T_{q+3}, S_{q+3}]$ into two subregimes
$$
[T_{q+3}, S_{q+3}] = \bigcup\limits_{i=1}^2 \cI_i, 
$$
where $\cI_1:=[T_{q+3},T_{q+2}]\cup [S_{q+2},S_{q+3}]$ and $\cI_2:=[T_{q+2},S_{q+2}]$. 
Note that $\cI_2$ contains the perturbed temporal regime $\cI^{\rm p}$.

\medskip 
First, in the non-perturbed case where $t\in \cI_1:=[T_{q+3},T_{q+2}]\cup [S_{q+2},S_{q+3}]$, by the telescoping construction, there are no perturbations added to the velocity and magnetic fields at level $q+1$,
\begin{align}
w_{q+1}(t)= 0,\quad d_{q+1}(t)= 0.
\end{align}
It follows immediately that 
$$\mathring{R}_{lin}^B=\mathring{R}_{osc}^B  =\mathring{R}_{cor}^B=0,\quad \mathring{R}_{lin}^u=\mathring{R}_{osc}^u  =\mathring{R}_{cor}^u=0,$$
and
$$\mathring{R}_{cut}^B=\mathring{R}_{q}^B=\mathring{R}_{0}^B,\quad \mathring{R}_{cut}^u=\mathring{R}_{q}^u=\mathring{R}_{0}^u.$$
Therefore, by the estimates of the initial Reynolds and magnetic stresses in \eqref{est-r0-c-end} and \eqref{est-r0-w-1} below, we obtain the desirable upper bound 
\begin{align}\label{est-ruq1-c-p1}
	\| (\mathring{R}_{q+1}^{u},\mathring{R}_{q+1}^{B}) \|_{ C^{1}_{\cI_1}\cL^1_x}+ \| (\mathring{R}_{q+1}^{u},\mathring{R}_{q+1}^{B}) \|_{C_{\cI_1}\cW^{1,1}_x} = \| (\mathring{R}_{0}^{u},\mathring{R}_{0}^{B}) \|_{ C^{1}_{\cI_1}\cL^1_x}+ \| (\mathring{R}_{0}^{u},\mathring{R}_{0}^{B}) \|_{C_{\cI_1}\cW^{1,1}_x} \lesssim \lambda_{q+1}^8.
\end{align}

\medskip 
Next, we focus on the more delicate temporal regime where $t\in \cI_2:=[T_{q+2},S_{q+2}]$.

\paragraph{\bf $(i)$ Linear error.}
First, for the linear error $\mathring{R}_{lin}^B$, 
estimating in similar manner as in the proof of \eqref{r-lin-t} we have
\begin{align}\label{Rlin-t-c}
	&\quad\| \cq \mathcal{R}^B\partial_t( d_{q+1}^{(p) +(h)+(c)})\|_{C^{1}_{\cI_2}L_x^1} 
	\notag\\
    &\lesssim  \sum_{k \in \Lambda_B }\|\cq \|_{C^{1}_{T}} \| \partial_t( a_{(k)} D^c_{(k)}) \|_{C^{1}_{\cI_2}L_x^p} + \sum_{k \in \Lambda_h}\|\cq \|_{C^{1}_{T}} \|\partial_t( (\gamma^{(h)}_{\ell_q })^{\frac12}  W^c_{(k)})   \|_{C^1_{\cI_2}L_x^p} 
	\notag\\
    &\lesssim \lambda_{q+1}^{8},
\end{align}
and
\begin{align}\label{Rlin-t-h}
	\| \cq \mathcal{R}^B\partial_t( d_{q+1}^{(p) +(h)+(c)})\|_{C_{\cI_2}W_x^{1,1}} \lesssim 	\|  \mathcal{R}^B\partial_t( d_{q+1}^{(p) +(h)+(c)})\|_{C_{\cI_2}W_x^{1 ,p}}
	\lesssim&\  \lambda_{q+1}^8,
\end{align}
where the integrability exponent $p$ is given by \eqref{def-p}.

Regarding the viscosity term, an application of interpolation and Lemma \ref{totalest} gives
\begin{align}\label{viscosity-c}
	\norm{\nu\mathcal{R}^B\Delta( d_{q+1}- \wt d_{q+1}^{(H)})}_{C^{1}_{\cI_2}L_x^1}\lesssim \norm{ d_{q+1}- \wt d_{q+1}^{(H)}}_{C^{1}_{\cI_2}W_x^{1,p}}
	&\lesssim  \lambda_{q+1}^{8},
\end{align}
and
\begin{align}\label{viscosity-h}
	\norm{\nu\mathcal{R}^B \Delta( d_{q+1}- \wt d_{q+1}^{(H)})}_{C_{\cI_2}W_x^{1,1}}\leq \norm{\mathcal{R}^B\Delta( d_{q+1}- \wt d_{q+1}^{(H)})}_{C_{\cI_2}W_x^{1,p}}  &\lesssim \lambda_{q+1}^{8}.
\end{align}

For the remaining nonlinear terms in \eqref{rbp}, we have by Lemmas~\ref{totalest} and \ref{lem-heat-est},
\begin{align} \label{linear estimate1-c}
	\norm{  B_{q} \otimes w_{q+1} - w_{q + 1} \otimes B_{q}}_{C^{1}_{\cI_2}L_x^1} 
	\lesssim\, \norm{B_{q}}_{C^{1}_{\cI_2}L_x^2}   \norm{w_{q+1}}_{C^{1}_{T,x}} 
	\lesssim  \lambda_{q+1}^8
\end{align}
and
\begin{align} \label{linear estimate1-h}
\norm{  B_{q} \otimes w_{q+1} - w_{q + 1} \otimes B_{q}}_{C_{\cI_2}W_x^{1,1}} 
	\lesssim \norm{B_{q}}_{C_{\cI_2}H_x^{1}}   \norm{w_{q+1}}_{C^{1}_{T,x}} \lesssim \lambda_{q+1}^8.
\end{align}
Similar arguments also give 
\begin{align} \label{linear estimate1-c-b}
	\norm{  d_{q+1}\otimes u_{q} - u_{q} \otimes d_{q + 1}}_{C^{1}_{\cI_2}L_x^1}\lesssim \lambda_{q+1}^8
\end{align}
and
\begin{align} \label{linear estimate1-h-b}
\norm{  d_{q+1}\otimes u_{q} - u_{q} \otimes d_{q + 1}}_{C_{\cI_2}W_x^{1,1}}\lesssim \lambda_{q+1}^8.
\end{align}

Thus, estimates \eqref{Rlin-t-c}-\eqref{linear estimate1-h-b} altogether \eqref{b-beta-ve} lead to 
\begin{align}   \label{linear estimate-c}
	\norm{\mathring{R}_{lin}^B }_{C^{1}_{\cI_2}L_x^1} + \norm{\mathring{R}_{lin}^B }_{C_{\cI_2}W_x^{1,1}}\lesssim \lambda_{q+1}^8.
\end{align}

\paragraph{\bf $(ii)$ Oscillation error.}
In order to treat the oscillation error, applying the space intermittency estimates in Lemma \ref{buildingblockestlemma} and 
the amplitude estimates in Lemma \ref{mae} we get
\begin{align}  \label{ro-esti-c}
	\norm{\mathring{R}_{osc}^B  }_{C^{1}_{\cI_2}L_x^1}
	& \lesssim 	\|\cq^2\|_{C_T^1}	\norm{\mathring{R}_{q}^B }_{C^{1}_{\cI_2}L_x^1}+	\|\cq^2\|_{C_T^1}	\norm{\mathring{R}_{\ell_q}^B }_{C^{1}_{\cI_2}L_x^1} \notag\\
    &\quad+ \sum_{ k \in \Lambda_B } \|\cq^2\|_{C_T^1} \norm{ \na (a^2_{(k)})}_{C^1_{T,x}}\norm{W_{(k)} }_{C^{1}_TL^{2p}_x }\norm{D_{(k)} }_{C^{1}_TL^{2p}_x }\notag\\
	&\quad  +  \|\cq^2\|_{C_T^1} \sum_{k\in\Lambda_B }
	\norm{a_{(k)}^2 }_{C^3_{T,x}}
	\norm{\psi_{(k_1)}^2}_{C^{1}_TL^{p}_x }\norm{\phi_{(k)}}_{L^{2p}_x }^2\norm{\phi_{(k_2)}}_{L^{2p}_x  }^2 \nonumber  \\
	& \lesssim  \lambda_{q+1}^{8}.
\end{align}
We note that when treating $\mathring{R}_{osc.2}^B $, we used the heat kernel estimate
\begin{align*}
\norm{ \int_0^t e^{(t-\tau)\Delta}\div(D_{(k)} \otimes W_{(k)}-W_{(k)} \otimes D_{(k)} )\d \tau}_{C_TL^{p}_x }\lesssim \norm{D_{(k)} \otimes W_{(k)}-W_{(k)} \otimes D_{(k)} }_{C_TL^{p}_x }.
\end{align*}

Estimating in an analogous manner, we also have
\begin{align}  \label{ro-esti-h}
\norm{\mathring{R}_{osc}^B  }_{C_{\cI_2}W_x^{1,1}}
& \lesssim 	\|\cq^2\|_{C_T}\norm{\mathring{R}_{q}^B  }_{C_{\cI_2}W_x^{1,1}} +	\|\cq^2\|_{C_T}	\norm{\mathring{R}_{\ell_q}^B }_{C_{\cI_2}W_x^{1,1}} \notag\\
&\quad+\|\cq^2\|_{C_T} \sum_{ k \in \Lambda_B } \norm{ \na (a^2_{(k)})}_{C^1_{T,x}}\norm{W_{(k)} }_{C_TW_x^{1,2p}}\norm{D_{(k)} }_{C_TW_x^{1,2p}}\notag\\
&\quad  + \|\cq^2\|_{C_T}\sum_{k\in\Lambda_B }
\norm{  (a_{(k)}^2) }_{C^3_{T,x}}
\norm{\psi_{(k_1)}^2\phi_{(k)}^2\phi_{(k_2)}^2}_{C_TW_x^{1,p}} \nonumber  \\
& \lesssim  \lambda_{q+1}^{8}.
\end{align}

\paragraph{\bf $(iii)$ Corrector error.}
Using the interpolation inequality and
applying Lemmas \ref{totalest} and \ref{lem-heat-est} to \eqref{rbp2} we get
\begin{align}\label{corrector estimate-c}
	\norm{\mathring{R}_{cor}^B   }_{C^{1}_{\cI_2}L_x^1}
	&\lesssim \norm{\wt  w_{q+1}^{(c)+{(H)}}}_{C^{1}_{\cI_2}L_x^2}  (\norm{\wt  d^{(p)+(h)}_{q+1}  }_{C^{1}_{\cI_2}L_x^2}  + \norm{d_{q+1} }_{C^{1}_{\cI_2}L_x^2} )\notag\\
	&\quad +\norm{\wt d_{q+1}^{(c)+{(H)}}}_{C^{1}_{\cI_2}L_x^2}  (\norm{\wt w^{(p)+(h)}_{q+1}  }_{C^{1}_{\cI_2}L_x^2}  + \norm{w_{q+1} }_{C^{1}_{\cI_2}L_x^2} )\notag\\
	&\lesssim  \lambda_{q+1}^{8},
\end{align}
and
\begin{align}\label{corrector estimate-h}
	\norm{\mathring{R}_{cor}^B   }_{C_{\cI_2}W_x^{1,1}}
	&\lesssim \norm{\wt  w_{q+1}^{(c)+{(H)}}}_{C_{\cI_2}H_x^{1}}  (\norm{\wt  d^{(p)+(h)}_{q+1}  }_{C_{\cI_2}H_x^{1}} + \norm{d_{q+1} }_{C_{\cI_2}H_x^{1}} )\notag\\
	&\quad +\norm{\wt d_{q+1}^{(c)+{(H)}}}_{C_{\cI_2}H_x^{1}} (\norm{\wt w^{(p)+(h)}_{q+1}  }_{C_{\cI_2}H_x^{1}}  + \norm{w_{q+1} }_{C_{\cI_2}H_x^{1}} )\notag\\
	&\lesssim \lambda_{q+1}^8.
\end{align}

\paragraph{\bf $(iv)$ Cut-off error.}
By Lemmas \ref{totalest} and \ref{lem-heat-est} again, \eqref{est-r0-c-end}, \eqref{est-r0-w-1} and the interpolation inequality, 
one has 
\begin{align}\label{cut-est-c}
    \|\mathring{R}^B _{cut}\|_{C^{1}_{\cI_2}L_x^1} 	&\leq \| \mathcal{R}^B \(  \partial_t\chi_{q+1}d^{(p)+(c)+(h)}_{q+ 1}+\partial_t(\chi_{q+1}^2)d^{(H)}_{q+ 1}\)+ (1-\chi_{q+1}^2 )\mathring{R}_{q}^B\|_{C^{1}_{\cI_2}L_x^1}\notag\\
    &\lesssim \|\partial_t\chi_{q+1}\|_{C^{1}_{\cI_2}} \|d^{(p)+(c)+(h)}_{q+ 1}\|_{C^{1}_{\cI_2}L_x^p} + \|\partial_t(\chi_{q+1}^2)\|_{C^{1}_{\cI_2}} \|d^{(H)}_{q+ 1}\|_{C^{1}_{\cI_2}L_x^p} \notag\\ 
    &\quad +\|(1-\chi_{q+1}^2 )\|_{C^{1}_{\cI_2}} \|\mathring{R}_{q}^B\|_{C^{1}_{\cI_2}L_x^1}\notag\\
    &\lesssim \lambda_{q+1}^{8}
\end{align}
and 
\begin{align}\label{cut-est-h}
    \|\mathring{R}^B _{cut}\|_{C_{\cI_2}W_x^{1,1}} &\lesssim   \|\partial_t\chi_{q+1}\|_{C_{\cI_2}} \|d^{(p)+(c)+(h)}_{q+ 1}\|_{C_{\cI_2}W_x^{1,p}} + \|\partial_t(\chi_{q+1}^2)\|_{C_{\cI_2}} \|d^{(H)}_{q+ 1}\|_{C_{\cI_2}W_x^{1,p}} \notag\\
    &\quad +\|(1-\chi_{q+1}^2 )\|_{C_{\cI_2}} \|\mathring{R}_{q}^B\|_{C_{\cI_2}W_x^{1,1}}\notag\\
     &\lesssim \lambda_{q+1}^{8}.
\end{align}

Therefore,  combining  estimates \eqref{linear estimate-c}
\eqref{cut-est-h} altogether 
we obtain the desirable bound at level $q+1$: 
\begin{align} \label{rq1b-c}
	\|\mathring{R}_{q+1}^B   \|_{C^{1}_{\cI_2}L_x^1}
	&\leq \| \mathring{R}_{lin} ^B  \|_{C^{1}_{\cI_2}L_x^1} +  \| \mathring{R}_{osc}^B  \|_{ C^{1}_{\cI_2}L_x^1}
	+  \|\mathring{R}_{cor} ^B  \|_{C^{1}_{\cI_2}L_x^1}
 +  \|\mathring{R}_{cut} ^B  \|_{C^{1}_{\cI_2}L_x^1}
	\lesssim \lambda_{q+1} ^8,
\end{align}
and 
\begin{align} \label{rq1b-h}
	\|\mathring{R}_{q+1} ^B  \|_{C_{\cI_2}W_x^{1,1}}
	\lesssim \lambda_{q+1} ^8.
\end{align}
Similar arguments also lead to 
\begin{align} \label{rq1b-h-u}
	&  \|  \mathring{R}_{q+1}^{u}  \|_{ C^{1}_{\cI_2} L^1_x}+ \|  \mathring{R}_{q+1}^{u}  \|_{  C_{\cI_2} W^{1,1}_x} \lesssim \lambda_{q+1}^8.
\end{align}

In conclusion,  estimates \eqref{rq1b}, \eqref{rq1u}, \eqref{est-ruq1-c-p1}, \eqref{rq1b-c}, \eqref{rq1b-h} and \eqref{rq1b-h-u} together verify the iterative estimates \eqref{rubl1s} and \eqref{rbl1b-s} for the Reynolds and magnetic stresses at level $q+1$, thereby finishing the proof of Proposition~\ref{Prop-Iterat}.

\section{Proof of main results}

This section is devoted to proving the main results in Theorem \ref{Thm-Non-MHD}.  
We will treat the rough $\mathcal{L}_x^2$ initial data and the critical $\mathcal{H}_x^\frac 12$ 
initial data, respectively, in 
Subsections \ref{Subsub-L2-data}  
and \ref{Subsec-dis-energy} below.

\subsection{Continuous energy solutions  with rough $L_x^2$ initial data}  
\label{Subsub-L2-data}

In this subsection, 
we first prove the continuous energy solutions 
in Theorems~\ref{Thm-Non-MHD} 
with the rough $\mathcal{L}_x^2$ initial data.  

Since the telescoping convex integration is performed on the time interval $[0,T]$, 
not on the outer regime $[T,\infty)$, 
one has that $(u_q(t),B_q(t))=(\wt u(t), \wt B(t))$ for all $t\in [T,+\infty)$ and $q\geq 0$, 
where $(\wt u(t), \wt B(t))$ is the solution to the $\Lambda$-MHD \eqref{equa-mhd-2}. 
In view of the eventual regularity in Theorem \ref{thm-fns} $(iii)$, 
$(\wt u, \wt B)$ is already a dissipative solution to the MHD system \eqref{equa-MHD} after $T$, 
and so $(\wt u, \wt B)$ obeys the continuous energy on $[T,\infty)$.  
Thus, in the following, 
we only need to consider the temporal regime $[0,T]$.

\medskip 
Let us start with verifying the inductive estimates \eqref{rubl1s}-\eqref{rbl1b-s} at the initial step $q=0$. 
More precisely, by the regularity estimates of the solution to the $\Lambda$-MHD system in  \eqref{spa-regu}, \eqref{spa-regu-2}, \eqref{tem-regu-1} and \eqref{est-tem-reg-2}, taking $a$ large enough we have
\begin{align}\label{est-u0-c1}
    \|\wt u \|_{C_{[T_{2},S_2],x}^{1}}
    &\lesssim \| \p_t \wt u  \|_{C_{[T_{2},S_2]}H^2_x }+\|\wt u  \|_{ C_{[T_{2},S_2]}H_x^{3}}\notag\\
     &\lesssim (1+T_2^{-\frac{10}{3}}+\Lambda^{10}(S_{2}) )(1+\|(v_0,H_0) \|_{\cL^2_x }^2) \lesssim \lambda_0^4.
\end{align}
Similarly, by \eqref{spa-regu} and \eqref{tem-regu-1}, 
\begin{align}\label{est-wtu-c}
    \|\wt u \|_{ C_{[T_{q+2},S_{q+2}],x}^{1}}\lesssim (1+T_{q+2}^{-\frac{10}{3}}+\Lambda^{10}(S_{q+2}) )(1+\|(v_0,H_0) \|_{\cL^2_x }^2)\lesssim \lambda_q^4.
\end{align}
Analogous argument also yields
\begin{align}\label{est-wtb-c}
 \|\wt B\|_{C^{1}_{[T_{2},S_{2}]}L^{2}_x}\lesssim \lambda_0^4,  \quad \|\wt B\|_{C^{1}_{[T_{q+2},S_{q+2}]}L^{2}_x}\lesssim \lambda_q^4.
\end{align}
Moreover, by the explicit expression \eqref{r0u} 
of the initial Reynolds stress, 
one has 
\begin{align}\label{est-r0-c-1}
\| \mathring{R}_{0}^u \|_{C^{1}_{[T_{q+2},S_{q+2}]}L^1_x} &\lesssim  \|\P_{\geq \Lambda}(\P_{< \Lambda}\wt u\mathring \otimes \P_{< \Lambda}\wt u) \|_{C^{1}_{[T_{q+2},S_{q+2}]}L^1_x} +\|(\P_{< \Lambda}\wt u\mathring \otimes \P_{\geq \Lambda}\wt u ) \|_{C^{1}_{[T_{q+2},S_{q+2}]}L^1_x} \notag\\
&\quad + \|(\P_{\geq\Lambda}\wt u\mathring \otimes \wt u) \|_{C^{1}_{[T_{q+2},S_{q+2}]}L^1_x} +\|\P_{\geq \Lambda}(\P_{< \Lambda}\wt B\mathring \otimes \P_{< \Lambda}\wt B) \|_{C^{1}_{[T_{q+2},S_{q+2}]}L^1_x}\notag \\
 &\quad +\|(\P_{< \Lambda}\wt B\mathring \otimes \P_{\geq \Lambda}\wt B ) \|_{C^{1}_{[T_{q+2},S_{q+2}]}L^1_x} + \|(\P_{\geq\Lambda}\wt B\mathring \otimes \wt B) \|_{C^{1}_{[T_{q+2},S_{q+2}]}L^1_x} \notag\\
&\lesssim \|(\wt u,\wt B) \|_{C^{1}_{[T_{q+2},S_{q+2}]}\cL^{2}_x}^2.
\end{align}
Plugging estimates \eqref{est-wtu-c} and \eqref{est-wtb-c} into \eqref{est-r0-c-1}, we get the desirable  bound 
\begin{align}\label{est-r0-c-end}
    \| \mathring{R}_{0}^u \|_{C^{1}_{[T_{q+2},S_{q+2}]}L^1_x} &\lesssim \lambda_{q}^8.
\end{align}
Arguing in the analogous manner one also has 
\begin{align}\label{est-r0-w-1}
\| \mathring{R}_{0}^u \|_{C_{[T_{q+2},S_{q+2}]}W^{1,1}_x} +\| \mathring{R}_{0}^B \|_{C^{1}_{[T_{q+2},S_{q+2}]}L^1_x}+ \| \mathring{R}_{0}^B \|_{C_{[T_{q+2},S_{q+2}]}W^{1,1}_x}&\lesssim  \lambda_{q}^8.
\end{align}

Hence, 
with $a$ and $\delta_0$ sufficiently large, by 
\eqref{asp-0}, \eqref{est-u0-c1}, \eqref{est-wtu-c}, 
\eqref{est-r0-c-end} and \eqref{est-r0-w-1}, 
the iterative estimates \eqref{ubl2}-\eqref{energy-est} are satisfied at the initial level $q=0$. 
By virtue of Proposition \ref{Prop-Iterat}, 
there exists 
a sequence of relaxed solutions $(u_{q},B_{q},\rr_{q}^u,\rr_{q}^B)$ to \eqref{mhd1}
which satisfy estimates \eqref{ubl2}-\eqref{energy-est} for all $q\geq 0$.

\medskip 
Below we claim that $\{(u_{q}, B_{q})\}_q$ is a Cauchy sequence
in $C_T \mathcal{L}^2_x$ and the corresponding limit satisfies \eqref{equa-MHD} in the sense of Definition \ref{weaksolu}. 
To this end, by \eqref{u-B-L2tx-conv}, we derive 
from \eqref{u-B-L2tx-conv} that 
\begin{align}\label{interpo}
\sum_{q \geq 0} \|(u_{q+1} - u_q, B_{q+1} - B_q)\|_{C_T\cL^{2}_{ x}}
	\leq \sum_{q \geq 0} \delta_{q+1}^{\frac12} <\9.
    \end{align}
This yields that $\{(u_q,B_q)\}_{q\geq 0}$ is a Cauchy sequence in  $C_T\cL^2_x $, 
and so, there exists $(u,B)\in C_T\cL^2_x$ such that
\begin{align}   \label{uqBq-uB-0}
	\lim_{q\rightarrow+\infty}(u_q,B_q)=(u,B)\ \ in\ \ C_T \mathcal{L}^2_x.
\end{align}
Moreover,
since by the telescoping construction, $(w_{q},d_{q})(0)=0$ and $(u_q,B_q)(0)=(v_0,H_0)$ for all $q\geq 0$, we have $(u,B)(0)=(\wt u,\wt B)(0)=(v_0,H_0)$.
Taking into account the fact that $\lim_{q \to +\infty} (\mathring{R}_{q}^u, \mathring{R}_{q}^B) = 0 $ in $C([0,T];\cL^1_x)$,  
implied by \eqref{rubl1s},  
we infer that $(u,B)$ satisfies the MHD system \eqref{equa-MHD} in the sense of Definition \ref{weaksolu}, as claimed.

\medskip 
At last, the iterative energy and cross helicity estimates \eqref{energy-est} 
and \eqref{helicity-est} 
and the strong convergence in \eqref{uqBq-uB-0} 
yield that 
\begin{align*}
	e(t)- (\|u(t)\|_{L^2_x}^2+\|B(t)\|_{L^2_x}^2) =\lim_{q \to +\infty} e(t)- (\|u_q (t)\|_{L^2_x}^2+\|B_q(t)\|_{L^2_x}^2)= 0 
\end{align*}
and
\begin{align*}
	h(t)-\int_{\T^3} u(t)\cdot B(t) \,\d x =\lim_{q \to +\infty} h(t)- \int_{\T^3}  u_q(t) \cdot B_q(t) \,\d x= 0, 
\end{align*} 
which lead to the energy and cross helicity identities in \eqref{energy-d}.  
Therefore, we obtain infinitely many  continuous energy solutions for arbitrarily prescribed $\mathcal{L}_x^2$ initial data.

\subsection{Dissipative energy solutions with critical initial data} 
\label{Subsec-dis-energy}

This subsection is devoted to proving the energy decreasing part for the critical $\mathcal{H}_x^\frac 12$ initial data in Theorem~\ref{Thm-Non-MHD}. 
To begin with, 
let us first consider the more regular $\mathcal{H}_x^3$ initial data.

\subsubsection{Improved estimate of Reynolds and magnetic stresses}

For the regular initial data, e.g., $(v_0,H_0)\in \cH^3_x$, 
using the method analogous to the proof of Theorem~\ref{thm-fns}, one has 
a solution $(\wt u,\wt B)$ to the $\Lambda$-MHD system 
in the regular space $C_T\cH^3_x$.  
The following improved decay estimate of the Reynolds and magnetic stresses is important to construct decreasing energy profile. 

\begin{lemma}[Improved decay estimate of Reynolds and magnetic stresses]\label{lem-im-decay} 
Let $M:=\|(v_0,H_0)\|_{\cH^3_x}$.  
It holds the improved decay estimate that for any {$t$} close to $0$,
   \begin{align} \label{decay-R0-refined}
\|(\mathring{R}_{0}^u,\mathring{R}_{0}^B ) \|_{C_{[0,t]}\cL^{1}_x} \lesssim E_*^{-3} tM^2 e^{M}, 
\end{align} 
where $E_*$ is the large parameter in the 
wavenumber $\Lambda$ 
as in Subsection \ref{Subsec-back-CI}. 
\end{lemma}

\begin{proof} 
As in  the decay estimates \eqref{est-non-u}-\eqref{est-non-b}, one has for the initial Reynolds stress 
\begin{align}\label{r0-est-1-2}
	\|\mathring{R}_{0}^u\|_{C_{[0,t]} L^{1}_x} 
	&\lesssim \|\P_{\geq \Lambda}(\P_{< \Lambda}\wt u\mathring \otimes \P_{< \Lambda}\wt u) \|_{C_{[0,t]}L^{1}_x}+\|\wt u \|_{C_{[0,t]}L^{2}_x}\|\P_{ \geq \Lambda} \wt u \|_{C_{[0,t]}L^2_x}\notag\\
 &\quad+ \|\P_{\geq \Lambda }(\P_{< \Lambda }\wt B\mathring \otimes \P_{< \Lambda }\wt B) \|_{C_{[0,t]}L^{1}_x}+\|\wt B \|_{C_{[0,t]}L^{2}_x}\|\P_{ \geq \Lambda} \wt B \|_{C_{[0,t]}L^2_x}, 
\end{align} 
where $(\wt u, \wt B)$ is the solution to the $\Lambda$-MHD system \eqref{equa-mhd-2}.

By the energy balance \eqref{eq-e-2},  
we first note that 
\begin{align}\label{u-l2-2r-b}
    \|(\wt u,\wt B )\|_{ C_{[0,t]}\cL^2_x}\leq \|(v_0,H_0) \|_{\cL^2_x}\leq M.
\end{align} 
Moreover, 
since $\Lambda(t)= E_* t^{-\frac13}$ 
for $t$ close to $0$, 
\begin{align}\label{ref-l2}
    \|\P_{ \geq \Lambda} \wt u \|_{ C_{[0,t]}L^2_x}\lesssim \| |\xi|^{-3}(1-\varphi_\Lambda)\|_{C_{[0,t]}L^\infty_\xi} \| |\xi|^3 \widehat{\wt u} \|_{ C_{[0,t]}L^2_\xi}\lesssim E_*^{-3} t \|\wt u\|_{ C_{[0,t]}H^3_x },
\end{align}
and
\begin{align}\label{est-non-decay-2-2}
\|\P_{\geq \Lambda}(\P_{< \Lambda}\wt u\mathring \otimes \P_{< \Lambda}\wt u)\|_{C_{[0,t]}L^1_x}
&\lesssim  ( \|\wt u \|_{ C_{[0,t]}L^2_x}  \| \P_{[\frac{\Lambda}{6},2\Lambda ]}\wt u\|_{ C_{[0,t]}L^2_x}  
+ \| \P_{[\frac{\Lambda}{6},2\Lambda]}\wt u\|_{ C_{[0,t]}L^2_x}^2 )\notag\\
&\lesssim E_*^{-3} t \|\wt u \|_{ C_{[0,t]} H^3_x}^2.
\end{align}

In order to estimate the $\mathcal{H}_x^3$ regularity of $(\wt u, \wt B)$, taking the $\dot \cH^3_x$ inner product of \eqref{equa-mhd-2} with $(\wt u, \wt B)$ and using Kato's inequality we get
\begin{align}
    \frac12\frac{\d}{\d t}\|(\wt u ,\wt B )\|_{\dot \cH^3_x}^2+\nu_1\| \wt u \|_{\dot H^4_x}^2+\nu_2\|\wt B \|_{\dot{H}_x^{4}}^2&\lesssim \|( \P_{< \Lambda } \nabla \wt u , \P_{< \Lambda } \nabla \wt B )\|_{ \cL^\infty_x} \|( \P_{< \Lambda } \wt u , \P_{< \Lambda } \wt B )\|_{ \dot \cH^3_x}^2.
\end{align}
By Sobolev's embedding $H^{\frac74}_x\hookrightarrow L^\infty_x$, 
\begin{align}   \label{uH3-est}
   \int_0^t \|( \P_{< \Lambda } \nabla \wt u , \P_{< \Lambda } \nabla \wt B )\|_{ \cL^\infty_x}  \d s
   &\lesssim  \int_0^t \| (1+|\xi|^2)^{\frac{11}{8}} \varphi_{\Lambda(s)}(\xi) (\widehat {\wt u}(s),\widehat {\wt B}(s)  ) \|_{\cL^2_\xi}  \d s\notag\\
   &\lesssim \int_0^t (1+4\Lambda^2(s))^\frac{11}{8} \d s \|(\wt u,\wt B)\|_{C_t  \cL^2_x} \lesssim M, 
\end{align} 
where the last step was due to the blow 
rate $t^{-1/3}$ of $\Lambda$ 
near the initial time  to achieve the convergence of the integral.
Hence, an application of Gronwall’s inequality yields that for $t\in (0,T)$, 
\begin{align}\label{gron}
    \|(\wt u,\wt B)\|_{C_t \dot \cH^3_x}^2\lesssim \|(v_0,B_0)\|_{\dot \cH^3_x}^2 \exp\left\{ \int_{0}^t \|(\P_{<\Lambda(s)}\wt u(s),\P_{<\Lambda(s)}\wt B(s))\|_{\dot \cH^3_x} \d s  \right\}\lesssim M^2 e^M.
\end{align} 

Therefore, plugging \eqref{ref-l2} and \eqref{est-non-decay-2-2} into \eqref{r0-est-1-2} and using \eqref{gron} we arrive at
\begin{align*}
\|\mathring{R}_{0}^u\|_{C_{[0,t]}L^{1}_x} &\lesssim E_*^{-3}t\|(\wt u,\wt B ) \|_{ C_{[0,t]} \cH^3_x}^2\lesssim E_*^{-3} t M^2 e^M, 
\end{align*}  
which gives the improved decay estimate 
for the Reynolds stress. 
The proof for the magnetic stress follows in an analogous manner.  
\end{proof}

\subsubsection{The regular $\mathcal{H}_x^3$  case}\label{subsec-con-dis-energy}
  
Next, we prove that there exist infinitely many  energy profiles $e: [0,\infty) \rightarrow [0,\9)$ that are strictly deceasing on time 
for arbitrarily prescribed $\mathcal{H}_x^3$ 
initial data.  
As explained at the beginning of Section 
\ref{Subsub-L2-data}, 
$(u_q, B_q) \equiv (\wt u, \wt B)$ 
dissipates the energy on the outer interval $[T,\infty)$, 
we only need to consider the inner temporal regime $(0,T)$. 

\medskip 
Unlike in Subsection~\ref{Subsec-time}, 
we choose $T_1$ sufficiently close to initial time such that  
the improved decay estimate in Lemma~\ref{lem-im-decay} holds on $(0,T_1]$ 
and, via the strong continuity at the initial time, 
\begin{align}\label{def-t1}
    \|( u_0(T_1+{\ell_0}), B_0(T_1+{\ell_0}))\|_{\cL^2_x}\geq \frac12 \|( v_0, H_0)\|_{\cL^2_x}.
 \end{align}
Then, we choose 
the explicit sequence of backward times $\{T_{q}\}_q$ such that $T_{q+1}=T_q/2$ for $q\geq 1$ and 
the amplitude parameter 
\begin{align}\label{edf-deltaq-2}
    \delta_{q+2}:= 16c_*^{-1}E_*^{-3}T_q M^2 e^{M}. 
\end{align} 
The forward time sequence $\{S_{q}\}_q$ is still chosen  as in Subsection~\ref{Subsec-time}, with $\delta_q$ modified by \eqref{edf-deltaq-2}.
We also keep using the shorthand 
$\overline{T}_{q+1} :=T_{q+1} + \ell_q$ 
and $\underline{T}_{q+1} :=T_{q+1} - \ell_q$, $q\geq 0$. 

Note that 
\begin{align}\label{asp-q-1}
	\|(\mathring{R}_{0}^u,\mathring{R}_{0}^B )\|_{  C_{[0,T_{q}+\ell'_{q-1}]} \cL^{1}_x}\leq \frac18c_*\delta_{q+2}
\end{align}
for all $q\geq 0$. 

Let $\nu:=\min\{\nu_1,\nu_2\}$ and $e$ satisfy \eqref{asp-e-0}-\eqref{asp-e-q}. By the energy balance \eqref{eq-e-2} and Poincar\'e's inequality, 
	\begin{align}\label{eq-e-st}
		\frac12 \frac{\d}{\d t} \| (u_0(t),B_0(t))\|_{\cL^2_x}^2= -\nu  \|(\nabla u_0(t),\nabla B_0(t))\|_{\cL^2_x}^2\leq -\nu C  \| (u_0(t),B_0(t))\|_{\cL^2_x}^2,
	\end{align}
 where $C$ is a universal constant. 
 In particular, $\|( u_0(t), B_0(t))\|_{\cL^2_x}$ is strictly decreasing on time 
 for $t\in [0,T_1]$  and 
\begin{align*}
     \|( u_0(t), B_0(t))\|_{\cL^2_x}\geq \|( u_0(\ot_1), B_0(\ot_1))\|_{\cL^2_x}>0,
 \end{align*}
 for all $t\in [0,T_1]$.

\medskip 
As a consequence, for any non-trivial initial datum $(v_0,H_0)\in \mathcal{H}_x^3$, we can choose two continuous upper and lower  energy profiles $\overline e$ and $\underline e $, respectively, 
defined by 
\begin{align}
    \overline e:= \|( u_0(t), B_0(t))\|_{\cL^2_x}^2+\overline f(t),\quad \underline e:= \|( u_0(t), B_0(t))\|_{\cL^2_x}^2+\underline f(t), 
    \quad t\in [0,T], 
\end{align}  
where 
the energy profiles  $\overline e$ and $\underline e$ coincide with the energy of the $\Lambda$-MHD solution $(u_0,B_0)$ at the initial time and 
the endpoint time $T$, 
and in the inner regime $(0,T)$ 
the upper variation profile  $\overline f$ 
is linearly continuous on each pieces 
$[\ot_{q+1},\ot_{q}]$ and $[\us_q,\us_{q+1}]$, $q\in \mathbb{N}_+$, 
and constant on $[\ot_1,\us_1]$, 
that is,  
\begin{align}\label{def-of}
&\overline f|_{[\ot_1,\us_1]}=\delta_2,\quad \overline{f}|_{[\ot_2,\ot_1]}=\delta_2+\frac{\delta_3-\delta_2}{\ot_2-\ot_1}(t-\ot_1),\quad \overline{f}|_{[\ot_{q+1},\ot_{q}]}= \delta_{q+1}+\frac{\delta_{q+2}-\delta_{q+1}}{\ot_{q+1}-\ot_q}(t-\ot_q),\notag\\
&\overline{f}|_{[\us_1,\us_2]}=\delta_2+\frac{\delta_3-\delta_2}{\us_2-\us_1}(t-\us_1),\quad \overline{f}|_{[\us_{q},\us_{q+1}]}= \delta_{q+1}+\frac{\delta_{q+2}-\delta_{q+1}}{\us_{q+1}-\us_q}(t-\us_q),
\end{align} 
while the lower variation profile 
$\underline{f}$ is linearly continuous 
on slightly different time pieces: 
\begin{align}\label{def-uf}
&\underline f|_{[\ut_1,\os_1]}=\frac34\delta_2,\quad \underline{f}|_{[\ut_2,\ut_1]}=\frac34\delta_2+\frac34\frac{\delta_3-\delta_2}{\ut_2-\ut_1}(t-\ut_1),\quad \underline{f}|_{[\ut_{q+1},\ut_{q}]}= \frac34\delta_{q+1}+\frac34\frac{\delta_{q+2}-\delta_{q+1}}{\ut_{q+1}-\ut_q}(t-\ut_q),\notag\\
&\underline{f}|_{[\os_1,\os_2]}=\frac34\delta_2+\frac34\frac{\delta_3-\delta_2}{\os_2-\os_1}(t-\os_1),\quad \underline{f}|_{[\os_{q},\os_{q+1}]}= \frac34\delta_{q+1}+\frac34\frac{\delta_{q+2}-\delta_{q+1}}{\os_{q+1}-\os_q}(t-\os_q).
\end{align}

In view of the strong continuity of the $\Lambda$-MHD solution on $[0,T]$, 
we see that the upper and lower energy profiles $\overline e$ and $\underline e$ are continuous on $[0,T]$. 
Moreover, it is not difficult to check that, since $\ell_q$ is very small such that 
$$
\ell_q<\min\{\frac{T_{q+1}-T_{q+2}}{20}, \frac{S_{q+2}-S_{q+1}}{20}\},$$
one has 
$$
 \underline f(\ut_q)<\overline f(\ut_q),\quad   \underline f(\ot_q)<\overline f(\ot_q),\quad  \underline f(\us_q)<\overline f(\us_q),\quad  \underline f(\os_q)<\overline f(\os_q),\quad \text{for all}\quad q\geq 1,
 $$ 
and thus
\begin{align*}
    \underline f(t)<\overline{f}(t), \quad \forall\ t\in (0,T]. 
\end{align*}
This shows that 
the upper and lower energy profiles 
are strictly different for all positive times 
\begin{align}\label{ue-oe}
    \underline e< \overline e,\quad \forall\ t\in (0,T).
\end{align}

On the forward regime  $[\us_1,T)$, 
by the construction, 
we note that $\overline f$ and $\underline{f}$ are decreasing on $[\us_1,T)$, 
so are the energy profiles $\overline e$ and $\underline{e}$.

\medskip 
On the backward regime 
$(0,\ot_1]$,  
though the variation profiles  $\overline f$ and $\underline{f}$ 
are increasing in time,  
we show that, 
by choosing $E_*$ sufficiently large in the wavenumber $\Lambda$,  
they are indeed relatively small compared to the decreasing energy of the $\Lambda$-MHD solution $(u_0, B_0)$, 
and so, the energy profiles $\overline e$ and $\underline e$ are still strictly decreasing on $(0,\ot_1]$ and $(0,\ut_1]$, respectively. 

To this end, 
we compute from \eqref{eq-e-st} and the above explicit expressions of $\overline f$ and $\underline f$ that 
for any $q\geq 1$,
\begin{align}\label{oe-ineq}
    \overline{e}'(t) &\leq  -2C\nu  \| (u_0(\ot_1),B_0(\ot_1))\|_{\cL^2_x}^2+ \frac{\delta_{q+2}-\delta_{q+1}}{\ot_{q+1}-\ot_q} \notag  \\
     &\leq  -2C\nu  \| (u_0(\ot_1),B_0(\ot_1))\|_{\cL^2_x}^2+ 64c_*^{-1} E_*^{-3} M^2 e^{M}, 
     \ \ t\in (\ot_{q+1},\ot_q), 
\end{align}
and  
\begin{align}\label{ue-ineq}
    \underline{e}'(t) &\leq  -2C\nu  \| (u_0(\ut_1),B_0(\ut_1))\|_{\cL^2_x}^2+ \frac34\frac{\delta_{q+2}-\delta_{q+1}}{\ut_{q+1}-\ut_q} \notag  \\
     &\leq  -2C\nu  \| (u_0(\ot_1),B_0(\ot_1))\|_{\cL^2_x}^2+ 64c_*^{-1} E_*^{-3} M^2 e^{M}, 
     \ \ t\in (\ut_{q+1},\ut_q), 
\end{align} 
where $M:=\|(v_0,H_0)\|_{\cH^3_x}$. 
Choosing $E_*$ sufficiently large such that
\begin{align}\label{decre-condi-c}
-C\nu\|(v_0,H_0)\|_{\cL^2_x}^2+64c_*^{-1}E_*^{-3} M^2 e^{M}<0. 
\end{align} 
Note that, by \eqref{def-t1}, 
\begin{align}\label{decre-condi-c-1}
-2C\nu\|(u_0(\ot_1),B_0(\ot_1))\|_{\cL^2_x}^2+64c_*^{-1}E_*^{-3} M^2 e^{M}<0. 
\end{align}  
Taking into account \eqref{oe-ineq} and \eqref{ue-ineq} one obtains that 
\begin{align}\label{oe-de}
    \overline e'(t)<0,\quad \forall\, t\in (\ot_{q+1},\ot_q),\ q\geq 1,
\end{align}
and
\begin{align}\label{ue-de}
     \underline e'(t)<0,\quad \forall\, t\in (\ut_{q+1},\ut_q),\ q\geq 1.
\end{align} 
Thus,  $\overline e$ and $\underline e$ are strictly decreasing on $(0,\overline T_1]$ and  $(0,\ut_1]$ 
respectively, as claimed. 

\medskip 
Now, $\overline{e}$ is decreasing on $[\us_1,T]$, $\overline f$ is a constant on $[\ot_1, \us_1]$ and $\|( u_0(t), B_0(t))\|_{\cL^2_x}$ is non-increasing on $[\overline T_1, \us_1]$, we thus infer that $\overline{e}$ is non-increasing on $[0, T]$. Similar argument also gives the non-increasing 
property of the lower energy profile $\underline{e}$ on $[0, T]$. Taking into account that  
the upper and lower energy profiles are strictly different on $(0,T)$,  
as revealed by \eqref{ue-oe}, 
we thus infer that there exist infinitely many strictly decreasing energy profile $e$ 
inside the energy channel formed by the  upper and lower energy profiles, 
that is, $\underline e(t)\leq e(t)\leq \overline e(t)$ for any $t\in [0,T]$,  
$e$ is decreasing on $[0,T]$, 
and $e$ is smooth for $t\in (0,T)$.

\medskip 
Finally, because the upper and lower energy profiles still satisfy the iterative estimates \eqref{asp-e-0}-\eqref{bdd-e-h-c1}, proceeding as in the telescoping constructions in Subsection \ref{Subsub-L2-data}, 
one can construct weak solutions $(u,B)$ to the MHD system with the decreasing energy profile $e$ on $[0,T]$. 
Consequently, 
as $(u,B)$ is a strong dissipative solution on $[T,+\infty)$, 
we construct infinitely many solutions with decreasing energy profiles globally on $[0,\infty)$ 
with arbitrarily prescribed $\mathcal{H}_x^3$ initial data.

\subsubsection{The critical $\mathcal{H}_x^\frac 12$ case}

Now, we are in position to deal with the critical case where the initial data $(v_0,H_0)$ are in the critical space $\cH^{\frac12}_x$. 

\medskip 
We first claim that for any given $\cH^\frac12_x$ initial data, 
the MHD system \eqref{equa-MHD} 
has a regular solution 
near the initial time. 

To this end, let $(u_n,B_n)$ be the standard  Galerkin approximations, we decompose $(u_n,B_n)$ by
\begin{align*}
    (u_n,B_n)=(v_n,H_n)+(w_n,z_n),
\end{align*}
where $v_n$ solves the linear equation 
\begin{align}\label{equa-vn}
	\p_t v_n- \nu_1 \Delta v_n=0, \quad  
    v_n(0,x)= \P_n v_0
\end{align}
and so does $H_n$:
\begin{align}\label{equa-hn}
    \p_t H_n- \nu_2 \Delta H_n=0, \quad 
     H_n(0,x)= \P_n H_0 
\end{align}
with $\P_n$ given by \eqref{def-pn}. 
Then, $(w_n,z_n)$ solves
\begin{align}\label{equa-wz}
	\begin{cases}
	\p_t w_n- \nu_1 \Delta w_n+ \P_n  ( u_n\cdot\nabla  u_n- B_n\cdot\nabla B_n)=0,  \\
    \p_t z_n- \nu_2 \Delta z_n+ \P_n  ( u_n\cdot\nabla  B_n- u_n\cdot\nabla B_n)=0,  \\
     w_n(0)=0,\ z_n(0)=0.
	\end{cases}
\end{align}
By the energy estimates, 
one has 
\begin{align}
    \frac12 \|(v_n,H_n)(t)\|_{\dot \cH^\frac12_x}^2+ \nu \int_0^t \|(v_n,H_n)(s)\|_{\dot \cH^\frac32_x}^2 \d s\leq \frac12 \|(v_0,H_0) \|_{\dot \cH^\frac12_x }^2,
\end{align}
where $\nu:=\min\{\nu_1,\nu_2\}$. 
Hence, by the interpolation, 
$(v_n,H_n)\in L^4(0,T;\dot{\cH}^1_x)$. 
Moreover, taking the $\mathcal{L}_x^2$ inner product of \eqref{equa-wz} with $(\nabla w_n,\nabla z_n)$, using Young's inequality, the interpolation inequality and the Sobolev embeddings  $\dot{H}^1_x\hookrightarrow L^6_x$ and $\dot{H}^\frac12_x\hookrightarrow L^3_x$ we get
\begin{align*}
\frac{\d}{\d t}\|(w_n ,z_n )\|_{\dot \cH^\frac12_x}^2+\nu \|(w_n ,z_n) \|_{\dot \cH^\frac32_x}^2\leq  C\|(w_n ,z_n)\|_{\dot \cH^\frac12_x}^2\|(w_n ,z_n)\|_{\dot \cH^\frac32_x}^2 + \|(v_n ,H_n)\|_{\dot \cH^1_x}^4.
\end{align*}
Thus, applying \cite[Lemma 10.3]{rrs16} for some $T'\in (0,T)$ sufficiently small, 
one gets $ (w_n,z_n)\in L^\infty_{T'}\dot{\cH}^{\frac12}_x\cap L^2_{T'}\dot{\cH}^{\frac32}_x$. Analogous arguments as in the proof of \cite[Theorem 10.1]{rrs16} 
then show that there exists a solution $(u',B')\in L^\infty_{T'}\dot{\cH}^{\frac12}_x\cap L^2_{T'}\dot{\cH}^{\frac32}_x$ to the MHD system  \eqref{equa-MHD} which obeys the energy balance \eqref{eq-e-2} on the small time interval $[0,T']$. 

Moreover, by the Sobolev embedding $\dot{H}^{\frac{9}{10}}_x\hookrightarrow L^5_x$ and the interpolation inequality, we derive
\begin{align}
    \int_{0}^{T'} \|(u',B')\|_{\cL^5_x}^5 \d t \lesssim \int_{0}^{T'} \|(u',B')\|_{\dot{\cH}^{\frac{9}{10}}_x}^5\d t\lesssim \|(u',B')\|_{L^\infty_{T'}\dot{\cH}^{\frac12}_x}^3  \|(u',B')\|_{L^2_{T'}\dot{\cH}^{\frac32}_x}^2,
\end{align}
which yields that $(u',B')\in L^5_{T'}\cL^5_x$. 
Hence, in view of the regularity criterion of MHD system in 
\cite[Theorem 1.3]{CD15}
and 
the embedding of Besov spaces (\cite[P169, (21)]{ST87})
\begin{align*}
   L^5(\mathbb{T}^3) \subseteq B^0_{5,\infty} (\mathbb{T}^3), 
\end{align*} 
we thus obtain that the local solution $(u',B')$ is regular on $[0,T']$. 
In particular, 
$(u',B')\in C_{T'}\cH^3_x$. 

\medskip 
Thus, one can choose $(u'(T'),B'(T'))\in \cH_x^3$ 
as the initial datum to construct 
infinitely many weak solutions on $[T',+\infty)$ with decreasing energy as 
in Subsection \ref{subsec-con-dis-energy}. 
Taking into account that $(u',B')$ satisfies the energy balance on $[0,T']$, 
and so, 
dissipates the energy, 
we thus obtain infinitely many weak solutions with decreasing energy 
on the whole time regime $[0,+\infty)$. 
The proof of Theorem \ref{Thm-Non-MHD} is finally complete. \hfill $\square$

\section{Appendix}
We first recall from \cite{bbv20} 
the following two geometric lemmas used in the construction of the velocity and magnetic perturbations.

\begin{lemma} [First Geometric Lemma] \cite[Lemma 4.1]{bbv20}) \label{Lem-Geo-Anti}
	\label{geometric lem 1}
	There exists a set $\Lambda_B \subset \mathbb{S}^2 \cap \mathbb{Q}^3$ that consists of vectors $k$ with associated orthonormal bases $(k, k_1, k_2)$,  $\varepsilon_B > 0$, and smooth positive functions $\gamma_{(k)}: B_{\varepsilon_B}(0) \to \mathbb{R}$, where $B_{\varepsilon_B}(0)$ is the ball of radius $\varepsilon_B$ centered at 0 in the space of $3 \times 3$ skew-symmetric matrices, such that for  $A \in B_{\varepsilon_B}(0)$ we have the following identity:
	\begin{equation}
		\label{antisym}
		A = \sum_{k \in \Lambda_B} \gamma_{(k)}^2(A) (k_2 \otimes k_1 - k_1 \otimes k_2) .
	\end{equation}
\end{lemma}

\begin{lemma} [Second Geometric Lemma]\cite[Lemma 4.2]{bbv20})
	\label{geometric lem 2}
	There exists a set $\Lambda_u \subset \mathbb{S}^2  \cap \mathbb{Q}^3$ that consists of vectors $k$ with associated orthonormal bases $(k, k_1, k_2)$,  $\varepsilon_u > 0$, and smooth positive functions $\gamma_{(k)}: B_{\varepsilon_u}(\Id) \to \mathbb{R}$, where $B_{\varepsilon_u}(\Id)$ is the ball of radius $\varepsilon_u$ centered at the identity in the space of $3 \times 3$ symmetric matrices,  such that for  $S \in B_{\varepsilon_u}(\Id)$ we have the following identity:
	\begin{equation}
		\label{sym}
		S = \sum_{k \in \Lambda_u} \gamma_{(k)}^2(S) k_1 \otimes k_1 .
	\end{equation}
	Furthermore, we may choose $\Lambda_u$ such that $\Lambda_B \cap \Lambda_u = \emptyset$.
\end{lemma}

By the choice of $\Lambda_u$ and $\Lambda_B$ in \cite{bbv20}, there exists $N_{\Lambda} \in \mathbb{N}$ (for instance, $N_{\Lambda} = 65$) such that
\begin{equation} \label{NLambda}
	\{ N_{\Lambda} k,N_{\Lambda}k_1 , N_{\Lambda}k_2 \} \subset N_{\Lambda} \mathbb{S}^2 \cap \mathbb{Z}^3.
\end{equation}
Also, let $M_*$ be a geometric constant such that
\begin{align}	\label{M bound}
	\sum_{k \in \Lambda_{u}} \norm{\gamma_{(k)}}_{C^4(B_{\varepsilon_u}(\Id))}
	+ \sum_{k \in \Lambda_{B}} \norm{\gamma_{(k)}}_{{C^4(B_{\varepsilon_B}}(0))} \leq M_*.
\end{align}
This parameter  is universal and is used in the estimates of the size of perturbations.

\medskip 
Next, we recall the definition of inverse-divergence operators $\mathcal{R}^u$ and $\mathcal{R}^B$, defined by
\begin{align}
	& (\mathcal{R}^u v)^{kl} := \partial_k \Delta^{-1} v^l + \partial_l \Delta^{-1} v^k - \frac{1}{2}(\delta_{kl} + \partial_k \partial_l \Delta^{-1})\div \Delta^{-1} v, \label{operaru} \\
	& (\mathcal{R}^Bf)_{ij} :=  \varepsilon_{ijk} (-\Delta)^{-1}(\curl f)_k,\label{operarb}
\end{align}
where $\int_{\mathbb{T}^3} v dx =0$, $\div f=0$,
and $\varepsilon_{ijk}$ is the Levi-Civita tensor, $i,j,k,l \in \{1,2,3\}$. 
The operator $\mathcal{R}^u$ returns symmetric and trace-free matrices,
while the operator $\mathcal{R}^B$ returns skew-symmetric matrices.
Moreover, one has the algebraic identities
\begin{align*}
	\div \mathcal{R}^u(v) = v,\ \ \div \mathcal{R}^B(f) = f.
\end{align*}
Both
$|\nabla|\mathcal{R}^u$ and $|\nabla|\mathcal{R}^B$ are Calderon-Zygmund operators
and thus they are bounded in the spaces $L^p$, $1<p<+\infty$.
See \cite{bbv20,dls13} for more details.

\medskip 
Below we recall the standard H\"{o}lder estimates from \cite[(130)]{bdis15}.
\begin{proposition}[Standard  H\"{o}lder estimates]\label{prop-holder}
Let $\Omega \subset \mathbb{R}^N$ and $f: \Omega \rightarrow \mathbb{R}$, $g: \mathbb{R}^n \rightarrow \Omega$ be two smooth functions. Then, there exists a constant $C$, such that
for every $m \in \mathbb{N} \backslash\{0\}$ it holds
\begin{align}
	 & \|f\circ g\|_{C^m} \leq C\left( \|f\|_{\dot C^1}\| g\|_{ \dot C^m}+\|\nabla f\|_{ C^{m-1}}\|g\|_{C^{m-1}}\| g\|_{ \dot C^m}\right), \label{holder1}\\
	 & \|f\circ g\|_{C_x^m} \leq C\left(\|f\|_{ \dot C^1}\| g\|_{\dot C^m}+\|\nabla f\|_{ C^{m-1}}\|g\|_{ \dot C^1}^{m-1}\right).\label{holder2}
\end{align}
\end{proposition}

\medskip 
The stationary phase lemma below is a main tool to handle the errors of Reynolds stress.
\begin{lemma}[\cite{lt20}, Lemma 6; see also \cite{bv19b}, Lemma B.1] \label{commutator estimate1}
	Let $a \in C^{2}\left(\mathbb{T}^{3}\right)$. For all $1<p<+\infty$ we have
	$$
	\left\||\nabla|^{-1} \P_{\neq 0}\left(a \P_{\geq k} f\right)\right\|_{L^{p}\left(\mathbb{T}^{3}\right)} \lesssim k^{-1}\left\|\nabla^{2} a\right\|_{L^{\infty}\left(\mathbb{T}^{3}\right)}\|f\|_{L^{p}\left(\mathbb{T}^{3}\right)},
	$$
	holds for any smooth function $f \in L^{p}\left(\mathbb{T}^{3}\right)$.
\end{lemma}

The following  lemmas are the main tools to control the $L^2$-estimates of the perturbations.
\begin{lemma}[\cite{cl21}, Lemma 2.4; see also \cite{bv19b},  Lemma 3.7]   \label{Decorrelation1}
	Let $\sigma\in \mathbb{N}$ and $f,g:\mathbb{T}^3\rightarrow \R$ be smooth functions. Then for every $p\in[1,+\9]$,
	\begin{equation}\label{lpdecor}
		\big|\|fg(\sigma\cdot)\|_{L^p(\T^3)}-\|f\|_{L^p(\T^3)}\|g\|_{L^p(\T^3)} \big|\lesssim \sigma^{-\frac{1}{p}}\|f\|_{C^1(\T^3)}\|g\|_{L^p(\T^3)}.
	\end{equation}
\end{lemma}

\begin{lemma}[\cite{bms21}, Proposition 2]\label{lem-mean}
	Let $a \in C^{\infty}\left(\mathbb{T}^3 ; \mathbb{R}\right), v \in C_0^{\infty}\left(\mathbb{T}^3 ; \mathbb{R}\right)$. Then, for any $r \in[1, \infty]$ and $\lambda\in \mathbb{N}$,
	$$
	\left|\int_{\mathbb{T}^3} a v(\lambda x)\d x\right| \leq \lambda^{-1} C_r\|\nabla a\|_{L^r_x}\|v\|_{L^{r'}_x}.
	$$
\end{lemma}

\medskip
\noindent{\bf Acknowledgment.} Z. Zeng is supported by Natural Science Foundation of Jiangsu Province 
 (No. SBK20240 43113). D.\ Zhang is grateful for the NSFC grants (No. 12271352, 12322108) 
and Shanghai Frontiers Science Center of Modern Analysis.


\begin{thebibliography}{99}

\bibitem{A09}
H. Aluie.
\newblock Hydrodynamic and magnetohydrodynamic turbulence: Invariants, cascades, and locality.
\newblock {Ph.D. thesis}, Johns Hopkins University, 2009.
	

\bibitem{bbv20}
R. Beekie, T. Buckmaster, and V. Vicol.
\newblock Weak solutions of ideal {MHD} which do not conserve magnetic
helicity.
\newblock {\em Ann. PDE}, 6(1):Paper No. 1, 40, 2020.


\bibitem{bcv21}
T. Buckmaster, M. Colombo, and V. Vicol.
\newblock Wild solutions of the Navier-Stokes equations whose singular sets in time have Hausdorff dimension strictly less than 1.
\newblock {\em J. Eur. Math. Soc.}, 24(9):3333-3378, 2022.

\bibitem{bdis15}
T. Buckmaster, C. De~Lellis, P. Isett, and L. Sz\'{e}kelyhidi, Jr.
\newblock Anomalous dissipation for {$1/5$}-{H}\"{o}lder {E}uler flows.
\newblock {\em Ann. of Math. (2)}, 182(1):127--172, 2015.



\bibitem{bdsv19}
T. Buckmaster, C. De~Lellis, L. Sz\'{e}kelyhidi, Jr., and
V. Vicol.
\newblock Onsager's conjecture for admissible weak solutions.
\newblock {\em Comm. Pure Appl. Math.}, 72(2):229--274, 2019.



\bibitem{bv19b}
T. Buckmaster and V. Vicol.
\newblock Nonuniqueness of weak solutions to the {N}avier-{S}tokes equation.
\newblock {\em Ann. of Math. (2)}, 189(1):101--144, 2019.


\bibitem{bv19r}
T. Buckmaster and V. Vicol.
\newblock Convex integration and phenomenologies in turbulence.
\newblock {\em EMS Surv. Math. Sci.}, 6(1-2):173--263, 2019.


\bibitem{bv21}
T. Buckmaster and V. Vicol.
\newblock Convex integration constructions in hydrodynamics.
\newblock {\em Bull. Amer. Math. Soc. (N.S.)}, 58(1):1--44, 2021.


\bibitem{bms21}
J. Burczak, S. Modena, and L. Sz\'{e}kelyhidi, Jr.
\newblock Non Uniqueness of Power-Law Flows.
\newblock {\em Commun. Math. Phys.}, 388:199-243, 2021.




\bibitem{CD15}
A. Cheskidov and M. Dai.
\newblock Regularity criteria for the 3D Navier-Stokes and MHD equations.
\newblock arXiv:1507.06611v6, 2015.

\bibitem{CK}
A. Cheskidov and L. Kavlie.
\newblock Degenerate pullback attractors for the 3D Navier-Stokes equations.
\newblock{J. Math. Fluid Mech.}, 17:411-421, 2015.

\bibitem{cl21}
A. Cheskidov and X. Luo.
\newblock Nonuniqueness of weak solutions for the transport equation at
critical space regularity.
\newblock {\em Ann. PDE}, 7(1):Paper No. 2, 45, 2021.

\bibitem{cl20.2}
A. Cheskidov and X. Luo.
\newblock Sharp nonuniqueness for the Navier-Stokes equations.
\newblock {\em Invent. math.}, 229(3):987-1054, 2022.

\bibitem{cl23}
A. Cheskidov and X. Luo.
\newblock $L^2$-critical nonuniqueness for the 2D Navier-Stokes equations.
\newblock  {\em Ann. PDE}, 9(13):Paper No. 13, 56, 2023.


\bibitem{czz24}
A. Cheskidov, Z. Zeng and D. Zhang.
\newblock Strong non-uniqueness of finite energy solutions of the 3D deterministic and stochastic Navier-Stokes equations.
\newblock arXiv:2407.17463, 2024.

\bibitem{CF}
P. Constantin and C. Foias.
\newblock {\em {N}avier-{S}tokes Equations.}
\newblock Chicago Lectures in Mathematics. University of Chicago Press, Chicago, IL, 1988.



\bibitem{dai21}
M. Dai. Non-uniqueness of Leray-Hopf weak solutions of the 3d Hall-MHD system. {\em SIAM J. Math.
Anal.}, 53(5):5979--6016, 2021.



\bibitem{d2001}
P.~A. Davidson.
\newblock {\em An Introduction to Magnetohydrodynamics}.
\newblock Cambridge Texts in Applied Mathematics. Cambridge University Press,
2001.

\bibitem{dls09}
C. De~Lellis and L. Sz\'{e}kelyhidi, Jr.
\newblock The {E}uler equations as a differential inclusion.
\newblock {\em Ann. of Math. (2)}, 170(3):1417--1436, 2009.

\bibitem{dls10}
C. De~Lellis and L. Sz\'{e}kelyhidi, Jr.
\newblock On admissibility criteria for weak solutions of the {E}uler equations.
\newblock {\em Arch. Rational Mech. Anal.}, 195, 225--260, 2010.



\bibitem{dls13}
C. De~Lellis and L. Sz\'{e}kelyhidi, Jr.
\newblock Dissipative continuous {E}uler flows.
\newblock {\em Invent. Math.}, 193(2):377--407, 2013.




\bibitem{FL18}
D. Faraco and S. Lindberg. Magnetic helicity and subsolutions in ideal MHD. arXiv:1801.04896,  2018.



\bibitem{fls21}
D. Faraco, S. Lindberg, and L. Sz\'{e}kelyhidi, Jr.
\newblock Bounded solutions of ideal {MHD} with compact support in space-time.
\newblock {\em Arch. Ration. Mech. Anal.}, 239(1):51--93, 2021.

\bibitem{fls22}
D. Faraco, S. Lindberg, and L. Sz\'{e}kelyhidi, Jr.
\newblock Rigorous results on conserved and dissipated quantities in ideal MHD turbulence. 
\newblock {\em Geophys. Astrophys. Fluid Dyn.},   116(4):237--260, 2022.

\bibitem{fls21.2}
D. Faraco, S. Lindberg, and L. Sz\'{e}kelyhidi, Jr.
\newblock Magnetic helicity, weak solutions and relaxation of ideal MHD.
\newblock {\em Comm. Pure Appl. Math.}, 77(4):2387-2412, 2024.




\bibitem{FK64}
H. Fujita and T. Kato. On the Navier–Stokes initial value problem. I. {\em Arch. Ration. Mech. Anal.}, 16, 269–315, 1964.



\bibitem{hopf1951}
E. Hopf.
\newblock \"{U}ber die {A}nfangswertaufgabe f\"{u}r die hydrodynamischen
{G}rundgleichungen.
\newblock {\em Math. Nachr.}, 4:213--231, 1951.

\bibitem{I18}
P. Isett.
\newblock A proof of {O}nsager's conjecture.
\newblock {\em Ann. of Math. (2)}, 188(3):871--963, 2018.


\bibitem{KL07}
E. Kang and J. Lee.
\newblock Remarks on the magnetic helicity and energy conservation for ideal
magneto-hydrodynamics.
\newblock {\em Nonlinearity}, 20(11):2681--2689, 2007.

\bibitem{kato84}
T. Kato. Strong Lp-solutions of the Navier–Stokes equation in R
m, with applications to weak
solutions. {\em Math. Z.}, 187:471–480, 1984.

\bibitem{KT02}
H. Koch and D. Tataru. Well-posedness for the Navier–Stokes equations. {\em Adv. Math.}, 157:22–35, 2001.


\bibitem{leray1934}
J. Leray.
\newblock Sur le mouvement d'un liquide visqueux emplissant l'espace.
\newblock {\em Acta Math.}, 63(1):193--248, 1934.



\bibitem{lqzz22}
Y. Li, P. Qu, Z. Zeng, and D. Zhang.
\newblock Sharp non-uniqueness for the 3D hyperdissipative Navier-Stokes equations: beyond the Lions exponent.
\newblock {\em J. Math. Pures Appl.}, 190: 103602, 2024.



\bibitem{lzz21}
Y. Li, Z. Zeng, and D. Zhang.
\newblock Non-uniqueness of weak solutions to 3D magnetohydrodynamic equations.
\newblock {\em J. Math. Pures Appl.}, 165(9):232-285, 2022.


\bibitem{lzz21.2}
Y. Li, Z. Zeng, and D. Zhang.
\newblock  Sharp non-uniqueness of weak solutions to 3D magnetohydrodynamic equations.
\newblock {\em J. Funct.Anal.}, 287(7):110528, 2024.



\bibitem{L96}P.-L. Lions. Mathematical Topics in Fluid Mechanics. Vol. 1. Incompressible Models. Clarendon Press, 1996.



\bibitem{lt20}
T. Luo and E.~S. Titi.
\newblock Non-uniqueness of weak solutions to hyperviscous {N}avier-{S}tokes
equations: on sharpness of {J}.-{L}. {L}ions exponent.
\newblock {\em Calc. Var. Partial Differential Equations}, 59(3):Paper No. 92, 15, 2020.


\bibitem{luo19}
X. Luo.
\newblock Stationary solutions and nonuniqueness of weak solutions for the
{N}avier-{S}tokes equations in high dimensions.
\newblock {\em Arch. Ration. Mech. Anal.}, 233(2):701--747, 2019.




\bibitem{MY22}
C. Miao and K. Ye.
\newblock On the weak solutions for the MHD systems with controllable total energy and cross helicity.
\newblock {\em J. Math. Pures Appl.},  181(9): 190--227, 2024.

\bibitem{MS06}
H. Miura and O. Sawada. On the regularizing rate estimates of Koch-Tataru’s solution to the Navier-Stokes equations. {\em Asymptotic Analysis} 49, No 1-2, 1--15, 2006.

\bibitem{rrs16}
J. C. Robinson, J. L. Rodrigo, and W. Sadowski.
\newblock {\em The three-dimensional Navier-Stokes equations, Classical theory.}
\newblock { Cambridge Studies in Advanced Mathematics, 157.} Cambridge University Press, Cambridge, 2016.


\bibitem{ST87}
	H. Schmeisser and H. Triebel.
	\newblock {\em Topics in Fourier analysis and function spaces.}
	\newblock A Wiley-Interscience Publication. John Wiley \& Sons, Ltd., Chichester, 300 pp, 1987.

\bibitem{ST83}
M. Sermange and R. Temam.
Some mathematical questions related to the MHD equations.
\newblock {\em Comm. Pure Appl. Math.}, 36(5):635-664, 1983.

\bibitem{T}
R. Temam.
\newblock {\em {N}avier-{S}tokes Equations: Theory and Numerical Analysis.}
\newblock AMS Chelsea, Providence, Rhode Island, 2000.





\end{thebibliography}
\end{document}